\documentclass[a4paper]{amsart}
\usepackage[all,cmtip]{xy}
\usepackage{amsmath}
\usepackage{amssymb}
\usepackage{amsxtra}
\usepackage{color}
\usepackage{float}
\usepackage{pdflscape}
\usepackage{varioref}
\usepackage{colortbl}
\usepackage[bookmarksnumbered,colorlinks,plainpages]{hyperref}
\usepackage{cite}
\usepackage{stmaryrd}

\DeclareMathOperator{\rank}{rk}

\DeclareMathOperator{\lcm}{lcm}
\renewcommand{\mod}{\operatorname{mod}}
\DeclareMathOperator{\CM}{CM}
\newcommand{\rperp}[1]{#1^\perp}

\newcommand{\sCM}{\underline{\CM}}
\newcommand{\CMgr}[2]{\CM^{#1}\textrm{-}#2}
\newcommand{\sCMgr}[2]{\underline{\CM}^{#1}\textrm{-}#2}
\newcommand{\La}{\Lambda}

\newcommand{\mmod}[1]{\mathrm{mod}\textrm{-}#1}
\newcommand{\Der}[1]{\mathrm{D}^b(#1)}
\newcommand{\Knull}[1]{\mathrm{K}_0(#1)}
\newcommand{\Dsing}[2]{\mathrm{D}_{\mathrm{Sg}}^{b,#1}(#2)}
\newcommand{\bu}{\bullet}
\newcommand{\ra}{\rightarrow}

\newcommand{\al}{\alpha}
\newcommand{\be}{\beta}
\newcommand{\de}{\delta}
\newcommand{\De}{\Delta}
\newcommand{\ga}{\gamma}
\newcommand{\Ga}{\Gamma}
\newcommand{\si}{\sigma}
\newcommand{\ci}{\circ}

\newcommand{\bp}{\bar{p}}

\newcommand{\coh}[1]{\mathrm{coh}\textrm{-}#1}
\newcommand{\vect}[1]{\mathrm{vect}\textrm{-}#1}
\newcommand{\svect}[1]{\underline{\mathrm{vect}}\textrm{-}#1}
\newcommand{\vx}{{\vec{x}}}
\newcommand{\vz}{{\vec{z}}}
\newcommand{\vc}{{\vec{c}}}
\newcommand{\vAA}{{\vec{\AA}}}
\newcommand{\vom}{\vec{\omega}}
\newcommand{\modgr}[2]{\mathrm{mod}^{#1}\textrm{-}#2}
\newcommand{\Modgr}[2]{\mathrm{Mod}^{#1}\textrm{-}#2}
\newcommand{\projgr}[2]{\mathrm{proj}^{#1}\textrm{-}#2}

\newcommand{\Sing}[2]{\mathrm{D}^{b,#1}_{Sg}(#2)}
\newcommand{\XX}{\mathbb{X}}
\newcommand{\YY}{\mathbb{Y}}
\newcommand{\NN}{\mathbb{N}}
\newcommand{\PP}{\mathbb{P}}

\newcommand{\fund}{F}
\newcommand{\ZZ}{\mathbb{Z}}

\newcommand{\QQ}{\mathbb{Q}}
\newcommand{\Si}{\Sigma}

\newcommand{\Vv}{\mathcal{V}}
\def\SS{\mathbb{S}}
\def\union{\cup}

\newcommand{\Cc}{\mathcal{C}}
\newcommand{\sCc}{\underline{\mathcal{C}}}
\newcommand{\Oo}{\mathcal{O}}
\newcommand{\Tt}{\mathcal{T}}
\def\AA{\mathbb{A}}
\newcommand{\DD}{\mathbb{D}}
\newcommand{\EE}{\mathbb{E}}
\newcommand{\Hom}[3]{\mathrm{Hom}_{#1}(#2,#3)}
\newcommand{\Hhom}[3]{\mathcal{H}om_{#1}(#2,#3)}

\newcommand{\sHom}[3]{\mathrm{\underline{Hom}}_{#1}(#2,#3)}
\newcommand{\Ext}[4]{\mathrm{Ext}^{#1}_{#2}(#3,#4)}
\newcommand{\Tor}[4]{\mathrm{Tor}_{#1}^{#2}(#3,#4)}
\newcommand{\Hh}{\mathcal{H}}

\newcommand{\define}[1]{\emph{#1}}
\newcommand{\eulerchar}[1]{\chi_{#1}}
\newcommand{\Ll}{\mathcal{L}}
\newcommand{\lra}{\longrightarrow}

\newcommand{\up}[1]{\stackrel{#1}{\lra}}
\newcommand{\slope}[1]{\mu\,#1}
\DeclareMathOperator{\dual}{D}
\newcommand{\odual}[1]{#1\spcheck{}}
\DeclareMathOperator{\rflxdual}{d}

\newcommand{\dt}{\cdot}
\newcommand{\sq}{\rule[.2ex]{1ex}{1ex}}
\newcommand{\sz}[1]{\langle#1\rangle}
\newcommand{\can}{\mathrm{can}}
\newcommand{\cub}{\mathrm{cub}}
\newcommand{\str}[1]{{#1}^\ast}
\newcommand{\sclass}[1]{{\llbracket #1 \rrbracket}}

\newtheorem*{Thm-A}{Theorem A}
\newtheorem*{Thm-B}{Theorem B}
\newtheorem*{Thm-C}{Theorem C}
\newtheorem*{Lemma-B}{Lemma B}
\newtheorem{proposition}{Proposition}[section]
\newtheorem{theorem}[proposition]{Theorem}
\newtheorem{corollary}[proposition]{Corollary}

\newtheorem{lemma}[proposition]{Lemma}

\theoremstyle{definition}
\newtheorem{definition}[proposition]{Definition}
\newtheorem{remark}[proposition]{Remark}

\numberwithin{equation}{section}
\newcommand{\LL}{\mathbb{L}} \newcommand{\set}[1]{\{#1\}}

 \newcommand{\vy}{\vec{y}}
\newcommand{\bx}{\bar{x}}

\newcommand{\incl}{\hookrightarrow}
\newcommand{\End}[2]{\mathrm{End}_{#1}(#2)}
\newcommand{\sEnd}[2]{\mathrm{\underline{End}}_{#1}(#2)}
\newcommand{\Pic}[1]{\mathrm{Pic}(#1)}
\newcommand{\Aut}[1]{\mathrm{Aut}(#1)}
 \newcommand{\iso}{\cong}

\newcommand{\extb}[2]{E_{#1}\langle #2\rangle}

\newcommand{\injh}[1]{\Im(#1)}
\newcommand{\projh}[1]{\mathfrak{P}(#1)}
\newcommand{\simp}[1]{\mathcal{S}_{#1}}
\newcommand{\sdir}{\rtimes}
\newcommand{\ind}[1]{\mathrm{ind}(#1)}
\newcommand{\genab}[1]{\langle#1\rangle_{\mathrm{ab}}}
\newcommand{\gentri}[1]{\langle#1\rangle_{\mathrm{tri}}}

\newcommand{\euler}[2]{\langle #1,#2\rangle}

\newcommand{\bv}{\bar{v}}

\newcommand{\coxpol}[1]{\chi_{#1}}
\newcommand{\CYdim}[1]{\mathrm{CY{\textrm{-}}dim}(#1)}
\newcommand{\vdom}{\vec{\,\delta\,}}

\begin{document}

\title[Triangle singularities]{Triangle singularities, ADE-chains, and
  weighted projective lines}

\author[D. Kussin]{Dirk Kussin}
\address{Dipartimento di Informatica -- Settore di Matematica\\
Universit\`a degli Studi di Verona\\
Strada le Grazie 15 -- Ca' Vignal 2\\
37134 Verona\\
Italy}
\email{dirk@math.uni-paderborn.de}

\author[H. Lenzing]{Helmut Lenzing}
\address{Institut f\"{u}r Mathematik\\
Universit\"{a}t Paderborn\\
33095 Paderborn\\
Germany}
\email{helmut@math.uni-paderborn.de}

\author[H. Meltzer]{Hagen Meltzer}
\address{Instytut Matematyki\\
Uniwersytet Szczeci\'nski\\
70451 Szczecin\\
Poland}
\email{meltzer@wmf.univ.szczecin.pl}

\date{\textbf{March 21, 2012}}

\begin{abstract}
We investigate the triangle singularity $f=x^a+y^b+z^c$, or $S=k[x,y,z]/(f)$,  attached to the weighted projective line $\XX$ given by a weight triple $(a,b,c)$. We investigate the stable category $\svect\XX$ of vector bundles on $\XX$ obtained from the vector bundles by factoring out all the line bundles. This category is triangulated and has Serre duality. It is, moreover, naturally equivalent to the stable category of graded maximal Cohen-Macaulay modules over $S$ (or matrix factorizations of $f$), and then by results of Buchweitz and Orlov to the singularity category of $f$.

We show that $\svect\XX$ is fractional Calabi-Yau whose CY-dimension is a function of the Euler characteristic of $\XX$. We show the existence of a tilting object which has the shape of an $(a-1)\times(b-1)\times(c-1)$-cuboid. Particular attention is given to the weight types $(2,a,b)$ yielding an explanation of Happel-Seidel symmetry for a class of important Nakayama algebras. In particular, the weight sequence $(2,3,p)$ corresponds to an ADE-chain, the $\EE_n$-chain, extrapolating the exceptional Dynkin cases $\EE_6,\EE_7$ and $\EE_8$ to a whole sequence of triangulated categories.
\end{abstract}

\maketitle

\section{Introduction and main results} \label{sect:intro}
Let $\XX=\XX(a,b,c)$ denote the weighted projective line of weight type $(a,b,c)$, where the integers $a,b,c$ are at least $2$, and where the base field $k$ is supposed to be algebraically closed. Following~\cite{Geigle:Lenzing:1987}, the category $\coh\XX$ of coherent sheaves on $\XX$ is obtained by applying Serre's construction~\cite{Serre:fac} to the (suitably graded) \define{triangle singularity} $x_1^a+x_2^b+x_3^c$. We recall that the Picard group of $\XX$ is naturally isomorphic to the rank one abelian group $\LL=\LL(a,b,c)$ on three generators $\vx_1,\vx_2,\vx_3$ subject to the relations $a\vx_1=b\vx_2=c\vx_3=:\vc$, where $\vc$ is called the \define{canonical element}.  Up to isomorphism the line bundles are then given by the system $\Ll$ of twisted structure sheaves $\Oo(\vx)$ with $\vx\in\LL$.

By $\vect\XX$ we denote the full subcategory of the category $\coh\XX$ of coherent sheaves formed by all vector bundles, that is, all locally free sheaves.  We recall that a \define{Frobenius category} is an exact category (Quillen's sense) which has enough projectives and injectives, and where the projective and the injective objects coincide.

The complexity of $\coh\XX$ (resp.\ $\vect\XX$) with respect to a classification of indecomposable objects is largely determined by the \define{Euler characteristic} given for weight type $(p_1,\ldots,p_t)$ by the expression $ \eulerchar\XX=2-\sum_{i=1}^t (1-\frac{1}{p_i}) $.

\begin{Thm-A}
Let $\XX$ have weight type $(a,b,c)$, $a,b,c\geq2$. Then the following holds.
\begin{enumerate}
\item[(1)] The category $\vect\XX$ is a Frobenius category with the system $\Ll$ of all line bundles as the indecomposable projective-injective objects.
\item[(2)] The stable category $\svect\XX=\vect\XX/[\Ll]$ is triangulated and has a tilting object.
\item[(3)] The category $\svect\XX$ is fractional Calabi-Yau of CY-dimension $1-2\eulerchar\XX$ (after cancelation).
\end{enumerate}
\end{Thm-A}
The exact values of the CY-dimension are given in Proposition~\ref{prop:CY:2ab} and Proposition~\ref{prop:CY}.

A decisive role in the proof of Theorem~A is taken by our next result. We recall that an object $E$ in an abelian or a triangulated category is called \define{exceptional} if its endomorphism ring is a skew field (hence equal to $k$ in our case) and, moreover, all self-extensions of $E$, of non-zero degree, vanish. For the next result we need two further distinguished elements of $\LL$, the \define{dualizing element} $\vom=\vc-(\vx_1+\vx_2+\vx_3)$ and the \define{dominant element} $\vdom=2\vom +\vc=(a-2)\vx_1+(b-2)\vx_2+(c-2)\vx_3$. For each $0\leq\vx\leq\vdom$ we define the \define{extension bundle} $\extb{}{\vx}$ as the central term of the 'unique' non-split exact sequence $0\ra \Oo(\vom)\ra\extb{}{\vx}\ra \Oo(\vx)\ra 0$.

\begin{Thm-B}
Let $\XX$ be a weighted projective line with weight triple $(a,b,c)$, $a,b,c\geq2$. Then the following properties hold.

(i) Each indecomposable rank-two bundle $F$ is exceptional in $\coh\XX$ and in $\svect\XX$. Moreover, up to line bundle twist (grading shift),  $F$ is isomorphic to an \define{extension bundle}.

(ii) The extension bundles $\extb{}{\vx}$, $0\leq \vx \leq \vdom$, form a tilting object, the \define{tilting cuboid}, $T_\cub$ of $\svect\XX$ whose endomorphism ring
$$
\Si_{(a,b,c)}=k\vAA_{a-1}\otimes k\vAA_{b-1}\otimes k\vAA_{c-1}
$$
is isomorphic to the incidence algebra of the poset $[1,a-1]\times [1,b-1]\times [1,c-1]$ of, possibly degenerated, cubical shape.
\end{Thm-B}
Here, $k\vAA_n$ denotes the path algebra of the equioriented $\AA_n$-quiver, and $[1,n]$ denotes the linearly ordered set $\set{1,2,\ldots,n}$. Then $\Si_{(a,b,c)}$ is given as the incidence algebra of the poset $[1,a-1]\times[1,b-1]\times[1,c-1]$, that is the algebra given by the fully commutative quiver below \small
\begin{figure}[H]\label{fig:cuboid:346}
$$
\def\b{\circ}
\def\c{\circ}
\xymatrix@C=1pt@R=1pt{
&&&\b\ar[rrrr]&&&&\b\ar[rrrr]&&&&\b\ar[rrrr]&&&&\b\ar[rrrr]&&&&\b\\
&&&&&&&&&&&&&&&&&&&\\
\b\ar[rrruu]\ar[rrrr]&&&&\b\ar[rrruu]\ar[rrrr]&&&&\b\ar[rrruu]\ar[rrrr]&&&&\b\ar[rrruu]\ar[rrrr]&&&&\b\ar[rrruu]\\
&&&&&&&&&&&&&&&&&&&\\
&&&\c\ar'[r][rrrr]\ar'[uu][uuuu]&&&&\c\ar'[r][rrrr]\ar'[uu][uuuu]&&&&\c\ar'[r][rrrr]\ar'[uu][uuuu]&&&&\c\ar'[r][rrrr]\ar'[uu][uuuu]&&&&\c\ar[uuuu]\\
&&&&&&&&&&&&&&&&&&&\\
\b\ar[rrruu]\ar[uuuu]\ar[rrrr]&&&&\b\ar[rrruu]\ar[uuuu]\ar[rrrr]&&&&\b\ar[rrruu]\ar[uuuu]\ar[rrrr]&&&&\b\ar[rrruu]\ar[uuuu]\ar[rrrr]&&&&\b\ar[rrruu]\ar[uuuu]\\
&&&&&&&&&&&&&&&&&&&\\
&&&\c\ar'[r][rrrr]\ar'[uu][uuuu]&&&&\c\ar'[r][rrrr]\ar'[uu][uuuu]&&&&\c\ar'[r][rrrr]\ar'[uu][uuuu]&&&&\c\ar'[r][rrrr]\ar'[uu][uuuu]&&&&\c\ar[uuuu]\\
&&&&&&&&&&&&&&&&&&&\\
\b\ar[rrruu]\ar[uuuu]\ar[rrrr]&&&&\b\ar[rrruu]\ar[uuuu]\ar[rrrr]&&&&\b\ar[rrruu]\ar[uuuu]\ar[rrrr]&&&&\b\ar[rrruu]\ar[uuuu]\ar[rrrr]&&&&\b\ar[rrruu]\ar[uuuu]\\
}
$$
\caption{The tilting $2\times 3\times 5$ cuboid for type $(a,b,c)=(3,4,6)$}
\end{figure}
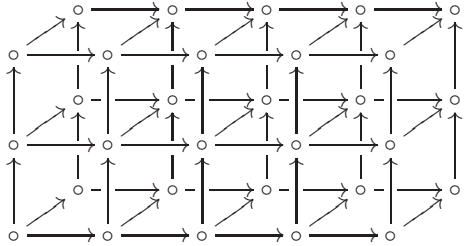\normalsize
We note that for type $(2,a,b)$, $a,b\geq3$, the cuboid degenerates to the $(a-1)\times(b-1)$-rectangle, while for type $(2,2,n)$, $n\geq2$, the cuboid degenerates to a line segment of length $n-1$, and for type $(2,2,2)$ it degenerates to a point. It is well-known that rectangular diagrams of this shape play a prominent role in singularity theory for the description of the Milnor lattice of a singularity, see for instance \cite{Gabrielov:Dynkin,Ebeling:book}.
We further note that the treatment of this paper is to some extent parallel to the treatment of Kleinian and Fuchsian singularities by \cite{KST-1}, \cite{KST-2} and \cite{Lenzing:Pena:2011}. Formally the intersection with the present paper only concerns the Kleinian weight type $(2,3,5)$ and the Fuchsian weight type $(2,3,7)$. However, the shape of results is rather different because of two reasons: (a) In the Kleinian resp.\ Fuchsian case the categories are mostly not fractionally Calabi-Yau, (b) In the Kleinian or Fuchsian case, only the Gorenstein numbers $+1$ and $-1$ appear, whereas for triangle singularities the (absolute value) of the Gorenstein number may be arbitrarily large.

\sloppy
Algebraically, an established method to investigate the complexity of a singularity is due to Buchweitz~\cite{Buchweitz:1986}, later revived by Orlov~\cite{Orlov:2009} who primarily deals with the graded situation. Given the $\LL$-graded triangle singularity $S=k[x_1,x_2,x_3]/(x_1^a+x_2^b+x_3^c)$ this amounts to consider the Frobenius category $\CMgr\LL{S}$ of $\LL$-graded maximal Cohen-Macaulay modules, its associated stable category $\sCMgr\LL{S}$ and the \define{singularity category} of $S$ defined as the quotient $\Dsing\LL{S}=\Der{\modgr\LL{S}}/\Der{\projgr\LL{S}}$. It is shown in \cite{Buchweitz:1986,Orlov:2009} that the two constructions yield naturally equivalent triangulated categories $\sCMgr\LL{S}=\Dsing\LL{S}$. It further follows from \cite{Geigle:Lenzing:1987} that sheafification yields a natural equivalence $q:\CMgr\LL{S}\up\sim\vect\XX$ with the indecomposable projective $\LL$-graded $S$-module $S(\vx)$ corresponding to the line bundle $\Oo(\vx)$ on $\XX$, and further $q$ induces a natural identification
\begin{equation}\label{eq:three:cats}
\Dsing\LL{S}=\sCMgr\LL{S}=\svect\XX, \textrm{ where }\XX=\XX(a,b,c).
\end{equation}
In particular, comparing the sizes of the triangulated categories $\Der{\coh\XX}$ and $\svect\XX$ by the ranks of their Grothendieck groups yields the \define{Gorenstein number}
\begin{equation}\label{eq:G-number}
a(\XX)=\rank (\Knull{\coh\XX})-\rank (\Knull{\svect\XX})=(a+b+c-1)-(a-1)(b-1)(c-1),
\end{equation}
of the coordinate algebra $S$ of $\XX$, a formula nicely illustrating the effects of an $\LL$-graded version of Orlov's theorem~\cite{Orlov:2009}, see Appendix~\ref{appx:Orlov}.

Weighted projective lines, and their defining equations of shape $x^a\pm y^b\pm z^z=0$ have a long history going back to Klein~\cite{Klein:1884} and Poincar\'{e}~\cite{Poincare:1882}. Accordingly their study has a high contact surface with many mathematical subjects, classical and modern. Among the many related subjects we mention representation theory of algebras~\cite{Ringel:1984}, invariant theory~\cite{Klein:1884}, function theory~\cite{Nasatyr:Steen:1995},\cite{Nasatyr:Steer:1995a},~\cite{Polishchuk:2006}, automorphic forms~\cite{Poincare:1882,Milnor:1975,Dolgachev:1975,Neumann:1977,Wagreich:1980,Lenzing:1994}, orbifolds~\cite[Chapter 13]{Thurston:1980},\cite{Montesinos:1987,Nasatyr:Steen:1995,Polishchuk:2006} 3-manifolds~\cite{Milnor:1975},\cite{Thurston:1980}, and singularities~\cite{Lenzing:1998,Ebeling:Takahashi:2010,Lenzing:Pena:2011,KST-1,KST-2,Lenzing:2011}. This latter aspect, in particular, includes the aspects of strange duality~\cite{Ebeling:Wall:1985,Ebeling:Takahashi:2010} and mirror symmetry~\cite{Ueda:2006,Ebeling:Takahashi:2011}.

In the center of the present paper is the singularity category attached to a triangle singularity which is studied through the structure of the corresponding stable categories of vector bundles. This analysis is in the spirit of Buchweitz~\cite{Buchweitz:1986} and Orlov~\cite{Orlov:2009} and is related to, but different from, recent work by Kajiura, Saito, and Takahashi~\cite{KST-1}, \cite{KST-2} and work by Lenzing and la Pe\~{n}a~\cite{Lenzing:Pena:2011} on Kleinian and Fuchsian singularities, see Section~\ref{ssect:triangle:fuchsian}.

We have incorporated three major applications of our treatment of triangle categories incorporated in this paper:
\begin{itemize}
\item[(a)] A factorization property for the domestic or tubular case (Theorem~\ref{thm:factorization_domestic} and Theorem~\ref{thm:factorization_tubular}) showing that many morphisms in the category of vector bundles factor through a direct sum of line bundles;
\item[(b)] A surprisingly simple explanation of Happel-Seidel symmetry  by Theorem~\ref{thm:Happel:Seidel:symmetry}, a phenomenon first observed by Happel-Seidel in \cite{Happel:Seidel:2010}. This symmetry concerns the class of Nakayama algebras $A_n(r)$ given by the path algebra of the equioriented quiver $\AA_n$, subject to nilpotency of degree $r$ for each sequence of $r$ consecutive arrows. In our treatment we show that $A_n(r)$ with $n=(a-1)(b-1)$ and $r=a$ or $r=b$ are realized as endomorphism rings of tilting objects on $\svect\XX$ for weight type $(2,a,b)$. We further exhibit---on a still conjectural level---new spontaneous series of such algebras not covered by Happel-Seidel symmetry, see Section~\ref{ssect:Happel:Seidel}.
\item[(c)] The development of a framework for the ADE-chain problem, see~\cite{Ringel:2008}.
\end{itemize}
Another, quite surprising, application concerns the invariant subspace problem for nilpotent operators, a subject intensively studied by Ringel and Schmidmeier~\cite{Ringel:Schmidmeier:2008}; we also refer to related work by Simson~\cite{Simson:2007} and Pu Zhang~\cite{Zhang:Pu:2011}. In \cite{Kussin:Lenzing:Meltzer:2010apre} we relate this problem to the stable category of vector bundles $\svect\XX$ for weight type $(2,3,p)$, where $p$ takes the role of the nilpotency degree. In a more general analysis, we relate the study of two invariant flags of subspaces of flag-lengthes $(a-1)$, $(b-1)$ for nilpotency degree $c$ to the study of $\svect\XX$ for weight type $(a,b,c)$ in~\cite{Kussin:Lenzing:Meltzer:2011pre}. The latter investigation presupposes the major results on extension bundles from Section~\ref{sect:rank_two}.

The structure of the paper is as follows. In Section~\ref{sect:basics} we recall basic properties of a weighted projective line $\XX$, given by a weight triple. In Section~\ref{sect:frobenius} we discuss a Frobenius structure on the category $\vect\XX$ of vector bundles whose indecomposable projective-injective objects are just the line bundles. From the technical point of view, Section~\ref{sect:rank_two} is the heart of the paper; we prove that indecomposable vector bundles of rank two are exceptional in $\coh\XX$ and $\svect\XX$. This treatment exploits the detailed knowledge of the Picard group of $\XX$, viewed as an ordered group; it is fundamental for the rest of the paper. In Section~\ref{sect:impact:euler} we discuss the influence of the Euler characteristic on the shape of Auslander-Reiten components and categorical properties for $\svect\XX$. In Section~\ref{sect:tilting} we establish a tilting object of the category $\svect\XX$ for arbitrary weight triples, called the tilting cuboid, together with a special range of further tilting objects adapted to weight type $(2,a,b)$. Here, we investigate a class of naturally arising Nakayama-algebras for which we prove in Section~\ref{ssect:Happel:Seidel} Happel-Seidel symmetry, compare~\cite{Happel:Seidel:2010}. Section~\ref{sect:CY} deals with the Calabi-Yau property of the stable categories of vector bundles, a very distinctive property of these categories. Section~\ref{sect:ADE} is to some extend related; it deals with the formation of so-called ADE-chains, see~\cite{Ringel:2008} which, very roughly, extrapolate the properties of the exceptional Dynkin cases $\EE_6$, $\EE_7$, $\EE_8$ to an infinite sequence of triangulated categories. Our treatment concludes with Section~\ref{sect:comments:problems} treating with comments and open problems. Finally, we have included three appendices. Section~\ref{appx:comparison:scales} yields with Euler characteristic zero, and yields an analysis of the triangle-equivalence between $\svect\XX$ and $\Der{\coh\XX}$, including slope scales. In Section~\ref{appx:2-fold suspension} we show that two-fold suspension for $\svect\XX$ is induced by the Picard shift with the canonical element. Finally, in  Section~\ref{appx:Orlov} we state an $\LL$-graded version of Orlov's theorem~\cite{Orlov:2009}, adapted to the present paper. Our main references for weighted projective lines are \cite{Geigle:Lenzing:1987}, \cite{Lenzing:Pena:1997} and \cite{Meltzer:2004}. For the categorical properties of $\coh\XX$ we recommend \cite{Chen:Krause:2011} or \cite{Lenzing:2007}.

\section{Definitions and basic properties} \label{sect:basics}

We recall some basic notions and facts about weighted projective lines. We restrict our treatment to the case of three weights.  So let $p_1,\,p_2,\,p_3 \geq 2$ integers, called weights. Denote by $S$ the commutative algebra $$S=\frac{k[X_1,X_2,X_3]}{\bigl(X_1^{p_1}+X_2^{p_2}+X_3^{p_3}\bigr)} =k[x_1,x_2,x_3].$$ Let $\LL=\LL(p_1,p_2,p_3)$ be the abelian group given by generators $\vx_1,\,\vx_2,\,\vx_3$ and defining relations $p_1 \vx_1=p_2 \vx_2=p_3 \vx_3=:\vc$. The $\LL$-graded algebra $S$ is the appropriate object to study the triangle singularity $x_1^{p_1}+x_2^{p_2}+x_3^{p_3}$. The element $\vc$ is called the \define{canonical element}. Each element $\vx\in\LL$ can be written in \define{normal form}
\begin{equation}
  \label{eq:normalform}
  \vx=n_1 \vx_1+ n_2 \vx_2+n_3 \vx_3+m\vc
\end{equation}
with unique $n_i$, $m\in\mathbb{Z}$, $0\leq n_i <p_i$.

The algebra $S$ is $\LL$-graded by setting $\deg x_i=\vx_i$ ($i=1,2,3$), hence $S=\bigoplus_{\vx\in\LL}S_{\vx}$. By an $\LL$-graded version of the Serre construction~\cite{Serre:fac}, the weighted projective line $\XX=\XX(p_1,p_2,p_3)$ of weight type $(p_1,p_2,p_3)$ is given by its category of coherent sheaves $\coh\XX=\mod^{\LL}(S)/\mod_0^{\LL}(S)$, the quotient category of finitely generated $\LL$-graded modules modulo the Serre subcategory of graded modules of finite length.  The abelian group $\LL$ is ordered by defining the positive cone $\{\vx\in\LL\mid\vx\geq\vec{0}\}$ to consist of the elements of the form $n_1 \vx_1 +n_2 \vx_2 +n_3 \vx_3$, where $n_1,\,n_2,\,n_3 \geq 0$.  Then $\vx\geq\vec{0}$ if and only if the homogeneous component $S_{\vx}$ is non-zero, and equivalently, if in the normal form~\eqref{eq:normalform} of $\vx$ we have $m\geq 0$.

The image $\Oo$ of $S$ in $\mod^{\LL}(S)/\mod^{\LL}_0 (S)$ serves as the structure sheaf of $\coh\XX$, and $\LL$ acts on the above data, in particular on $\coh\XX$, by grading shift. Each line bundle has the form $\Oo(\vx)$ for a uniquely determined $\vx$ in $\LL$, and we have natural isomorphisms $$\Hom{}{\Oo(\vx)}{\Oo(\vy)}=S_{\vy-\vx}.$$ Defining the \define{dualizing element} from $\LL$ as $\vom=\vc-(\vx_1+\vx_2+\vx_3)$, the category $\coh\XX$ satisfies Serre duality in the form $$\dual\Ext{1}{}{X}{Y}=\Hom{}{Y}{X(\vom)}$$ functorially in $X$ and $Y$. Moreover, Serre duality implies the existence of almost split sequences for $\coh\XX$ with the Auslander-Reiten translation $\tau$ given by the shift with $\vom$.

Concerning the ordering on $\LL$, with the canonical and the dualizing
element for $\vx\in\LL$ we have
\begin{equation} \label{order:of_L}
\vx\not\geq 0\ \ \Leftrightarrow\
\vx\leq\vc+\vom,
\end{equation}
expressing that $\LL$ is almost linearly ordered.

For each line bundle $L$, the extension term of the almost split
sequence $$0\lra L(\vom)\lra E(L)\lra L\lra 0$$ is called the
\define{Auslander bundle} corresponding to $L$.  More generally, we
consider \define{extension bundles} $\extb{L}{x}$ defined as the
extension terms of the 'unique' non-split exact sequence
$$0\ra L(\vom)\lra \extb{L}{\vx}\ra L(\vx)\ra 0,$$
where $0\leq \vx\leq \vdom=2\vom+\vc$, that is,
$\vx=\sum_{i=1}^3\ell_i\vx_i$ with $0\leq\ell_i\leq (p_i-2)$.  Dealing
with three weights $p_i \geq 2$ ($i=1,2,3$) each extension bundle
$E=\extb{L}{\vx}$, in particular each Auslander bundle, is exceptional
by Theorem~\ref{thm:extension_bundle} in $\coh\XX$, that is, satisfies
$\End{}{E}=k$ and $\Ext{1}{}{E}{E}=0$.

The category $\vect\XX$ carries the structure of a Frobenius
category such that the system $\Ll$ of all line bundles is the system
of all indecomposable projective-injectives: A sequence $\eta\colon 0\ra
E'\ra E\ra E''\ra 0$ in $\vect\XX$ is \define{distinguished exact} if
all the sequences $\Hom{}{L}{\eta}$ with $L$ a line bundle are exact
(equivalently all the sequences $\Hom{}{\eta}{L}$ are exact).

Accordingly, see~\cite{Happel:1988}, the stable
category $$\svect\XX=\vect\XX/[\Ll]$$ is triangulated. Moreover, it
follows that the triangulated category $\svect\XX$ is Krull-Schmidt
with Serre duality induced from the Serre duality of $\coh\XX$. The
triangulated category $\svect\XX$ is homologically finite. Notice that
all line bundles disappear when passing from $\vect\XX$ to
$\svect\XX$. The Auslander bundles and, more generally, the extension
bundles will serve as some type of replacement, allowing us sufficient
grip on the properties of $\svect\XX$.  In particular, we will see
later that $\svect\XX$ always has a tilting object, see
Theorem~\ref{thm:std_tilting_obj} consisting of extension bundles.

It is shown in~\cite{Geigle:Lenzing:1987} that the quotient functor
$q\colon\mod^{\LL}(S)\ra\coh\XX$ induces an equivalence
$\CM^{\LL}(S)\up{\sim}\vect\XX$, where $\CM^{\LL}(S)$ denotes the
category of $\LL$-graded (maximal) Cohen-Macaulay modules over $S$.
Under this equivalence indecomposable graded projective modules over
$S$ correspond to line bundles in $\vect\XX$, resulting in a natural
equivalence
$$\sCM^{\LL}(S)\simeq\svect\XX$$
used from now on as an identification.
Stable categories of (graded) Cohen-Macaulay modules play an important
role in the analysis of singularities,
see~\cite{Buchweitz:1986,KST-1,KST-2,Orlov:2009}.

There are two important $\ZZ$-linear forms, rank $\rank{}$ and degree $\deg{}$, on the Grothendieck group $\Knull\XX$. The rank vanishes on (classes of) finite length sheaves; it is strictly positive on non-zero vector bundles, and constant on Auslander-Reiten orbits. The degree is zero on the structure sheaf $\Oo$, and is strictly positive on simple sheaves. In more detail, putting $\bp=\lcm(p_1,p_2,p_3)$, we have $\deg{\simp{}}=\bp/p_i$ if $\simp{}$ is a  simple sheaf concentrated in the $i$-th exceptional point $x_i$, $i=1,2,3$, and $\deg{\simp{}}=\bp$ if $\simp{}$ is a simple sheaf concentrated in an ordinary point.

If $X$ is a non-zero sheaf then $\rank{X}$ or $\deg{X}$ is non-zero, making the \define{slope} $\slope{X}=\deg{X}/\rank{X}$ a well defined member of $\bar\QQ=\QQ\union\set{\infty}$.

In this paper we also need the \define{determinant} which is a useful refinement of the degree. We recall from~\cite{Geigle:Lenzing:1987,Lenzing:Meltzer:1993} that the determinant is the unique homomorphism $\det:\Knull{\coh\XX}\ra \LL$ such that $\det(\Oo(\vx))=\vx$ for each $\vx$ from $\LL$ and, moreover, $\det(\simp{})>0$ for each simple sheaf. In
more detail, if $\simp{}$ is concentrated at the $i$-th exceptional point $x_i$, then $\det(\simp{})=\vx_i$, and if $\simp{}$ is concentrated in an ordinary point then $\det(\simp{})=\vc$. If $\de:\LL\ra\ZZ$ is the \define{degree map} which is the group homomorphism given on the generators $\vx_i$, $i=1,2,3$, of $\LL$ by $\de(\vx_i)=\bp/p_i$, then we have $\delta(\det(X))=\deg(X)$ for each $X$ in $\coh\XX$.

\section{The Frobenius category of vector bundles} \label{sect:frobenius}
Recall that a \define{Frobenius category} over $k$ is a $k$-linear category equipped with an exact structure in the sense of Quillen~\cite{Quillen:1973}, see also \cite[Appendix~A]{Keller:exact}, such that there exist sufficiently many (relative) injective resp.\ projective objects and where, moreover, the projectives coincide with the injectives.
Assume now that we deal with a weight triple $(p_1,p_2,p_3)=(a,b,c)$ with $p_i\geq2$. Then the $\LL$-graded coordinate algebra $S=k[x_1,x_2,x_3]/(x_1^a+x_2^b +x_3^c)$ is a (graded) hypersurface singularity, in particular (graded) Gorenstein. It follows that the category $\CMgr\LL{S}$ of maximal $\LL$-graded Cohen-Macaulay modules inherits from the category $\modgr\LL{S}$ of all finitely generated $\LL$-graded $S$-modules the exact structure of a Frobenius category. By means of sheafification we obtain an equivalence $\phi:\CMgr\LL{S}\ra\vect\XX$, see~\cite{Geigle:Lenzing:1987}, mapping the indecomposable projective $S(\vx)$ to the twisted structure sheaf $\Oo(\vx)$ for any $\vx\in\LL$. By transport of structure this turns $\vect\XX$ into a Frobenius category, whose exact structure is given by the following class of distinguished exact sequences.
\begin{definition}
A sequence $0\ra X' \up{u} X \up{v} X''\ra  0$ in $\vect\XX$ is called \define{distinguished exact} if for each line bundle $L$ the induced sequence $0\ra \Hom{}{L}{X'}\ra \Hom{}{L}{X}\ra\Hom{}{L}{X''}\ra 0$ is exact. In this situation $u$ (resp.\ $v$) will be called a \define{distinguished monomorphism} resp.\ a \define{distinguished epimorphism}.
\end{definition}
It follows immediately that each almost-split sequence $0\ra X \ra Y \ra Z \ra 0$ of vector bundles is distinguished exact, provided $Z$, equivalently, $X=Z(\vom)$ is not a line bundle.
\begin{proposition}
The distinguished exact sequences define an exact structure on $\vect\XX$ which is Frobenius. Moreover, the indecomposable projectives (resp.\ injectives) are exactly the line bundles.~\qed
\end{proposition}
For the proof of the next result we need that the graded global section functor
$$
\Ga: \coh\XX \lra \Modgr\LL{S},\quad X\mapsto \bigoplus_{\vx\in \LL}\Hom{}{\Oo(\vx)}{X}
$$
is a full embedding. We also note that  \define{vector bundle duality} $\odual{}\colon\vect\XX \ra\vect\XX$, $X\mapsto \Hhom{}{X}{\Oo}$, sends line bundles to line bundles, and preserves distinguished exact sequences.
\begin{proposition}
Let $\eta:0\ra X \up{u} Y \up{v} Z \ra 0$ be a (not necessarily exact) sequence in $\vect\XX$. Then the following assertions are equivalent:

\begin{itemize}
\item[(i)]  $\eta$ is distinguished exact.
\item[(ii)] $\Hom{}{L}{\eta}$ is exact for each line bundle $L$.
\item[(iii)]$\Hom{}{\eta}{L}$ is exact for each line bundle $L$.
\end{itemize}
Moreover, if these conditions hold then $\eta$ is exact in $\coh\XX$.
\end{proposition}
\begin{proof}
The equivalence $(i)\iff (ii)$ is true by definition, and further $(ii)\iff (iii)$ follows by vector bundle duality. Hence (i), (ii) and (iii) are equivalent. Moreover, (i) implies that the sequence $\Ga(\eta)$ is exact in $\modgr\LL{S}$. By sheafification then the exactness of $\eta$ follows.
\end{proof}

For applications in \cite{Kussin:Lenzing:Meltzer:2010apre,Kussin:Lenzing:Meltzer:2011pre} we need a  more specific result.

\begin{corollary}
Assume that $X$ and $Y$ are vector bundles and $\eta:0\ra X \up{u} Y \up{v} Z \ra 0$ is a sequence such that $X \up{u} Y \up{v} Z \ra 0$ is exact in $\coh\XX$ and such that $\Hom{}{\eta}{L}$ is exact for each line bundle $L$. Then also $Z$ is a vector bundle, and $\eta$ is distinguished exact.
\end{corollary}

\begin{proof}
Let $Z_0$ be the torsion subsheaf of $Z$ such that $\bar{Z}=Z/Z_0$ is a vector bundle. For each line bundle $L$ we have  $\Hom{}{Z_0}{L}=0$, hence $\Hom{}{\bar{Z}}{L}=\Hom{}{Z}{L}=0$. Putting $\bar\eta: 0 \ra X\up{u} Y \up{\bar{v}}Z/Z_0 \ra 0$, it follows that $\Hom{}{\bar\eta}{L}$ is exact for each line bundle $L$, hence $\bar\eta$ is distinguished exact, in particular exact in $\coh\XX$. We thus obtain a commutative diagram
$$
\xymatrix@R=18pt@C=18pt{
 & X \ar[r]^{u}\ar@{=}[d] & Y \ar[r]^{v}\ar@{=}[d] & Z \ar[r]\ar[d]^h&0\\
0\ar[r] & X \ar[r]^{u} & Y \ar[r]^{\bar{v}} & Z/Z_0 \ar[r]&0\\
}
$$
with exact rows, showing that $h$ is an isomorphism. Thus $Z_0=0$ such that also $Z$ is a vector bundle. The claim follows.
\end{proof}
A distinguished epimorphism $u:P\ra X$ in $\vect\XX$ in $\vect\XX$ is called a \define{projective cover} of $X$ if $P$ is projective and, moreover, the request that a composition $Y\up{v}P\up{u}X$ is a distinguished epimorphism forces $v$ to be a distinguished epimorphism. It follows that projective covers (if they exist) are unique up to isomorphism. We denote by $\projh{X}\up{\pi_X}X$ the projective cover (resp.\ by $X\up{j_X}\injh{X}$) the injective cover of $X$.

\begin{proposition}
The category $\vect\XX$, equipped with the exact structure given by the distinguished exact sequences, has projective and dually also injective covers.
\end{proposition}
\begin{proof}
Consider the category $\Ll$ of all line bundles. Using that the coordinate algebra $S$ is $\LL$-graded local noetherian, it follows that for each vector bundle $X$ there exists an irredundant finite system of morphisms $u_i:L_i\ra X$ generating the functor $\Ll(-,X)$. (Irredundant means that no $u_i$ factors through the system of the remaining $u_j$s.) It is then straightforward to check that $u=(u_i):\bigoplus_{i\in I}L_i\ra X$ is a projective cover of $X$. The proof for the existence of injective covers is dual.
\end{proof}
The interest in projective and injective covers comes from the next result.
\begin{proposition} \label{prop:suspension:definition}
Assume $X$ is a vector bundle. Consider the distinguished exact sequence $0\ra X \up{j_E}\injh{X}\up{\pi} Y \ra 0$. Then  $Y$ has no line bundle summands. Moreover, if $X$ has no line bundle summands, then $\injh{X}$ is a projective cover of $Y$.  In this case $X$ is indecomposable if and only if $Y$ is indecomposable.
\end{proposition}
\begin{proof}
Concerning the first claim assume that $L$ is a line bundle summand of $Y$. Then $L$ lifts to a line bundle summand of $\injh{X}$, showing that $\injh{X}/L$ is also an injective hull of $X$, thus contradicting minimality of the injective cover.
Next we how the second claim. Assume $Y=Y_1\oplus Y_2$. Forming the direct sum of the sequences $0\ra X_i\lra \projh{Y_i}\lra Y_i \ra 0$, $i=1,2$, it follows that $X=X_1\oplus X_2$, proving the claim.
\end{proof}

Since the category $\vect\XX$ is Frobenius with the system $\Ll$ of all line bundles as the indecomposable projective-injectives, a general result of Happel~\cite{Happel:1988} asserts that the attached stable category $\svect\XX=\vect\XX/[\Ll]$ is triangulated. Here, $[\Ll]$ denotes the two-sided ideal $[\Ll]$ of all morphisms factoring through a finite direct sum of line bundles. Moreover, the suspension functor $[1]$ of $\svect\XX$ is given by the formation of co-syzygies and, dually, the functor $[-1]$ is given by the formation of syzygies. Thus in the context of Proposition~\ref{prop:suspension:definition} we have $Y=X[1]$ and $X=Y[-1]$.

We recall that a triangulated category $\Tt$ is \define{homologically finite} if for any two objects we have $\Hom\Tt{X}{Y[n]}=0$ for $|n|\gg0$.
\begin{theorem}[Serre duality] \label{thm:Serre:duality}
The triangulated category $\svect\XX$ is Hom-finite, Krull-Schmidt and homologically finite. Moreover, $\svect\XX$ has Serre duality given by functorial isomorphisms
$$
{\sHom{}{X}{Y[1]}}=\dual{\sHom{}{Y}{X(\vom)}}.
$$
In particular, $\svect\XX$ has Auslander-Reiten triangles, and the shift with $\vom$ also serves as the AR-translation for $\svect\XX$.
\end{theorem}
\begin{proof}
As a factor category of a Hom-finite category $\svect\XX$ inherits Hom-finiteness, and the Krull-Schmidt property from $\vect\XX$. Concerning Serre-duality we argue as follows:
We apply the functor $\Hom{}{X}{-}$ (resp.\ $\Hom{}{-}{X(\vom)}$) to the exact sequence $\mu:0\ra Y \ra \injh{Y}\ra Y[1]\ra 0$ and obtain exact sequences
\begin{eqnarray}\label{eq:Serre:duality:1}
\Hom{}{X}{\injh{Y}}&\ra& \Hom{}{X}{Y[1]}\ra \sHom{}{X}{Y[1]}\ra 0\\
\Hom{}{\injh{Y}}{X(\vom)}&\ra&\Hom{}{Y}{X(\vom)}\ra \sHom{}{Y}{X(\vom)}\ra 0
\end{eqnarray}
By Serre duality in $\coh\XX$ dualization of the lower sequence yields exactness of
\begin{equation}\label{eq:Serre:duality:2}
0\ra\dual{\sHom{}{Y}{X(\vom)}}\ra \Ext1{}{X}{Y}\ra \Ext1{}{X}{\injh{Y}}.
\end{equation}
Invoking the long exact Hom-Ext sequence $\Hom{}{X}{\mu}$ we obtain from the two sequences \eqref{eq:Serre:duality:1} and \eqref{eq:Serre:duality:2} a natural isomorphism $\Hom{}{X}{Y[1]}=\dual{\Hom{}{Y}{X(\vom)}}$, as claimed.

To show homological finiteness of $\svect\XX$ we use Corollary~\ref{cor:double_extension} to analyze $\sHom{}{X}{Y[n]}$. If $n\ll0$ and $n=2m$ is even (resp.\ $n=2m+1$ is odd) we obtain $Y[n]=Y(m\vc)$ (resp.\ $Y[n]=Y[1](m\vc)$) with $m\ll0$. Invoking line bundle filtrations for $X,Y,Y[1]$ we conclude $\Hom{}{X}{Y[n]}=0$, hence $\sHom{}{X}{Y[n]}=0$, for $n\ll0$. For $n\gg0$ we invoke Serre duality $\sHom{}{X}{Y[n]}=\dual{\sHom{}{Y[n-1]}{X(\vom)}}$ and, distinguishing the cases $n$ even (resp.\ $n$ odd) argue in a similar way as before.
\end{proof}

Homological finiteness of $\svect\XX$ enables to introduce the \define{Euler form} on the Grothendieck group $\Knull{\svect\XX}$ by means of the expression
$$
\euler{\sclass{X}}{\sclass{Y}}=\sum_{n=-\infty}^{\infty}(-1)^n\dim_{k}{\sHom{}{X}{Y[n]}}
$$
where, for distinction, we denote classes in $\Knull{\svect\XX}$ by double brackets.

We conclude this section with a useful information on the component maps of the projective (resp.\ injective) cover of $X$.
\begin{proposition}
Let $X$ be a vector bundle and $L$ a line bundle. The following holds.

(i) If $v:X\ra L$ is a component map of the injective cover $j_X:X\ra \injh{X}$, then $v$ is an epimorphism in $\coh\XX$.

(ii) If $u:L\ra X$ is a component map of the projective cover $\pi_X:\projh{X}\ra X$, then the cokernel of $u$, formed in $\coh\XX$, is a vector bundle.
\end{proposition}
\begin{proof}
We only prove (i), the proof of (ii) is dual. Since $v$ is a component map of the injective cover, it is non-zero. Assume, for contradiction, that $v$ is not an epimorphism. Then $v$ has a non-trivial factorization $v=[X\up{v'}L'\incl L]$, where $L'$, defined as the image of $v$, is again a line bundle. We may thus replace the component map $v:X\ra L$ by $v':X\ra L'$, obtaining another injective cover of $X$, and contradicting uniqueness of the injective cover.
\end{proof}
\section{Indecomposable vector bundles of rank two} \label{sect:rank_two}
This section presents the key features of indecomposable vector bundles of rank two. All major results of this paper are based on these properties. They are derived from the very explicit knowledge of the \define{Picard group} $\Pic\XX$ of $\XX$ which, as a partially ordered group, is isomorphic to the grading group $\LL$ on generators $\vx_1,\vx_2,\vx_3$ with $\sum_{i=1}^3 \NN\vx_i$ as its positive `cone', see \cite{Geigle:Lenzing:1987}.

\subsection{Extension bundles and Auslander bundles}
We assume that $(p_1,p_2,p_3)$ is a weight triple with $p_i\geq2$ for
each $i=1,2,3$. We put
\begin{equation}
\vdom=\vc+2\vom=\sum_{i=1}^{3}(p_i-2)\vx_i,
\end{equation} and note that $0\leq
\vx\leq \vdom$ holds if and only if we have $\vx=\sum_{i=1}^3\ell_i\vx_i$ with
$0\leq \ell_i\leq p_i-2$ for each $i$. This expression for $\vx$ is unique, and
we will write $\ell_i(\vx)=\ell_i$ if necessary. We also note that
$\dual{\Ext1{}{L(\vx)}{L(\vom)}}=\Hom{}{L}{L(\vx)}=k$ holds for each $0\leq\vx\leq\vdom$.
\begin{definition}
  For a line bundle $L$ and $0\leq\vx\leq\vdom$ let $\eta_\vx: 0\ra
  L(\vom)\up{\iota} E\up{\pi} L(\vx)\ra 0$ be non-split. The central
  term $E=\extb{L}{\vx}$, which is uniquely defined up to isomorphism,
  is called the \define{extension bundle} given by the data
  $(L,\vx)$. If $L=\Oo$ then we just write $\extb{}{\vx}$.
\end{definition}
For $\vx=0$ the sequence $\eta_0:0 \ra
L(\vom)\up{\iota}\extb{L}{0}\up{\pi}L \ra 0$ is almost-split, and
$E(L):=\extb{L}{0}$ is called the \define{Auslander bundle} associated
with $L$. If $L=\Oo(\vx)$ we often write $E(\vx)$ for $E(\Oo(\vx))$.

Our next result highlights the special role of weighted projective
lines with three weights.

\begin{theorem} \label{thm:extension_bundle} Assume that $\XX$ is
given by a weight triple and $F$ is an indecomposable vector bundle of rank two. Then $F$ is
an extension bundle $F=\extb{L}{\vx}$ for some line bundle $L$ and $\vx\in\LL$ where
$0\leq\vx\leq\vdom$. Further, each of the spaces  $\Hom{}{F}{L(\vom)}$, $\Hom{}{L(\vx)}{F}$ and $\Ext1{}{F}{L(\vom)}$ vanishes.

Moreover, each extension bundle, hence $F$, is
exceptional\footnote{Exceptionality in $\svect\XX$ will be shown
later, see Corollary~\ref{cor:exceptional}} in $\coh\XX$, hence in $\Der{\coh\XX}$.
\end{theorem}

\begin{proof}
\underline{Step 1:} We show that $F$ is an extension bundle. For this we choose a line bundle $L$ such that (a) $\Hom{}{L(\vom)}{F}\neq0$ and (b) $\Hom{}{L(\vom+\vx_i)}{F}=0$ for $i=1,2,3$. For instance, we may choose $L$ to be of maximal degree such that (a) is satisfied. Since $F$ is indecomposable, we obtain a non-split exact sequence $\eta:0\ra L(\vom) \ra F \ra C \ra 0$. The cokernel term $C$ has rank one and because of (b) has no torsion. Indeed, if $F'/L(\vom)$ is a simple subsheaf of $C=F/L(\vom)$, then $F'$ has rank one, thus is a line bundle $L(\vy)$ for some $\vy$ in $\LL$ with $\vom<\vy$. By the definition of the order on $\LL$ this implies $\vom+\vx_i\leq\vy$ for some $i=1,2,3$, in contradiction to (b).  Summarizing we conclude that $C$ is a line bundle $L(\vx)$ for some $\vx$ from $\LL$. Since by assumption $0\neq\Ext1{}{L(\vx)}{L(\vom)}=\dual{\Hom{}{L}{L(\vx)}}$ we obtain $\vx\geq0$. Applying $\Hom{}{L(\vom+\vx_i)}{-}$ to $\eta$ and invoking (b) again, we obtain exactness of the sequence $0=(L(\vom+\vx_i),F)(L(\vom+\vx_i),L(\vx))\ra {}^1(L(\vom+\vx_i),L(\vom))=\dual{(L,L(\vom+\vx_i))}=0$. This yields $\Hom{}{L(\vom+\vx_i)}{L(\vx)}=0$, hence $\vx-\vom-\vx_i\leq\vc+\vom$ and then $\vx\leq(2\vom +\vc)+\vx_i$ for $i=1,2,3$, implying $\vx\leq\vdom$.

\underline{Step 2:} We prove a slightly more general claim and show that the middle term $F$ of the non-split exact sequence $\eta: 0\ra L(\vom) \ra F \ra L(\vx)\ra 0$ is exceptional in $\coh\XX$ for each $\vx=\sum_{i=1}^3 \ell_i\vx_i$, where $0\leq \ell_i\leq p_i-1$ for $i=1,2,3$.\footnote{Note, that we reserve the term `extension bundle' for the more restricted situation $0\leq \ell_i\leq p_i-2$.} We first apply $(L(\vx),\eta)$ to obtain an exact sequence $0=(L(\vx),L(\vom))\ra (L(\vx),F)\ra (L(\vx),L(\vx))\ra {}^1(L(\vx),L(\vom))\ra {}^1(L(\vx),F)\ra{}^1(L(\vx),L(\vx))=0$. Since $\eta$ does not split, the boundary morphism of the Hom-Ext sequence is an isomorphism, yielding
\begin{equation} \label{eq:1}
\Hom{}{L(\vx)}{F}=0\quad \text{and} \quad \Ext1{}{L(\vx)}{F}=0.
\end{equation}
Next, we form $(L(\vom),\eta)$ and obtain exactness of $0\ra (L(\vom),L(\vom))\ra (L(\vom),F)\ra (L(\vom),L(\vx))\ra {}^1(L(\vom),L(\vom))\ra {}^1(L(\vom),F)\ra {}^1(L(\vom),L(\vx))=\dual{(L(\vx),L(2\vom))}=0$. This yields \begin{equation} \label{eq:2} \Hom{}{L(\vom)}{F}=k \quad \text{and} \quad \Ext1{}{L(\vom)}{F}=0.
\end{equation}
Then, we form $(\eta,F)$ to obtain exactness of the sequence $0\ra (L(\vx),F)\ra (F,F)\ra(L(\vom,F))\ra {}^1(L(\vx),F)\ra {}^1(F,F)\ra{}^1(L(\vom,F))\ra 0$. By means of \eqref{eq:1} and \eqref{eq:2} we obtain $\End{}{F}=k$ and $\Ext1{}{F}{F}=0$, that is, exceptionality of $F$ in $\coh\XX$. In particular this concerns any extension bundle $\extb{L}{\vx}$ with $0\leq\vx\leq\vdom$.

Next we show that $\Hom{}{F}{L(\vom)}=0=\Hom{}{L(\vx)}{F}$. Assume first that $u:F\ra L(\vom)$ is non-zero and consider $u\circ\iota:L(\vom)\ra L(\vom)$. If $u\circ\iota$ is non-zero, it is an isomorphism and we arrive at the contradiction that the sequence $\eta_\vx$ splits. If $u\circ\iota=0$ then $u$ induces a non-zero map $L(\vx)\ra L(\vom)$ which, again, is impossible. The proof of the second claim is similar.
\end{proof}
Because of the obvious cubical pattern, we call the system $\extb{L}{\vx}$, $0\leq\vx\leq\vdom$, the \define{standard cuboid}\footnote{Because of Theorem~\ref{thm:std_tilting_obj}, it is later called the tilting cuboid.} based in $L$. For $L=\Oo$ we just speak of the standard cuboid. Note that the standard cuboid may degenerate to a rectangle, a line segment or even a point.

We remark that the preceding results are an instance of mutations for an exceptional pair, compare \cite{Gorodentsev:Rudakov:1987} and \cite{Bondal:1989} for this. The proper framework for mutations is the world of triangulated categories; however, a simple modification makes this concept also work for hereditary abelian categories.  Recall therefore that a sequence of exceptional objects $(E_1,\ldots,E_n)$ in $\coh\XX$ is an \define{exceptional sequence} if for all indices $i>j$ and all integers $p$ we have $\Ext{p}{}{E_i}{E_j}=0$. (Here, only $p=0,1$ matter.) If $n=2$ we speak of an \define{exceptional pair}. In the above setup, $(L(\vx),L(\vom))$ is an exceptional pair with $\Hom{}{L(\vx)}{L(\vom)}=0$ and $\Ext1{}{L(\vx)}{L(\vom)}=k$. In this context, where $\Hom{}{L(\vx)}{L(\vom)}=0$, the \define{left mutation} $\Ll_{L(\vx)}L(\vom)$ of $L(\vom)$ over $L(\vx)$ is given by the central term $\extb{}{\vx}$ of the non-split short exact sequence $0\ra L(\vom)\ra \extb{}{\vx}\ra L(\vx)\ra 0$, yielding by general theory a new exceptional pair $(\extb{}{\vx},L(\vx))$.

\begin{lemma} \label{lem:extb:dual} Assume $L$ is a line bundle and
  $0\leq\vx\leq\vdom$. Then $$\odual{(\extb{L}{\vx})}=\extb{\odual{L}}{\vx}(-(\vx+\vom))$$
  holds. In particular, for $L=\Oo$ we obtain
  $\odual{\extb{}{\vx}}=\extb{}{\vx}(-(\vx+\vom))$.
\end{lemma}
\begin{proof}
Applying vector bundle duality and the grading shift by $\vx+\vom$ to the defining sequence $\eta_\vx: 0\ra L(\vom)\ra \extb{L}{\vx}\ra L(\vx)\ra 0$ for $\extb{L}{\vx}$, we obtain a non-split exact sequence $0\ra\odual{L}(\vom)\ra (\extb{L}{\vx})(\vx+\vom)\ra \odual{L}(\vx)\ra 0$ which implies the claim.
\end{proof}

We now show that there is a natural link between extension bundles and Auslander bundles by means of a \define{natural map} $\pi:\extb{L}{\vx}\ra E_L(\vx)$, provided $0\leq\vx\leq\vdom$.
\begin{proposition} \label{prop:natural_map} Assume $L$ is a line
  bundle, and $0\leq \vx\leq \vdom$,
  $\vx=\sum_{i=1}^3\ell_i\vx_i$. Then
  $\Hom{}{\extb{L}{\vx}}{E_L(\vx)}=k$ and there exists a map
  $\pi:\extb{L}{\vx}\ra E_L(\vx)$ yielding a commutative diagram
$$
\xymatrix@C18pt@R18pt{
&&0\ar[d]&0\ar[d]&&\\
\eta_\vx&0\ar[r]&L(\vom)\ar[r]^\iota\ar[d]_{x}&\extb{L}{\vx}\ar[d]^\pi\ar[r]^\pi& L(\vx)\ar[r]\ar@{=}[d]&0\\
\eta(\vx)&0\ar[r]&L(\vx+\vom)\ar[d]\ar[r]^{\al}& E_L(\vx)\ar[d]\ar[r]^{\be}&L(\vx)\ar[r]&0\\
&       &U_x\ar@{=}[r]\ar[d]&U_x\ar[d]&&&\\
&       &0            &0\\
}
$$
with exact rows and columns.  Here, the morphism $L(\vom)\up{x} L(\vx+\vom)$ is induced by multiplication with $x=\prod_{i=1}^3 x_i^{\ell_i}$. Further $U_x=\bigoplus_{i=1}^3\simp{i}(\vx+\vom)^{[\ell_i]}$ where $\simp{i}$ is the unique simple sheaf, concentrated in $x_i$, with $\Hom{}{L}{\simp{i}}\neq0$, and further $\simp{i}(\vx+\vom)^{[\ell_i]}$ is the unique indecomposable sheaf with socle $\simp{i}(\vx+\vom)$ and length $\ell_i$.
\end{proposition}
Keeping the above assumptions, we will show in Proposition~\ref{prop:extb:abdle} that also $\sHom{}{\extb{L}{\vx}}{E(\vx)}$ is one-dimensional.
\begin{proof}
We first recall that $\eta(\vx)$ is almost-split. Hence there exists a map $\pi:\extb{}{\vx}\ra E(\vx)$ making the right square commutative. Further, $\pi$ induces $\pi':L(\vom)\ra L(\vx+\vom)$ making the left square commutative. We claim that $\pi'$ is non-zero since otherwise $\eta(\vx)$ would split, a contradiction. Since $\Hom{}{L(\vom)}{L(\vx+\vom)}=kx$ we obtain that $\pi'$ is a non-zero multiple of $x$. Modifying maps, we then arrive at the commutative diagram above. The form of the cokernel of $x:L(\vom)\ra L(\vx+\vom)$ then follows from the theory of weighted projective lines.

It remains to show that $\Hom{}{\extb{}{\vx}}{E(\vx)}$ is one-dimensional. Applying $(\extb{}{\vx},-)$ to the sequence $\eta(\vx)$ we first obtain exactness of
$$
(\extb{}{\vx},L(\vx+\vom))\ra (\extb{}{\vx},E(\vx))\ra(\extb{}{\vx},L(\vx))\ra\Ext1{}{\extb{}{\vx}}{L(\vx+\vom)}
$$
Applying Serre duality to formula \eqref{eq:1} we obtain that the end terms of this sequence are zero, implying $\Hom{}{\extb{}{\vx}}{E(\vx)}\iso\Hom{}{\extb{}{\vx}}{L(\vx)}$. Applying $(-,L(\vx))$ to the sequence $0\ra L(\vom)\ra \extb{}{\vx}\ra L(\vx)\ra 0$, we obtain $\Hom{}{\extb{}{\vx}}{L(\vx)}=k$ since $L(\vx)$ has a trivial endomorphism ring. Further $\Ext1{}{\extb{\vx}}{L(\vx)}=0$ by Theorem~\ref{thm:extension_bundle}. Putting things together we have shown that $\Hom{}{\extb{}{\vx}}{E(\vx)}=k$.
\end{proof}

\subsection{Injective and projective covers}
In order to deal with extension bundles it is useful to know their
projective and injective covers. The next Lemma will be needed
for this.

\begin{lemma}\label{lem:induction-step-pull-back}
Assume $0\leq \vx \leq \vx+\vy\leq \vdom$, thus $\vy=\sum_{i=1}^3e_i\vx_i$ with $e_i\geq0$. We consider $y=\prod_{i=1}^3x_i^{e_i}$ as a morphism $y:L(\vx)\ra
    L(\vx+\vy)$. Then there is a commutative diagram with exact rows
\begin{equation*}
  \xymatrix@R=18pt@C=18pt{
    \eta_{\vx+\vy}:0\ar[r]&L(\vom) \ar[r]^\iota &
    \extb{L}{\vx+\vy}\ar[r]^\pi& L(\vx+\vy)\ar[r]&0\\
    \eta_{\vx}:0\ar[r]&L(\vom) \ar[r]^\iota\ar@{=}[u] &
    \extb{L}{\vx}\ar@{}[ru]|{\textrm{pull-back}}\ar[r]^\pi\ar[u]^{\str{y}}
    & L(\vx)\ar[u]_y\ar[r]&0\\
  }
\end{equation*}
We have $\Hom{}{\extb{L}{\vx}}{\extb{L}{\vx+\vy}}=k\str{y}\neq0$. Moreover, for  $i\neq j$ and $0\leq \vx < \vx+vx_i+\vx_j\leq\vdom$ the diagram below is commutative.
\begin{equation*}
\xymatrix@R=18pt@C=18pt{
\extb{}{\vx+\vx_j}\ar[r]^{\str{x_i}}&\extb{}{\vx+\vx_i+\vx_j}\\
\extb{}{\vx}\ar[r]^{\str{x_i}}\ar[u]^{\str{x_j}}&\extb{}{\vx+\vx_i}\ar[u]_{\str{x_j}}\\
}
\end{equation*}
\end{lemma}
\begin{proof}
By our assumption on $\vx$ the extension space $\Ext1{}{L(\vx)}{L(\vom)}$ is one-dimensional. It thus suffices to show that the pull-back $\eta_{\vx+\vx_i}.y$ along $y:L(\vx)\ra L(\vx+\vy)$ is non-zero, and then isomorphic to $\eta_\vx$. This is shown invoking Serre duality yielding the commutative diagram
\begin{equation*}
  \xymatrix@R=20pt@C=20pt{
    \Ext1{}{L(\vx+\vy)}{L(\vom)}\ar[r]^\cong\ar[d]_{\Ext1{}{y}{L(\vom)}} & \dual{\Hom{}{L}{L(\vx+\vy)}}\ar[d]^{\dual{\Hom{}{L}{y}}}\\
    \Ext1{}{L(\vx)}{L(\vom)}\ar[r]^\cong & \dual{\Hom{}{L}{L(\vx)}}\\
  }
\end{equation*}
Since $\vx$ and $\vx+\vy$ belong to the range $0\leq \vz \leq \vdom$, the map $\dual{\Hom{}{L}{y}}$ is an isomorphism. Hence also $\Ext1{}{y}{L(\vom)}$ is an isomorphism, and consequently the pull-back $\eta_{\vx+\vy}.y$ is non-zero. This shows the first claim. Concerning the second assertion, we first apply the Hom-functor $(-,L(\vx+\vy))$ to the sequence $\eta_\vx$ and deduce from the exactness of $0\ra (L(\vx),L(\vx+\vy))\ra (\extb{L}{\vx},L(\vx+\vy))\ra (L(\vom),L(\vx+\vy))=0$ that $\Hom{}{\extb{L}{\vx}}{L(\vx+\vy)}=k$. Next, we apply $(\extb{L}{\vx},-)$ to the sequence $\eta_\vx$ and obtain exactness of $0=(\extb{L}{\vx},L(\vom))\ra (\extb{L}{\vx},\extb{L}{\vx+\vy})\ra (\extb{L}{\vx},L(\vx+\vy))=k$ implying that $\Hom{}{\extb{L}{\vx}}{\extb{L}{\vx+\vy}}=k$ since we know already that it is non-zero. By a similar argument one shows that $\Hom{}{\extb{}{\vx}}{\extb{}{\vx+\vx_i+\vx_j}}=k$ implying commutativity $\str{x_i}\str{x_j}=\str{x_j}\str{x_i}$.
\end{proof}

\begin{theorem}[Projective and injective covers] \label{thm:hulls} Assume $\XX$ is given by the weight
  triple $(p_1,p_2,p_3)$. Let $\extb{L}{\vx}$, $0\leq\vx\leq\vdom$, be an extension
  bundle. With $\ell_i:=\ell_i(\vx)$ its injective hull
  $\injh{\extb{L}{\vx}}$ and its projective hull
  $\projh{\extb{L}{\vx}}$ are given by the following expressions:
\begin{eqnarray}
  \injh{\extb{L}{\vx}}&=&L(\vx)\oplus \bigoplus_{i=1}^3
  L((1+\ell_i)\vx_i+\vom),\label{eq:injhull}\\
  \projh{\extb{L}{\vx}}&=&L(\vom)\oplus\bigoplus_{i=1}^3
  L(\vx-(1+\ell_i)\vx_i).\label{eq:projhull}
\end{eqnarray}
Further, the four line bundle summands $(L_i)_{i=0}^3$ of $\injh{\extb{L}{\vx}}$ (resp.\ $\projh{\extb{L}{\vx}}$) are mutually
\define{Hom-orthogonal}, that is, they are satisfying
\begin{equation}\label{eq:hom-orth}
  \Hom{}{L_i}{L_j}=\begin{cases}k & \textrm{if } i=j \\  0 &
    \textrm{if } i\neq j.   \end{cases}
\end{equation}
\end{theorem}

\begin{proof}
The Hom-orthogonality~\eqref{eq:hom-orth} is easy to show. Concerning the hulls it suffices to deal with the injective case, since the arguments for the projective hulls are dual. The proof is by induction on $n=\sum_{i=1}^3 \ell_i$. For $n=0$ the assertion reduces to the claim that for the Auslander bundle $E$ given by the almost-split sequence $\eta\colon 0\ra\Oo(\vom)\stackrel{\iota}\ra E\stackrel{\pi}\ra\Oo\ra 0$ we have to show $\injh{E}=\Oo\oplus\bigoplus_{i=1}^3 \Oo(\vx_i+\vom)$. The four components of the inclusion $j_E\colon E\ra\injh{E}$ are given by $E\up{\pi}\Oo$ and the morphisms $E\up{x_i^\ast}\Oo(\vx_i+\vom)$, $1\leq i \leq 3$,  given by extending the map $\Oo(\vom)\stackrel{x_i}\ra\Oo(\vx_i+\vom)$ to $E$ (using that $\eta$ is almost-split) ($1\leq i\leq 3$).

We thus claim that the map $j_E\colon E\up{(\pi,(x_i^\ast))}\Oo\oplus\bigoplus_{i=1}^3 \Oo(\vx_i+\vom)$ is `the' injective hull of $E$ in $\vect{\XX}$. First we show that $j_E$ is a distinguished morphism. Let $h\colon E\ra L$ be a map to a line bundle $L$. Then $h'=h\iota$ is not an isomorphism since $\iota$ does not split. It follows that $h'$ has a presentation $h'=\sum_{i=1}^3 h_i x_i$, where $x_i$ is the obvious map $\Oo(\vom)\up{x_i}\Oo(\vx_i+\vom)$ and $h_i\colon\Oo(\vx_i+\vom)\ra L$ ($1\leq i\leq 3$). Now $\bigl(h-\sum_{i=1}^3 h_i x_i^\ast\bigr)\circ\iota=0$. Hence $h-\sum_{i=1}^3 h_i x_i^\ast$ factors through $\pi\colon E\ra\Oo$, say $h-\sum_{i=1}^3 h_i x_i^\ast=g\circ\pi$. Hence $h$ factors through $(\pi,(x_i^\ast))$, as claimed. Minimality of $j_E$ follows from the fact that the line bundle constituents of $\injh{E}$ are mutually Hom-orthogonal. This finishes the claim for $n=0$, that is, $\vx=0$. (Remark: This part of the proof extends to an arbitrary number of weights.)  Passing to the general case, we assume $0\leq\vx\leq\vx+\vx_i\leq\vdom$ and that the formula~\eqref{eq:injhull} holds for $\vx$. Invoking Lemma~\ref{lem:induction-step-pull-back} the induction step is straightforward.
\end{proof}

\begin{corollary}
The injective (resp.\ projective) hull of an Auslander bundle $E_L(\vx)$ are given by
\begin{eqnarray}
\injh{E_L(\vx)}&=&L(\vx)\oplus \bigoplus_{i=1}^3 L(\vx_i+\vom),\\
\projh{E_L(\vx)}&=&L(\vom)\oplus\bigoplus_{i=1}^3 L(\vx-\vx_i).
\end{eqnarray}
~\qed
\end{corollary}

Let $E$ be an indecomposable bundle and $\injh{E}$ be its injective hull in the Frobenius category $\vect\XX$. There results a distinguished exact sequence $0\ra E \ra \injh{E}\ra E[1] \ra 0$, where $E[1]$ is again indecomposable, and $E[1]$ takes the role of the suspension of $E$ in $\svect\XX$.

\begin{corollary}[Suspension for extension bundles] \label{cor:suspension:extb}
Let $L$ be a line bundle, and $0\leq\vx\leq\vdom$. Then the suspension of $\extb{L}{\vx}$ is given by
\begin{equation} \label{eq:suspension}
\extb{L}{\vx}[1]\iso\extb{L}{\vdom-\vx}(\vx-\vom)
\end{equation}or, in different notation,
$\tau(\extb{L}{\vx}[1])\iso\extb{L}{\vdom-\vx}(\vx)$. Similarly,
\begin{equation}\label{eq:inverse:suspension}
\extb{L}{\vx}[-1]=\extb{L}{\vdom-\vx}(\vx+\vom-\vdom),
\end{equation} \label{eq:inverse_suspension}
or $\tau^{-1}(\extb{L}{\vx})=\extb{L}{\vdom-\vx}(\vx-\vdom)$.

In particular, the suspension sends the Auslander bundle $E(\vx)$ to $(E(\vx))[1]=\extb{}{\vdom}(\vx-\vom)$.
\end{corollary}

\begin{proof}
For assertion~\eqref{eq:suspension} we have to show that $G:=\extb{L}{\vx}[1]$ and $H:=\extb{L}{\vdom-\vx}(\vx-\vom)$ are isomorphic. Noting that $G$ and $H$ are exceptional, this will follow from a general result of H\"{u}bner, see~\cite{Huebner:1996} or \cite[Prop. 4.4.1]{Meltzer:2004}, by showing that the classes $[G]=[\injh{\extb{L}{\vx}}]-[\extb{L}{\vx}]= \sum_{i=1}^3 [L(\vom+(1+\ell_i)\vx_i)]-[L(\vom)]$ and $[H]=[L(\vx)]+[L(\vom+\vc)]$ in the Grothendieck group $\Knull{\coh\XX}$ are the same. We are going to prove this by induction on $\sum_{i=1}^3 \ell_i$.

\emph{The case $\vx=0$}. Recall that there are exact sequences $0\ra L(-\vx_i)\up{x_i}L \ra \simp{i}\ra 0$ and $0\ra L(-\vc)\ra L \ra \simp{x}\ra 0$ defining simple sheaves $\simp{i}$, concentrated in an exceptional point $x_i$ ($1\leq i \leq 3$) and $\simp{x}$, concentrated in an ordinary point $x$. By means of \cite[Prop. 4.7]{Lenzing:1999} or
\cite[Prop. 8.1]{Lenzing:1996} we have
\begin{equation}\label{eq:Ktheoretic_tau}
[L(\vom)]=[L]-\sum_{i=1}^3[\simp{i}]+[\simp{x}].
\end{equation}
We now use that $\simp{i}(\vx_j)$ equals $\simp{i}$ or $\simp{i}(-\vom)$
according as $j\neq i$ or $j= i$, and further that
$[\simp{x}(\vy)]=[\simp{x}]$ holds for each $\vy$ from $\LL$. Then we
can rewrite \eqref{eq:Ktheoretic_tau} as
$[L(\vom)]=[L]-\sum_{i=1}^3[\simp{i}(\vom+\vx_i)]+[\simp{x}(\vom+\vc)]$
thus obtaining
\begin{eqnarray*}
  0&=&[L]-\sum_{i=1}^3\left([L(\vom+\vx_i)]-[L(\vom)]\right)+[L(\vom+\vc)]-2[L(\vom)]\\
  &=&\left([L]+[L(\vom+\vc)]\right)-\left(\sum_{i=1}^3[L(\vom+\vx_i)]-[L(\vom)]\right)\\
  &=& [H]-[G],
\end{eqnarray*}
forcing $[H]=[G]$ and settling the claim for $\vx=0$.

Next, we treat the induction step from $\vx$ to $\vx+\vx_i$, assuming
$\vx+\vx_i\leq\vdom$. We thus assume that
$[\extb{L}{\vx}[1]]=[\extb{L}{\vdom-\vx}(\vx-\vom)]$ and want to show
$[\extb{L}{\vx+\vx_i}[1]]=[\extb{L}{\vdom-\vx-\vx_i}(\vx+\vx_i+\vom)]$. We note that the two differences below
$$
[\extb{L}{\vx+\vx_i}[1]]=[L(\vom+(2+\ell_i)\vx_i)]-[L(\vom+(1+\ell_i)\vx_i)]=[\simp{i}(\vom+(2+\ell_i)\vx_i)]
$$
$$
[\extb{L}{\vdom-\vx-\vx_i}(\vx+\vx_i+\vom)]-[\extb{L}{\vdom-\vx}(\vx-\vom)]=[L(\vx+\vx_i)]-[L(\vx)]=[\simp{i}(\vx-\vx_i)]
$$
yield the same result, because $\simp{i}(\vom+\vx_i)=\simp{i}$, which
proves the first claim. The second claim \eqref{eq:inverse:suspension} is then a straightforward consequence.
\end{proof}

\begin{remark} \label{rem:suspension:not:preserves:rank}
In view of Theorem~\ref{thm:extension_bundle} the preceding result states that the suspension functor $[1]$ of $\svect\XX$ sends indecomposables of rank two to indecomposables of rank two. It will, however, not always preserve Auslander bundles: Assume tubular weight type $(3,3,3)$, and let $E$ be an Auslander bundle. Then $E[1]=\extb{}{\vdom}$ has slope $3/2$; but Auslander bundles have an integral slope for tubular weight triples. See also Proposition~\ref{prop:suspension:Ltwist} for a more general result.

Further, it is not true in general, that the suspension always preserves the rank (of indecomposables). To see this, assume again weight type $(3,3,3)$, and consider the tube $\Tt$ of $\tau$-period three in $\vect\XX$ containing $\Oo$ as a quasi-simple member. The suspension $[1]$ sends the tube $\underline{\Tt}$ in $\svect\XX$ to a tube $\underline{\Tt}'$ in $\svect\XX$ of slope $3/2$ whose quasi-simples have rank two. As a consequence, each indecomposable of $\Tt$ of quasi-length $n+1$ (hence of rank $n+1$) is sent by $[1]$ to an indecomposable of quasi-length $n$ (hence of rank $2n$) in $\Tt'$. In view of these facts, it is quite remarkable that rank two (for indecomposables) is preserved under suspension.

To make the picture complete we remark that, for weight type $(2,a,b)$ the suspension functor indeed preserves the rank of indecomposables (of rank $\geq2$), see Theorem~\ref{prop:2ab:suspension}. Moreover, for weight triples in general, two-fold suspension [2] always preserves the rank of indecomposables (of rank $\geq 2$) by Corollary~\ref{cor:double_extension}.
\end{remark}

\begin{lemma} \label{lemma:determinant}
Assume we are dealing with a weight triple.  Then the following holds.
\begin{enumerate}
\item[(i)] For each $\vx$ in $\LL$ and vector bundle $E$ we obtain
  $\det(E(\vx))=\det(E)+\rank(E)\vx$.
\item[(ii)] If $E$ is an indecomposable bundle of rank two, then
  $\det(E[1])=\det(E)+\vc$.
\item[(iii)] Assume $E,F$ are indecomposable bundles of rank two, then
  $\sHom{}{E}{F}\neq0$ implies $\det(F)-\det(E)\geq0$ with equality if
  and only if $E\iso F$.
\end{enumerate}
\end{lemma}
\begin{proof}
The first assertion follows by using a line bundle filtration for $E$. For assertion (ii) we may assume that $E$ is an extension bundle $\extb{L}{\vx}$ with $0\leq\vx\leq\vdom$. Passing to determinants for $0\ra E\ra \injh{E}\ra E[1]\ra 0$, the expression \eqref{eq:injhull} for the injective hull yields that $\det(E[1])=\det(E)+\vc$. Concerning (iii) we note that any $u:E\ra F$ yielding a non-zero member from $\sHom{}{E}{F}$ is a monomorphism, since otherwise $u$ would factor through a line bundle, the image of $u$. This implies that the cokernel $C$ of $u$ has finite length, such that the determinant $\det(C)$ of $C$, the sum of the determinants of the simple composition factors of $C$ is $\geq0$ in $\LL$ and, moreover, equal to zero if and only if $C$ has length zero. The claim now follows from $\det(F)-\det(E)=\det(C)$ by using that a rank zero object $C$ has $\det(C)=0$ if and only if $C=0$.
\end{proof}

\begin{corollary} \label{cor:exceptional}
Each indecomposable vector bundle $E$ of rank two is exceptional in $\svect\XX$.
\end{corollary}
\begin{proof}
By Theorem~\ref{thm:extension_bundle} we can assume that $E$ is an extension bundle. Passing to determinants in $0\ra E\ra \injh{E}\ra E[1]\ra 0$ the expression \eqref{eq:injhull} for the injective hull then yields that $\det(E[1])=\det(E)+\vc$. We know already that $\End{}{E}=k$ such that $\sEnd{}{E}=k$ follows. By Serre duality we further have $\sHom{}{E}{E[n]}=\dual{\sHom{}{E[n-1]}{E(\vom)}}$, and we have to prove that this expression is zero for each non-zero integer $n$. Assume for contradiction that it is non-zero for some $n\neq0$. It follows that the two inequalities (a) $n\vc\geq0$ and (b) $(n-1)\vc\leq 2\vom$ hold. Now, (a) is violated for $n<0$ and (b) is violated for $n>0$, thus proving our claim.
\end{proof}

\subsection{Comparison with Auslander bundles}
As before, let $E$ be the Auslander bundle given as the central term of the almost-split sequence $\eta:0\ra \Oo(\vom)\up{\iota}E\up{\pi}\Oo\ra 0$.
\begin{proposition} \label{prop:non_zero}
Let $X$ be a vector bundle without line bundle summands. We write its projective hull in the form $$\projh{X}=\bigoplus_{j\in J}\Oo(\vy_j+\vom)\up{u=(u_j)}X.$$ Then $\sHom{}{E(\vx)}{X}$ is non-zero if and only if $\vx=\vy_i$ for some $i\in J$.

In this case, the space $\sHom{}{E(\vx)}{F}$ has $k$-dimension one if, moreover, the $\Oo(\vy_j+\vom)$, $j\in J$, are mutually Hom-orthogonal.
\end{proposition}

\begin{proof}
Fixing an index $i$ of $J$, we first show that $\sHom{}{\extb{}{\vy_i}}{X}\neq0$.  Using the Auslander-Reiten property of $\eta$, we extend $u_i:\Oo(\vy_i+\vom)\ra X$ to the Auslander bundle $E(\vy_i)$, yielding a non-zero map $v:E(\vy_i)\ra X$. We claim that $v$ is non-zero in $\sHom{}{E(\vy_i)}{X}$. Assuming for contradiction that $v$ factors through $\projh{X}$, we then obtain a map $\bv:E(\vy_i)\ra \projh{X}$ and a diagram
$$
\xymatrix{
0\ar[r]& \Oo(\vy_i+\vom)\ar[r]^\al\ar[d]_{u_i}\ar[rd]^{\bar\al=(\al_j)}& E(\vy_i)\ar[r]^\be\ar[d]^{\bv=(\bv_j)}                    & \Oo(\vy_i)\ar[r]&0\\
       &   X                        & \bigoplus_{j\in J}\Oo(\vy_j+\vom)\ar[l]_{u=(u_j)}&&\\
}
$$
where  $v=u\bv$. We further put $\bar\al=\bv\al$ and thus obtain
$$
(*)\quad u_i=\sum_{j\in J}u_j \al_j,
$$
with maps $\al_j:\Oo(\vy_i+\vom)\ra \Oo(\vy_j+\vom)$ defined for each $j\in J$.

We distinguish two cases. \underline{Case (1):} We have $\al_i=0$. In this case (*) implies that $u_i=\sum_{j\neq i}u_j\al_j$, contradicting minimality of the projective hull $\projh{X}$.  \underline{Case (2):} We have $\al_i\neq0$. In this case $\al_i$ is an automorphism of $\Oo(\vy_i+\vom)$, implying that the map $\bar\al=(\al_j):\Oo(\vy_i+\vom)\ra \projh{X}$ is a split monomorphism. This implies the splitting of $\al$, hence of the almost-split sequence defining $E(\vx)$. This again is impossible. We have thus shown that $v:E(\vy_i)\ra X$ does not factor through a direct sum of line bundles.

Conversely assume that for some $\vx\in \LL$ we have some non-zero $h$ in $\sHom{}{E(\vx)}{X}$. Then $h'=h\al:\Oo(\vx+\omega)\ra X$ is non-zero since otherwise $h$ would factor through the line bundle $\Oo(\vx)$, contradicting our assumption on $h$. Since $u=(u_j)$ is a distinguished epimorphism, the map $h'$ lifts to $\projh{X}$ yielding a factorization $[\Oo(\vx+\vom)\up{h'}X]=[\projh{X}\up{u}X]\circ[\Oo(\vx+\vom)\up{v}\projh{X}]$, where $v=(v_j)$ and $u=(u_j)$.  We claim that at least one component $v_i:\Oo(\vx+\vom)\ra\Oo(\vy_i+\vom)$ of $v$ is an isomorphism, then forcing the wanted result $\vx=\vy_i$. Assume, for contradiction, that this is not the case. Invoking that $E(\vx)$ is defined by an almost-split sequence, then $v$ extends to $E(\vx)$, yielding a map $\bv:E(\vx)\ra \projh{X}$. We thus obtain a diagram as follows where the square and the lower triangle are commutative.
$$
\xymatrix@R=20pt@C=20pt{
E(\vx)\ar[r]^h\ar[rd]^{\bv} &   X \\
\Oo(\vx+\vom)\ar[u]_\al\ar[r]^{v=(v_j)}&\bigoplus_{j\in J}\Oo(\vy_j+\vom)\ar[u]_{u=(u_j)}\\
}
$$
Now $(h-u\bv)\al=0$. Hence $h-u\bv$ factors through $\Oo(\vx)$. Additionally $u\bv$ factors through $\projh{X}$, and then $h$ factors through $\Oo(\vx)\oplus\projh{X}$. This contradicts our assumption that $h$ is non-zero in the stable category, and finishes the proof.
\end{proof}
The theorem implies that there are sufficiently many Auslander bundles in the following sense.
\begin{corollary}\label{cor:Auslander-bundles-generate}
Assume $X$ has rank $\geq2$ and no line bundle summand. Then there exists an Auslander bundle $E(\vx)$ with $\sHom{}{E(\vx)}{X}\neq0$.~\qed
\end{corollary}
We put $\bx_i=\vx_i+\vom$ for $i=1,2,3$.
\begin{corollary} \label{cor:abundles:stable:morphisms}
Let $E_L$ be an Auslander bundle and let $\extb{L}{\vx}$ be an extension bundle, $\vx=\sum_{i=1}^3\ell_i\vx_i$. Then $\sHom{}{E_L(\vy)}{\extb{}{\vx}}\neq0$ if and only if $\vy$ is one of $\vom$, or $\vx-(1+\ell_i)\vx_i$ with $i=1,2,3$. Moreover in this case
$\sHom{}{E_L(\vy)}{\extb{L}{\vx}}\neq0$ has $k$-dimension one.

In particular $\sHom{}{E}{E(\vx)}\neq0$ if and only if $\vx\in\set{0,\bx_1,\bx_2,\bx_3}$, and in this case $\sHom{}{E}{E(\vx)}$ is isomorphic to $k$.
\end{corollary}
\begin{proof}
This uses Theorem~\ref{thm:hulls}.
\end{proof}
Next, we investigate the action of the Picard group on indecomposable bundles of rank two.
\begin{proposition} \label{prop:cuboid:symmetry}
Let $\vx=\sum_{i=1}^3\ell_i\vx_i$ with $0\leq\ell_i\leq p_i-2$. We put $\vy=\ell_2\vx_2+\ell_3\vx_3$ and $\vz=\ell_1\vx_1$. Then we obtain
 \begin{equation}
\extb{L}{\vx}(\bx_1-\vy)\iso\extb{L}{(p_2-2)\vx_2+(p_3-2)\vx_3-\vy+\vz}.
\end{equation}
\end{proposition}
\begin{proof}
Using the defining line bundle filtrations of the two extension bundles we have to show that $a_1=a_2$
holds in the Grothendieck group of $\coh\XX$ where
$$
a_1=[L((p_2-2)\vx_2+(p_3-2)\vx_3-\vy+\vz)]-[L(\bx_1+\vom-\vy)]\textrm{ and } a_2=[L(\vx+\bx_1-\vy)]-[L(\vom)].
$$
The two differences $\left((p_2-2)\vx_2+(p_3-2)\vx_3)-\vy+\vz\right)-\left(\bx_1+\vom-\vy\right)$ and $\left((\vx+\bx_1-\vy)-\vom\right)$ both evaluate to $(1+\ell_1)\vx_1$. Hence $a_1$ (resp.\ $a_2$) is the class of the cokernel term $C_1$ (resp.\ $C_2$) of the exact sequence
$$
0 \ra L(\bx_1+\vom-\vy)\up{x_1^{1+\ell_1}}L((p_2-2)\vx_2+(p_3-2)\vx_3-\vy+\vz)\lra C_1 \ra 0, \textrm{ resp.}
$$
$$
0 \ra L(\vom) \up{x_1^{1+\ell_1}}L(\vx+\bx_1-\vy)\lra C_2 \ra 0.
$$
Let $\simp{1}$ be the simple sheaf concentrated in the first exceptional point with $\Hom{}{L}{\simp{1}}=k$, then $C_1$ and $C_2$ are indecomposable sheaves of the same length $1+\ell_1$ and having the same top. We note for this that the top of $C_1$ (resp.\ of $C_2$) is the simple sheaf $\simp{1}((p_2-2)\vx_2+(p_3-2)\vx_3-\vy+\vz)=\simp{1}(\vz)$ (resp.\ the simple sheaf $\simp{1}(\vx_1+\bx_1-\vy)=\simp{1}(\vz)$). Summarizing, we obtain $a_1=a_2$, proving the claim.
\end{proof}

\begin{corollary} \label{cor:four:abundles}
Assume a weight triple not of type $(2,2,n)$. Then the standard cuboid (for type $(2,a,b)$ the standard rectangle) contains exactly four Auslander bundles  $E=\extb{}{0}$, $E(\bx_1)\iso\extb{}{(p_2-2)\vx_2}+(p_3-2)\vx_3$, $E(\bx_2)\iso\extb{}{(p_1-2)\vx_1+(p_3-2)\vx_3}$ and $E(\bx_3)\iso\extb{}{(p_1-2)\vx_1+(p_2-1)\vx_2}$.
\end{corollary}
\begin{proof}
For each $0\leq\vx\leq\vdom$ we have $\sHom{}{E}{\extb{}{\vx}}=k$. By Corollary~\ref{cor:abundles:stable:morphisms} this implies that $\extb{}{\vx}$ is an Auslander bundle if and only if $\extb{}{\vx}\iso E(\vy)$ with $\vy\in\set{0,\bx_1,\bx_2,\bx_3}$. On the other hand, the proposition shows that indeed $E(\bx_1)=\extb{}{(p_2-2)\vx_2}+(p_3-2)\vx_3$, $E(\bx_2)=\extb{}{(p_1-2)\vx_1+(p_3-2)\vx_3}$ and $E(\bx_3)=\extb{}{(p_1-2)\vx_1+(p_2-1)\vx_2}$.
\end{proof}
Proposition~\ref{prop:cuboid:symmetry} and Corollary~\ref{cor:four:abundles} have the following geometric interpretation:
$$
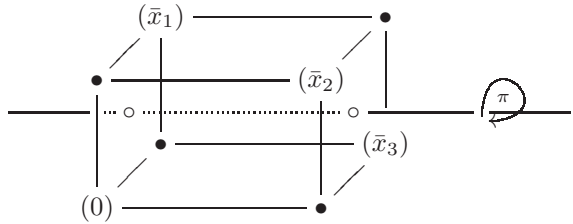
\def\b{\bullet}
\def\c{\circ}
\xymatrix@-1pc@!C=4pt@!R=0pt{
&&&&&(\bx_1)\ar@{-}[rrrrrrr]\ar@{-}[dddd]&&&&&&&\b\ar@{-}'[ddd][dddd]\\
&&&&&&&&&&\\
&&&\b\ar@{-}[rrrrrrr]\ar@{-}[dddd]\ar@{-}[rruu]&&&&&&&(\bx_2)\ar@{-}[dddd]\ar@{-}[rruu]&&\\
\ar@{-}[rrr]&&&\ar@{..}[r]&\c\ar@{..}[rrrrrrr]&&&&&&&\c\ar@{-}'[rrrr][rrrrrrr]&&&&\ar@<-0.8ex>@(u,r)[]_\pi&&&&\\
&&&&&\b\ar@{-}[rrrrrrr]&&&&&&&(\bx_3)\\
&&&&&&&&&&&&\\
&&&(0)\ar@{-}[rrrrrrr]\ar@{-}[rruu]&&&&&&&\b\ar@{-}[rruu]&&\\
}
$$
We have marked the four Auslander bundles $E(\vx)$, $\vx\in\set{0,\bx_1,\bx_2,\bx_3}$. By Proposition~\ref{prop:cuboid:symmetry}, each rotational symmetry with angle $\pi$ through one of the three central axes sends an extension bundle of the standard cuboid to another one from the same $\LL$-orbit. Notice that the cuboid degenerates to a point for weight type $(2,2,2)$, to a line for weight type $(2,2,n)$ with $n\geq2$, respectively to a rectangle for type $(2,a,b)$ with $a,b\geq3$. In the first case (resp.\ the second case) the four Auslander bundles collapse to a single one (resp.\ to the two ends of the line segment). In the third (rectangular) case the Auslander bundles form the four vertices of the rectangle.

 \begin{remark} \label{rem:sHom_in_Abuendel}
Invoking vector bundle duality, we immediately obtain dual versions of Proposition~\ref{prop:non_zero} and its corollaries. Keeping the above notations we obtain, for instance, that $\sHom{}{X}{E(\vx)}\neq0$ if and only if $\Oo(\vx)$ is a direct summand of the injective hull of $X$.
 \end{remark}

Denote by $\rflxdual\colon\vect{\XX}\ra\vect{\XX}$ the duality $X\mapsto X\odual{}(\vdom+\vom)$, with inverse given by $Y\mapsto\odual{Y}(-(\vdom+\vom))$. Since $\odual{\Oo(\vx)}=\Oo(-\vx)$ line bundles are send to line bundles by $\rflxdual$, and therefore this induces a duality $\rflxdual\colon\svect{\XX}\ra\svect{\XX}$ (denoted by the same symbol). Moreover, $\odual{L(\vx)}=\odual{L}(-\vx)$, and thus $\rflxdual{L(\vx)}=\odual{L}(-\vx)$. Invoking Lemma~\ref{lem:extb:dual} and Corollary~\ref{cor:suspension:extb} we get: \begin{proposition}\label{prop:spiegel-duality} Let $0\leq\vx\leq\vdom$ and $E=E_L$. The duality $\rflxdual\colon\svect{\XX}\ra\svect{\XX}$ has the following properties:
  \begin{enumerate}
  \item $\rflxdual{E(\vx)}=E(\vdom-\vx)$.
  \item
    $\rflxdual{\extb{}{\vx}}=\extb{}{\vx}(\vdom-\vx)=\tau\extb{}{\vdom-\vx}[1]$.
  \end{enumerate}
\end{proposition}
Our next result reveals interesting insight into the relationship between extension bundles and Auslander bundles. A conceptional interpretation of this relationship will be given in Corollary~\ref{cor:extb:simp:proj:inj}. \begin{proposition}\label{prop:extb:abdle} Let $L$ is a line bundle, and $0\leq \vx,\,\vy\leq \vdom$. Then we obtain
\begin{eqnarray*}
\sHom{}{\extb{L}{\vx}}{E_L(\vy)[n]}&=&
\begin{cases}
k & \textrm{if } n=0 \textrm{ and } \vx=\vy,\\
0& \textrm{otherwise.}
\end{cases}\\
\sHom{}{E_L(\vx)}{\extb{L(\vy)}{\vdom-\vy}[n]}&=&
\begin{cases}
k & \textrm{if } n=0 \textrm{ and } \vx=\vy,\\
0& \textrm{otherwise.}\end{cases}\\
\sHom{}{E_L(\vx-\vom)}{\extb{L}{\vy}[1-n]}&=&
\begin{cases}
k & \textrm{if } n=0 \textrm{ and } \vx=\vy,\\
0& \textrm{otherwise.}
\end{cases}
\end{eqnarray*}
In particular, there are \define{natural morphisms} $\iota:E_L(\vx)\ra \extb{L(\vx)}{\vdom-\vx}=\tau\extb{L}{\vx}[1]$.
\end{proposition}
Each of the three formulae yields a corresponding result for the Euler form. For distinction, we use single brackets for the classes in $\Knull{\coh\XX}$ and double brackets for classes in $\Knull{\svect\XX}$. Then we obtain for instance that the value $\euler{\sclass{\extb{}{\vx}}}{\sclass{E(\vy)}}$ of Euler form for the stable category $\svect\XX$ equals $1$ for $\vx=\vy$ and $0$ otherwise (always assuming that $0\leq \vx,\vy\leq \vdom$).

\begin{proof} It will suffice to prove the first formula, since the second is then obtained by vector bundle duality, taking Lemma~\ref{lem:extb:dual} and Corollary~\ref{cor:suspension:extb} into account. The third formula follows from the second by Serre duality.

\emph{We first assume that $n$ is even}, say $n=2m$. As usual, we write $\vx=\sum_{i=1}^3\ell_i\vx_i$. In view of Corollary~\ref{cor:double_extension} we have to deal with the expression $H:=\sHom{}{\extb{L}{\vx}}{E_L(\vy+m\vc)}$. In view of Corollary~\ref{rem:sHom_in_Abuendel} $H$ is non-zero if and only if $L(\vy+m\vc)$ is a direct summand of the injective hull $\injh{\extb{L}{\vx}}$. By Theorem~\ref{thm:hulls} we thus have to check whether it is possible that (a) $\vy+m\vc=\vx$ or (b) $\vy+m\vc=(1+\ell_i)\vx_i+\vom$ holds with $i$ from $\{1,2,3\}$. Since (a) implies that $m=0$, we immediately obtain that (a) holds if and only if $m=0$ and $\vx=\vy$. Next we assume, for contradiction, that (b) holds for some $i$. By symmetry, we may even assume $i=1$ yielding $\vy=(1+\ell_1)\vx_1+\vom -m\vc$. Since $0\leq\vy\leq 2\vom+\vc$ we obtain the two inequalities $0\leq \vy=\ell_1\vx_1-\vx_2-\vx_3+(1-m)\vc$, and $(1+\ell_1)\vx_1 -m\vc \leq \vom+\vc$.  Now, for $m\geq 0$ the first inequality is violated, while for $m\leq-1$ the left hand side of the second equation is $\geq0$, implying that also the second inequality is violated.

In particular, we have shown that $\sHom{}{\extb{}{\vx}}{E(\vx)}\neq0$. It now follows from Proposition~\ref{prop:natural_map} that $\sHom{}{\extb{}{\vx}}{E(\vx)}=k$ is actually generated by the natural map $\pi$.

\emph{We next assume that $n$ is odd}, say $n=2m+1$. We are going to show that $H:=\dual{\Hom{}{\extb{}{\vx}}{E(\vy)[2m+1]}}=0$ for every integer $m$. Invoking Corollary~\ref{cor:double_extension} and Serre duality we may write $H=\sHom{}{E(\vy+m\vc)}{\extb{}{\vx}(\vom)}$. From Proposition~\ref{prop:non_zero} we know that $H\neq0$ if and only if $L(\vy+m\vc)$ is a direct summand of the injective hull $\injh{\extb{}{\vx}}$, that is, if and only if $\vy+m\vc$ equals $\vom$ or $\vx-(1+\ell_i)\vx_i$ where $i$ is one of $1,2,3$. We first show that $\vy+m\vc=\vom$ is impossible by distinguishing the cases $m\leq-1$ and $m\geq0$. Assuming $m\leq-1$ we observe that the assumption $\vom-m\vc=\vy\leq2\vom+\vc$ yields the contradiction $0\leq-m\vc\leq\vom+\vc$. We now assume that $m\geq0$. Then the assumption $\vy+m\vc=\vom$ implies the inequality
$$
0\leq\vy =\vom-m\vc=(1-m)\vc-\vx_1-\vx_2-\vx_3
$$
which is impossible for $m\geq0$.

Next, we shows that $\vy+m\vc=\vx-(1+\ell_i)\vx_i$ is impossible, where by symmetry we may assume that $i=1$. For this we distinguish the cases $m\geq0$ and $m\leq0$. Indeed, the assumption implies $$0\leq \vx-(1+\ell_1)\vx_1-m\vc=-\vx_1+\ell_2\vx_2+\ell_3\vx_3-m\vc$$ which is impossible for $m\geq0$. The assumption also implies $\vx-(1+\ell_1)\vx_1-m\vc=\vy\leq 2\vom+\vc$ and then, assuming $m\leq0$, also the contradiction $$ 0\leq\ell_2\vx_2+\ell_3\vx_3+(1-m)\vc\leq \vom+\vc, $$ thus finishing the proof.
\end{proof}

\subsection{Important triangles}
We consider the maps $v_i\colon\extb{L}{\vx}\ra\extb{L}{\vx+\vx_i}$,
where we assume $0\leq\vx<\vx+\vx_i\leq\vdom$ ($1\leq i\leq 3$),
compare Lemma~\ref{lem:induction-step-pull-back}.
\begin{proposition}\label{prop:cone-v_i}
  Let $0\leq\vx<\vx+\vx_i\leq\vdom$. Then there is a distinguished
  triangle
  \begin{equation}
    \label{eq:cone-v_i}
    \extb{L}{\vx}\stackrel{v_i}\ra\extb{L}{\vx+\vx_i}
  \ra\extb{L}{\vx-\ell_i\vx_i}((1+\ell_i)\vx_i)
  \end{equation}
 in $\svect{\XX}$,
  where $\vx=\sum_{i=1}^3\ell_i\vx_i$.
\end{proposition}
\begin{proof}
  The cone of $v_i$ is calculated from the distinguished exact
  sequence $$0\ra\extb{L}{\vx}\ra\extb{L}{\vx+\vx_i}\oplus
  \injh{\extb{L}{\vx}}\ra C\ra 0.$$ This is the sequence associated
  with the push-out diagram in $\coh\XX$
  $$
  \xymatrix@R=15pt@C=15pt{ & 0\ar @{->}[d] & 0\ar @{->}[d] & & \\
    0\ar @{->}[r]&\extb{L}{\vx}\ar @{->}[r]^-{v_i}\ar @{->}[d]_-{j_E}&
    \extb{L}{\vx+\vx_i}\ar
    @{->}[r]\ar @{->}[d]& \simp{i}\ar @{->}[r]\ar @{=}[d] & 0\\
    0\ar @{->}[r]&\injh{\extb{L}{\vx}}\ar @{->}[r]\ar
    @{->}[d]& C\ar
    @{->}[r]\ar @{->}[d]& \simp{i}\ar @{->}[r] & 0\\
    &\extb{L}{\vx}[1]\ar @{->}[d]\ar @{=}[r] & \extb{L}{\vx}[1]\ar
    @{->}[d] & & \\
    & 0 & 0 & & }
  $$
where $\simp{i}$ is a simple sheaf concentrated in the $i$-th exceptional point $x_i$ (use $\det(\simp{i})=\vx_i$). After canceling common line bundle factors from $C$ and $\injh{\extb{L}{\vx}}$ we obtain a distinguished exact sequence $$0\ra\extb{L}{\vx}\ra\extb{L}{\vx+\vx_i}\oplus \Oo(\vx)\oplus\Oo(\vom+(1+\ell_i)\vx_i)\ra\underline{C}\ra 0,$$ where $\underline{C}$ is indecomposable of rank $2$, hence exceptional. In order to show $\underline{C}\simeq\extb{L}{\vx-\ell_3\vx_3}((1+\ell_3)\vx_3)$ it suffices to show that their classes in $\Knull{\coh\XX}$ are the same, and this is easily checked.
\end{proof}
An interesting special cases is the following.
\begin{corollary} For each $i=1,2,3$ we get a distinguished triangle
$$
E_L\up{v_i}\extb{L}{\vx_i}\up{\pi_i} E_L(\vx_i)\ra.
$$
\end{corollary}

\begin{corollary}\label{cor:three-tria-subcats-coincide}
  Let $L$ be a fixed line bundle. The three triangulated subcategories of
  $\svect{\XX}$ generated by the following objects, respectively,
  coincide:
  \begin{itemize}
  \item[(a)] the extension bundles $\extb{L}{\vx}$ ($0\leq\vx\leq\vdom$);
  \item[(b)] the Auslander bundles $E_L(\vx)$ ($0\leq\vx\leq\vdom$);
  \item[(c)] the bundles $\tau\extb{L}{\vdom-\vx}[1]$ ($0\leq\vx\leq\vdom$).
  \end{itemize}
Moreover, this triangulated subcategory is closed under all grading shifts with elements from $\ZZ\vom$.
\end{corollary}
\begin{proof}
By Corollary~\ref{prop:spiegel-duality} it is sufficient to show that the first two subcategories are the same; moreover, the closedness under $\vom$-shifts also follows then.

Without loss of generality $L=\Oo$. Let $\extb{}{\vy}$ be an extension bundle with $\vy=\ell_1\vx_1+\ell_2\vx_2+\ell_3\vx_3$ ($0\leq\ell_i\leq p_i-2$). Show by induction on $\sum_{i=1}^3 \ell_i$ that $\extb{}{\vy}$ is generated by Auslander bundles $E(\vx)$, where $0\leq\vx\leq\vy$. The start $\extb{}{0}=E(0)$ is trivial. \eqref{eq:cone-v_i} gives the triangle $$\extb{L}{(\ell_1-1)\vx_1}\ra\extb{L}{\ell_1\vx_1}\ra E(\ell_1\vx_1)\ra$$ which shows the assertion for extension bundles $\extb{}{\ell_1\vx_1}$ for $0\leq\ell_1\leq p_1-2$. Next, \eqref{eq:cone-v_i} yields the triangle $$\extb{L}{(\ell_1\vx_1+(\ell_2-1)\vx_2}\ra\extb{L}{\ell_1\vx_1+\ell_2\vx_2}\ra \extb{}{\ell_1\vx_1}(\ell_2\vx_2)\ra,$$ which shows the assertion for extension bundles of the form $\extb{L}{\ell_1\vx_1+\ell_2\vx_2}$. Finally, applying~\eqref{eq:cone-v_i} again gives the triangle $$\extb{L}{(\ell_1\vx_1+\ell_2\vx_2+(\ell_3-1)\vx_3}\ra\extb{L}{\ell_1\vx_1+\ell_2\vx_2+\ell_3\vx_3}\ra \extb{}{\ell_1\vx_1+\ell_2\vx_2}(\ell_3\vx_3)\ra,$$ which concludes the proof that the category generated by the first set lies in the category generated by the second set.

The converse inclusion follows by similar considerations, by going the steps of the preceding proof backwards.
\end{proof}

\subsection{Extending the standard cuboid by shift}

In this section we prove an elementary property of the Picard group $\LL=\LL(a,b,c)$ which is instrumental in proving that---for a fixed line bundle $L$---the direct sum of all extension bundles $\extb{L}{\vx}$, $0\leq\vx\leq\vdom$,  is a tilting object in $\svect\XX$.

\begin{lemma}\label{lem:shifts-cuboid}
  Denote by $\fund=\{\vx\in\LL\mid 0\leq\vx\leq\vdom\}$.
  \begin{enumerate}
\item[(1)] For weight type $(2,2,c)$ (with $c\geq 2$) we put $U=\ZZ\vx_1+\ZZ\vx_2+\ZZ\vom$. Then $\fund+U=\LL$.
\item[(2)] For weight type $(2,b,c)$ (with $b,\,c\geq 3$) we put $U=\ZZ\vx_1+\ZZ\vom$. Then $\fund+U=\LL$.
\item[(3)] For weight type $(a,b,c)$ (with $a,\,b,\,c\geq 3$) we put $U=\ZZ\vc+\ZZ\vom$. Then $\fund+U=\LL$, provided $(a,b,c)$ is different from $(3,3,3)$.
  \end{enumerate}
\end{lemma}
\begin{proof}
We put $(a,b,c)=(p_1,p_2,p_3)$ and write $\vy\in\LL$ in normal form $\vy=\ell_1\vx_1+\ell_2\vx_2+\ell_3\vx_3+\ell\vc$ ($0\leq\ell_i\leq p_i-1$). We note that the subgroup $U$ always contains $\ZZ\vc$ and, hence, without loss of generality we can assume $\ell=0$. In the following we write $\equiv$ for the induced congruence relation modulo $\ZZ\vc$; in particular we have $\vom\equiv -\vx_1-\vx_2-\vx_3$.

\emph{Case (1).} For type $(2,2,n)$ the subgroups $\ZZ\vx_1$ and $\ZZ\vx_2$ are contained in $U$. We then only case to show that $\vy=(c-1)\vx_3$ can be shifted by elements from $U$ into $\fund$. This is achieved by observing that $(c-1)\vx_3-\vom-\vx_1-\vx_2\equiv 0$.

\emph{Case (2).} In this case $\ZZ\vx_1$ is contained in $U$. Then by symmetry we only need to check whether $\vy=\ell_2\vx_2+(c-1)\vx_3$ with $0\leq\ell_2\leq b-1$ belongs to $\fund+U$: If $\ell_2=b-1$, then $\vy-\vom-\vx_1\equiv 0$. If $0\leq\ell_2\leq b-3$, then $\vy-\vom-\vx_1\equiv (\ell_2+1)\vx_2\in\fund$. Finally, if $\ell_2=b-2$, then $\vy-2\vom-2\vx_1\equiv\vx_3\in\fund$.

\emph{Case (3).} Now assume weight type $(a,b,c)$ with $a,\,b,\,c\geq 3$. If $0\leq\ell_i\leq p_i-2$ holds for each $i$, then $\vy$ belongs to $\fund$. If, on the other hand, we have $1\leq\ell_i\leq p_i-1$ for each $i$, then $\vy+\vom$ belongs to $\fund$. Thus, by symmetry, it remains to deal with the following two cases:
\centerline{(a) $\ell_1=0=\ell_2$ and $\ell_3=c-1$, (b) $\ell_1=0$ and $\ell_2,\ell_3\neq 0$.}
In case (a) we have $\vy-\vom\equiv\vx_1+\vx_2\in\fund$. By symmetry case (b) reduces to show $\ell_3=c-1$ and, depending on the value of $\ell_2$,  we have to consider three subcases :
\centerline{(i) $0<\ell_2\leq b-3$, (ii) $\ell_2=b-1$, and (iii) $\ell_2=b-2$.}
In case (i) $\vy-\vom\equiv\vx_1+(\ell_2+1)\vx_2\in\fund$; in case (ii) $\vy-\vom\equiv\vx_1\in\fund$. Note that so far we made no restriction to the weights $a,\,b,\,c\geq 3$. In case (iii) we have $\vy-2\vom\equiv 2\vx_1+\vx_3$, and this element lies in $\fund$ only when $a\geq 4$. This means that (3) is already proved if $a,\,b,\,c$ are all greater or equal $4$. We further observe that in case $(3,3,3)$ the element $\vy=2\vx_1+\vx_2$ does not lie in $U.\fund$. There remain the weight types $(3,b,c)$ where $b\geq 3$ and $c\geq 4$, where we have to check property (iii). Since for this we can not longer use symmetry, we end up with checking the following six cases: $\vy=\ell_1\vx_1+\ell_2\vx_2+\ell_3\vx_3$ where $(\ell_1,\ell_2,\ell_3)$ equals one of (a)  $(0,b-1,c-2)$, (b)  $(0,b-2,c-1)$, (c)  $(2,0,c-2)$, (d)  $(1,0,c-1)$, (e)  $(2,b-2,0)$, or (f) $(1,b-1,0)$.
In each of the six cases it is easily checked that, up to congruence, $\vc-n\vom$ belongs to $\fund$ for some integer $n$, which concludes the proof.
\end{proof}
Let $\Cc$ be a subclass of a triangulated category $\Tt$. Then the right \define{perpendicular category} $\rperp{\Cc}$ of $\Cc$ consists of all objects $X$ of $\Tt$ such that $\Hom{}{C}{X[n]}=0$ for all integers $n$ and all objects $C$ from $\Cc$.
\begin{corollary}\label{cor:zero-perpendicular}
Let $(a,b,c)$ be a weight triple with $a,\,b,\,c\geq 2$. Then the triangulated subcategory $\mathcal{C}$ of $\svect{\XX}$ from Corollary~\ref{cor:three-tria-subcats-coincide} contains all Auslander bundles and thus $\rperp{\mathcal{C}}=0$.
\end{corollary}
\begin{proof}
We note that for weight type $(3,3,3)$ the claim follows from the fact that $\svect{\XX}\simeq\Der{\coh{\XX}}$, see~Theorem~\ref{thm:tubular}. For the rest of the proof we may thus assume  $(a,b,c)\neq (3,3,3)$.

(1) By the preceding lemma each Auslander bundle $E(\vy)$ ($\vy\in\LL$) lies in $\mathcal{C}$: For this we use first that, as triangulated subcategory, $\mathcal{C}$  is closed under suspension $[1]$, and thus in particular under grading shift by $\pm\vc$ (Corollary~\ref{cor:double_extension}). By Corollary~\ref{cor:three-tria-subcats-coincide} the subcategory $\mathcal{C}$ is further closed under shifts by $\pm\vom$. Moreover, we use that for type $(2,2,c)$ the line bundle shifts by $\vx_1$ and $\vx_2$ coincide on $\svect{\XX}$ and, finally, we use that in case $(2,a,b)$ suspension $[1]$ is given by grading shift with $\vx_1$ (Proposition~\ref{prop:2ab:suspension}).

(2) Let $X\in\vect{\XX}$ without line bundle summands. By Corollary~\ref{cor:Auslander-bundles-generate} there is an Auslander bundle $E(\vx)$ such that $\sHom{}{E(\vx)}{X}\neq 0$. This shows $\rperp{\mathcal{C}}=0$.
\end{proof}

\section{Impact of the Euler characteristic} \label{sect:impact:euler}
As is well known the (signature of the) Euler characteristic has a decisive influence on the complexity of the classification problem for $\coh\XX$. The same phenomenon is true for the categories $\svect\XX$. \subsection{Positive Euler characteristic} \label{sect:euler:positive} In this section we deal with the stable categories $\svect\XX$ of vector bundles for weighted projective lines $\XX$ of strictly positive Euler characteristic. If the number $t_\XX$ of non-trivial weights is at most two, then each indecomposable vector bundle is a line bundle, and then $\svect\XX=0$. Hence in this section, we only have to deal with the \emph{weight triples} $(2,2,n)$, with $n\geq2$, $(2,3,3)$, $(2,3,4)$ and $(2,3,5)$.

\subsubsection{Tilting object defined by slope range}
We recall \cite[Proposition 5.5]{Geigle:Lenzing:1987} that for $\eulerchar\XX>0$ each indecomposable vector bundle $E$ is stable and has trivial endomorphism ring $\End{}{E}=k$ and $\dual{\Ext1{}{E}{E}}=\Hom{}{E}{E(\vom)}=0$ since, by stability, for indecomposable vector bundles $E$ and $F$ the existence of a non-zero morphism $u:E\ra F$ enforces $\mu E < \mu F$. Hence indecomposable bundles are exceptional. The first two statements of the following proposition are due to H\"{u}bner, see \cite{Huebner:1989} or \cite[Proposition 6.5]{Lenzing:Reiten:2006}. For the convenience of the reader we include the proof.

\begin{proposition} \label{prop:tilting_domestic}
We assume that $\XX$, given by the weight triple $(a,b,c)$, has $\eulerchar\XX>0$, yielding the weight types $(2,2,c)$, $(2,3,3)$, $(2,3,4)$ and $(2,3,5)$. Let $\De=[a,b,c]$ be the attached Dynkin diagram and $\tilde\De$ its extended Dynkin diagram. Then the following holds
\begin{enumerate}
\item[(1)]The indecomposable vector bundles $F$ with slope in the range $0\leq \slope{F}<-\de(\vom)$ form a tilting bundle $T$ in $\coh\XX$ whose endomorphism ring is the path algebra $k\vec{\tilde\De}$ of an extended Dynkin quiver $\vec{\tilde\De}$ with associated Dynkin diagram $\De=[a,b,c]$.
\item[(2)]The Auslander-Reiten quiver $\Gamma(\vect\XX)$ of $\vect\XX$ consists of a single standard component having the form $\ZZ\tilde\De$. Moreover, the category $\ind{\vect\XX}$ of indecomposable vector bundles on $\XX$ is equivalent to the mesh category of $\Gamma(\vect\XX)$.
\item[(3)] The indecomposable vector bundles $F$ of rank $\geq2$ with slope in the range $0\leq\slope{F}<-\de(\vom)$ form a tilting object $T'$ in the triangulated category $\svect\XX$
\item[(4)] The category $\svect\XX$ is triangle-equivalent to the derived category $\Der{\mmod{k\vec{\De}}}$ of the path algebra $k\vec\De$ for any quiver $\vec\De$ having $\De$ as underlying graph. Moreover, $\svect\XX$ is equivalent to the mesh category $\ZZ\De$.
  \end{enumerate}
\end{proposition}
\begin{proof}
\underline{ad (1)}: We form a representative system $\Tt$ of all indecomposable vector bundles $F$ in the slope range $0\leq\slope{F}< -\de(\vom)$. Assume that $E$ and $F$ belong to $\Tt$, then $\slope{E(\vom)}-\slope{F}<0$, hence $\dual{\Ext1{}{E}{F}}=\Hom{}{F}{E(\vom)}=0$ in view of stability. In particular, the classes $([E])_{E\in\Tt}$ form a linearly independent system in the Grothendieck group $\Knull{\coh\XX}=\ZZ^n$ where $n=\sum_{i=1}^3p_i -1$. We have shown that $\Tt$ is finite and $T=\bigoplus_{E\in\Tt}E$ has no self-extensions. We put $\Tt_n=\Tt(-n\vom)$ for each integer $n$.

By definition of $\Tt$ each indecomposable vector bundle $X$ has the form $X=E(-n\vom)$ with uniquely defined data $E\in \Tt$ and $n\in\ZZ$. Thus, on the level of objects, $\ind\vect\XX$ is the disjoint union of all $\Tt_n$ with $n\in\ZZ$. Let $\genab{\Tt}$ denote the \define{thick subcategory} generated by $\Tt$, that is, the smallest full subcategory of $\coh\XX$ closed under direct summands and under third terms of short exact sequences.  We are going to show that $\genab{\Tt}=\coh\XX$, implying that $T$ is a tilting bundle in $\coh\XX$. We assume for contradiction that there exists an integer $n\neq 0$ such that $\Tt_n$ does not belong to $\genab{\Tt}$. By symmetry, it suffices to deal with the case $n>0$, and we choose $X$ from $\Tt_n$ of minimal slope. We have an almost-split sequence $0\ra X(\vom) \ra Y \ra X\ra 0$ in $\coh\XX$, and write $Y$ as a direct sum of indecomposables $Y_j$, $j\in J$. Since $X(\vom)$ and all the $Y_j$ belong to $\Tt_{n-1}$ or $\Tt_n$ and further have a slope strictly less than $\slope{F}$, it follows that $X(\vom)$ and $Y$ belong to $\genab{\Tt}$, implying that also $X$ belongs to $\genab\Tt$, contradicting our assumption. We have shown that the vector bundle $T=\bigoplus_{F\in\Tt}F$ is tilting in $\coh\XX$.

Next, we are going to show that $H=\End{}{T}$ is hereditary. As another consequence of stability we have $\Hom{}{\Tt_{n+k}}{\Tt_n}=0$ for all $k>0$. Since $\coh\XX$ is hereditary, we arrive at the following schematic picture of $\mod{H}$ as a full subcategory of $\Der{\coh\XX}$, where $\Tt'_{n}$ stands for $\Tt_{-n}[1]$.
 \newcommand{\links}
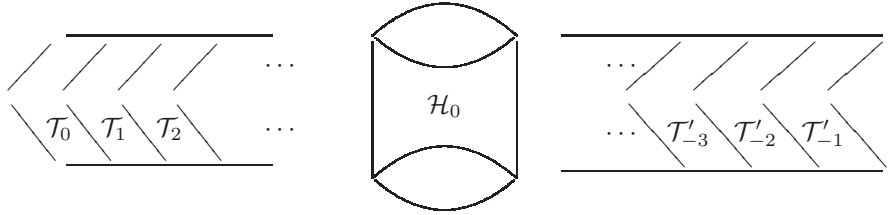
{\xymatrix@!C@C=3pt@R=3pt{
& \dt\ar@{-}[rrrr]&\dt&\dt&\dt&\\
&              &   &   &&\cdots&\\
\dt\ar@{-}[ruu]\ar@{-}[rdd]&\dt\ar@{-}[ruu]\ar@{-}[rdd]&\dt\ar@{-}[ruu]\ar@{-}[rdd]&\dt\ar@{-}[ruu]\ar@{-}[rdd]\\
&  \Tt_0      &\Tt_1&\Tt_2&   &\cdots\\
&\dt\ar@{-}[rrrr]&\dt&\dt&\dt&\\}}
 \newcommand{\mitte}{\xymatrix@!@1@=12pt{
\dt\ar@/_1pc/@{-}[rr]\ar@/^1pc/@{-}[rr]\ar@{-}[dd]&        &\dt\ar@{-}[dd]\\
   & \Hh_0 &   \\
\dt\ar@/_1pc/@{-}[rr]\ar@/^1pc/@{-}[rr]&        &\dt\\
}}
 \newcommand{\rechts}{\xymatrix@!C@C=3pt@R=3pt{
&& \dt&\dt&\dt&\dt\ar@{-}[lllll]&&\\
&\cdots&              &   &   &&&&\\
&\dt\ar@{-}[ruu]\ar@{-}[rdd]&\dt\ar@{-}[ruu]\ar@{-}[rdd]&\dt\ar@{-}[ruu]\ar@{-}[rdd]&\dt\ar@{-}[ruu]\ar@{-}[rdd]&\\
&\cdots&  \Tt'_{-3}      &\Tt'_{-2}&\Tt'_{-1}&\\
&&\dt&\dt&\dt&\dt\ar@{-}[lllll]&&\\}}
\begin{center}
$$
\def\dt{}
\begin{array}{ccc}
\links&\mitte&\rechts\\
&&\\
\end{array}
$$
\end{center}
Here $\Tt_0$ gets identified with the system of indecomposable projectives in $\mod{H}$, and morphisms exist only from left to right. This shape of the category $\mod{H}$ implies that submodules of (indecomposable) projectives are again projective, hence $H$ is a finite dimensional hereditary algebra, which additionally is connected since $\coh\XX$ is connected. Let $\Ga$ denote the underlying graph of the quiver of $H$ or, equivalently, the quiver of the category $\Tt$. Then the rank function of $\coh\XX$ induces an additive function on $\Ga$, forcing $\Ga$ to be extended Dynkin. It is then not difficult to check that $\Ga$ is the extended Dynkin graph belonging to the Dynkin graph $[a,b,c]$.

\underline{ad (2)}: This is straightforward from the properties established in (1) and its proof. Compare the arguments in \cite{Happel:1988} for the preprojective component of a tame hereditary algebra.

\underline{ad (3)}: We obtain $\Tt'$ from $\Tt$ by removing the line bundles. We first show that $\Tt'$ has no self-extensions in $\svect\XX$. For $E$, $F$ from $\Tt'$ we have to show that $\sHom{}{E}{F[n]}=\dual{\sHom{}{F[n-1]}{E(\vom)}}$ equals zero for each integer $n\neq 0$. Now for $n<0$ the left hand side is zero, while for $n\geq2$ the term on the right hand side is zero. We are left to deal with the case $n=1$, where already $\Hom{}{F}{E(\vom)}=\dual{\Ext1{}{E}{F}}$ is zero. Next, we are going to show that $\Tt'$ generates $\svect\XX$. Observing that almost-split sequences with end terms of rank at least two yield Auslander-Reiten triangles in $\svect\XX$, we can copy the proof of (1) to achieve this. This completes the proof that $T'$ is tilting in $\svect\XX$.

\underline{ad (4)}: The first assertions follow directly from (3). For the last claim compare~\cite{Happel:1988}.
 \end{proof}
We illustrate the situation for weight type $(2,3,4)$.
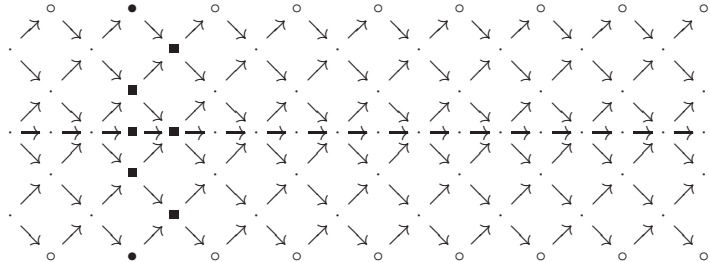
\newcommand{\hide}[1]{}
\scriptsize
$$
\def\c{\circ}
\def\b{\bullet}
\def\pt{\cdot}
\xymatrix@R6pt@C6pt{
& \c\ar[rd]       &                       &\b\ar[rd]       &                       &\c\ar[rd]       &                       &\c\ar[rd]       &               & \c\ar[rd]       &                       &\c\ar[rd]       &                       &\c\ar[rd]       &                       &\c\ar[rd]       &               &\c&\\
\pt\ar[ru]\ar[rd]       &                &\pt\ar[ru]\ar[rd]       &                &\sq\ar[ru]\ar[rd]       &                &\pt\ar[ru]\ar[rd]       &                &\pt\ar[ru]\ar[rd]&      &\pt\ar[ru]\ar[rd]       &                &\pt\ar[ru]\ar[rd]       &                &\pt\ar[ru]\ar[rd]       &                &\pt\ar[ru]\ar[rd]&&& &\\
&\pt\ar[ru]\ar[rd]&                       &\sq\ar[ru]\ar[rd]&                       &\pt\ar[ru]\ar[rd]&                       &\pt\ar[ru]\ar[rd]&                       &\pt\ar[ru]\ar[rd]&                       &\pt\ar[ru]\ar[rd]&                       &\pt\ar[ru]\ar[rd]&                       &\pt\ar[ru]\ar[rd]&                       &\pt\\
\pt\ar[ru]\ar[rd]\ar[r]&\pt\ar[r]        &\pt\ar[ru]\ar[rd]\ar[r]&\sq\ar[r]         &\sq\ar[ru]\ar[rd]\ar[r]&\pt\ar[r]        &\pt\ar[ru]\ar[rd]\ar[r]&\pt\ar[r] &\pt\ar[ru]\ar[rd]\ar[r]&\pt\ar[r]        &\pt\ar[ru]\ar[rd]\ar[r]&\pt\ar[r]         &\pt\ar[ru]\ar[rd]\ar[r]&\pt\ar[r]        &\pt\ar[ru]\ar[rd]\ar[r]&\pt\ar[r] &\pt\ar[ru]\ar[rd]\ar[r]&\pt&\\
& \pt\ar[ru]\ar[rd]&                       &\sq\ar[ru]\ar[rd]&                       &\pt\ar[ru]\ar[rd]&                       &\pt\ar[ru]\ar[rd]&                       & \pt\ar[ru]\ar[rd]&                       &\pt\ar[ru]\ar[rd]&                       &\pt\ar[ru]\ar[rd]&                       &\pt\ar[ru]\ar[rd]&                       &\pt&\\
\pt\ar[ru]\ar[rd]       &                &\pt\ar[ru]\ar[rd]       &                &\sq\ar[ru]\ar[rd]       &                &\pt\ar[ru]\ar[rd]       &                &\pt\ar[ru]\ar[rd]&                &\pt\ar[ru]\ar[rd]       &                &\pt\ar[ru]\ar[rd]       &                &\pt\ar[ru]\ar[rd]       &                &\pt\ar[ru]\ar[rd]       &            &\\
&\c\ar[ru]       &                       &\b\ar[ru]       &                       &\c\ar[ru]       &                       &\c\ar[ru]       &                       &\c\ar[ru]       &                       &\c\ar[ru]       &                       &\c\ar[ru]       &                       &\c\ar[ru]       &                       &\c& &\\}
$$
\normalsize
The fat symbols $\sq$ and $\bullet$ mark the indecomposable summands of the tilting bundle $T$ in $\vect\XX$. By removing the line bundles, represented by the symbols $\circ$ and $\bullet$, we obtain the category $\svect\XX$. Removing from $T$ the two line bundles, we obtain the tilting object $T'$ for $\svect\XX$.

\subsubsection{The tilting rectangle, domestic case}
We fix a line bundle $L$. As we are going to show in Theorem~\ref{thm:std_tilting_obj}, for any weight triple the extension bundles $\extb{L}{\vx}$, $0\leq \vx\leq\vdom$, form a tilting object $T_\cub$ in $\svect\XX$, called the \define{tilting rectangle}. For the domestic case $\eulerchar\XX>0$ we now present typical examples illustrating the position of $T_\cub$ in the AR-component of $\vect\XX$, where the proof of the tilting property in $\svect\XX$ is a straightforward exercise. In the illustrations, we will mark the positions of the extension bundles $\extb{L}{j\vx_3}$, $0\leq j\leq p-2$, by the symbol $\sq$, and the position of the involved line bundles $L(\vom)$ and $L(j\vx_3)$, $0\leq j\leq p-2$ by the symbol $\bullet$, the remaining line bundles are marked by the symbol $\circ$.

For weight type $(2,2,p)$ the tilting rectangle degenerates to (the path algebra of) the equioriented quiver $\vAA_{p-1}$ of type $\AA_{p-1}$, see Figure~\ref{fig:226}.  We also display the situation for weight types $(2,3,3)$, $(2,3,4)$ and $(2,3,5)$ in Figure~\ref{fig:233} and Figure~\ref{fig:235} where the tilting rectangle, and hence its endomorphism ring, has indeed rectangular shape.

We observe that in these examples the extension bundle $\extb{L}{\vdom}$ is always an Auslander bundle. This is, indeed, a general fact for weight type $(2,a,b)$, see Proposition~\ref{prop:2ab}. By contrast, in the preceding examples the extension bundle $\extb{L}{\vx_3}$ is not an Auslander bundle, the only exception being weight type $(2,3,3)$ where all indecomposable vector bundles of rank two are Auslander bundles.

\newcommand{\zsz}{{\scriptscriptstyle\xymatrix@-1pc @!R=4.5pt @!C=4.5pt{
&&&&&&&&&&&\\
&&&\ci\ar[rd]            &&\ci\ar[rd]&&\ci\ar[rd]&&\ci\ar[rd]            &&\ci\\
\sz{4\vx_3}&&\dt\ar[r]\ar[ru]\ar[rd]&\ci\ar[r]      &\dt\ar[ru]\ar[r]\ar[rd]&\ci\ar[r]&\sq\ar[ru]\ar[r]\ar[rd]&\ci\ar[r]&\dt\ar[ru]\ar[r]\ar[rd]&      \ci\ar[r]&\dt\ar[ru]\ar[rd]\ar[r]&\ci\\
\sz{3\vx_3}&&&\dt\ar[ru]\ar[rd]      &&\sq\ar[ru]\ar[rd]&&\dt\ar[ru]\ar[rd]&&      \dt\ar[ru]\ar[rd]&&\dt\\
\sz{2\vx_3}&&\dt\ar[ru]\ar[rd]&&\sq\ar[ru]\ar[rd]&&\dt\ar[ru]\ar[rd]&&\dt\ar[ru]\ar[rd]&&\dt\ar[ru]\ar[rd]&\\
\sz{\vx_3}&&&\sq\ar[ru]\ar[rd]      &&\dt\ar[ru]\ar[rd]&&\dt\ar[ru]\ar[rd]&&      \dt\ar[ru]\ar[rd]&&\dt\\
\sz{0}&&\sq\ar[r]\ar[ru]\ar[rd]&\ci\ar[r]      &\dt\ar[ru]\ar[r]\ar[rd]&\bu\ar[r]&\dt\ar[ru]\ar[r]\ar[rd]&\ci\ar[r]&\dt\ar[ru]\ar[r]\ar[rd]&      \bu\ar[r]&\dt\ar[ru]\ar[rd]\ar[r]&\ci\\
&\bu\ar[ru]&&\bu\ar[ru]      &&\ci\ar[ru]&&\bu\ar[ru]&&      \ci\ar[ru]&&\bu&\\
 &\vom&&0      &&\vx_3&&2\vx_3&&3\vx_3&&4\vx_3&\\
}}}
\begin{figure}[H]
\scriptsize
$$
\begin{array}{c}
\zsz\\
\end{array}
$$
\caption{Standard tilting object for $\svect\XX$, type $(2,2,6)$}\label{fig:226}
\end{figure}
\newcommand{\zdd}{{\scriptscriptstyle{\xymatrix@-1.5pc @!R=3pt @!C=3pt{
&             & & & &    & &\sz{\vx_3} & & \vx_2   & & & &             & & & &             & & & &      \\
&\ci\ar[ddrr] & & & &\ci\ar[ddrr] & & & & \bu\ar[ddrr] & & & &\ci\ar[ddrr] & & & &\ci\ar[ddrr] & & & & \ci  \\
  & & & & & & & & & & & & & &&&&&&&&\\
& \ci\ar[drr] & & \dt\ar[ddrr]\ar[uurr] & & \ci\ar[drr] & &\sq\ar@/^/[rrrrdddd]\ar[ddrr]\ar[uurr] & & \bu\ar[drr] & &\dt\ar[ddrr]\ar[uurr]& & \ci\ar[drr] & & \dt\ar[ddrr]\ar[uurr] & & \ci\ar[drr] & &\dt\ar[ddrr]\ar[uurr] & & \ci \\
& & & \dt\ar[drr]\ar[urr] & & & & \sq\ar@/_/[rrrrddd]\ar[drr]\ar[urr] & & & &\dt\ar[drr]\ar[urr] & & && \dt\ar[drr]\ar[urr] & & & & \dt\ar[drr]\ar[urr] & & \\
& \dt\ar[uurr]\ar[ddrr]\ar[urr] & & & & \dt\ar[uurr]\ar[ddrr]\ar[urr] && & & \dt\ar[uurr]\ar[ddrr]\ar[urr] && & & \dt\ar[uurr]\ar[ddrr]\ar[urr] & & & & \dt\ar[uurr]\ar[ddrr]\ar[urr] && & & \dt \\
& & & & & & & & & & & & & & &&&&&&& \\
& & & \sq\ar@/^/[rrrruuuu]\ar@/_/[rrrruuu]\ar[ddrr]\ar[uurr] & & & & \dt\ar[ddrr]\ar[uurr] & & & &\sq\ar[ddrr]\ar[uurr] & & & & \dt\ar[ddrr]\ar[uurr] & & & & \dt\ar[ddrr]\ar[uurr] & & \\
& & & & & & & & & & & & & & &&&&&& \\
& \bu\ar[uurr] & & & & \bu\ar[uurr] & & & & \ci\ar[uurr] & & & &\ci\ar[uurr] & & & & \ci\ar[uurr] & & & & \bu\\
& \vom    & &\sz{0}& & 0            & &\sz{\vx_2}& &\vx_3 &&\sz{\vx_2+\vx_3}& &             & & & &         & & & & \vx_2+\vx_3 \\
}}}}
\begin{figure}[H]
$$
\scriptsize
\zdd
$$
  \caption{Standard tilting rectangle for $\svect\XX$, type $(2,3,3)$}\label{fig:233}
\end{figure}
\newcommand{\dvz}{\scriptscriptstyle{\xymatrix@-1pc @!R=3pt @!C=3pt{
&&                  &&                 &&    &\sz{\vx_2}&          &&\vx_3      &&&\sz{\vx_2+2\vx_3}&   &&          &&\vx_2+\vx_3&\\
&&\ci\ar[rd]        &&\ci\ar[rd]       &&\ci\ar[rd]    &&\ci\ar[rd]&& \bu\ar[rd]&&\ci\ar[rd]&&\ci\ar[rd]&&\ci\ar[rd]&&\bu&\\
&&&\dt\ar[ru]\ar[rd]&&\dt\ar[ru]\ar[rd]&&\sq\ar[ru]\ar[rd]\ar@<1ex>@/_/[rrrdd]&&\dt\ar[ru]\ar[rd]&&\dt\ar[ru]\ar[rd]&&
\sq\ar[ru]\ar[rd]&&\dt\ar[ru]\ar[rd]&&\dt\ar[ru]\ar[rd]&\\
&&\dt\ar[ru]\ar[rd]&&\dt\ar[ru]\ar[rd]&&\dt\ar[ru]\ar[rd]&&\dt\ar[ru]\ar[rd]&                       &\dt\ar[ru]\ar[rd]&&\dt\ar[ru]\ar[rd]&&\dt\ar[ru]\ar[rd]&&\dt\ar[ru]\ar[rd]&&\dt&\\
&&\dt\ar[r]&\dt\ar[ru]\ar[rd]\ar[r]&\dt\ar[r]&\dt\ar[ru]\ar[rd]\ar[r]&\sq\ar[r]\ar@<-1ex>@/^/[rrrdd]\ar[r]\ar@<-0.5ex>@/_/[rrrr] &\dt\ar[ru]\ar[rd]\ar[r]&\dt\ar[r] &\dt\ar[ru]\ar[rd]\ar[r]&\sq\ar[r]\ar@<1ex>@/_/[rrruu]&\dt\ar[ru]\ar[rd]\ar[r]&\dt\ar[r]&\dt\ar[ru]\ar[rd]\ar[r]&\dt\ar[r]&
\dt\ar[ru]\ar[rd]\ar[r]&\dt\ar[r]&\dt\ar[ru]\ar[rd]\ar[r]\dt&\\
&& \dt\ar[ru]\ar[rd]&&\dt\ar[ru]\ar[rd]&&\dt\ar[ru]\ar[rd]& &\dt\ar[ru]\ar[rd]&& \dt\ar[ru]\ar[rd]&&\dt\ar[ru]\ar[rd]&&\dt\ar[ru]\ar[rd]&&\dt\ar[ru]\ar[rd]&&\dt&\\
&&&\sq\ar[ru]\ar[rd]\ar@<-1ex>@/^/[uurrr]\ar@/^/[uuuurrrr]&&\dt\ar[ru]\ar[rd]&&\dt\ar[ru]\ar[rd]&                &\sq\ar@/_/[rrrruuuu]\ar[ru]\ar[rd]&&\dt\ar[ru]\ar[rd]&&\dt\ar[ru]\ar[rd]&&\dt\ar[ru]\ar[rd]&&\dt\ar[ru]\ar[rd]& \\
&&\bu\ar[ru]&&\bu\ar[ru]&&\ci\ar[ru]&&\ci\ar[ru]&&\ci\ar[ru]&&\bu\ar[ru]&&\ci\ar[ru]&&\bu\ar[ru]&&\ci&\\
&&\vom      &\sz{0}&0         &&\sz{\vx_3}&&          &\sz{2\vx_3}&&&\vx_2     &&          &&2\vx_3    &&   &\\
}}}

\newcommand{\zdf}{\scriptscriptstyle{\xymatrix@-1pc @!R=6pt @!C=6pt{
&          &&          &&          &&\sz{\vx_3}                    &&          &&          &&\sz{2\vx_3}   &&&&\sz{\vx_2+\vx_3}&&&&&&\sz{\vx_2+2\vx_3}&&&&&&&\\
&\dt\ar[rd]&&\dt\ar[rd]&&\dt\ar[rd]&&\sq\ar[rd]\ar@/_2.8pc/[rrrrrr]\ar[rd]\ar@/_4.0pc/[rrrrrrrrrr]&&\dt\ar[rd]&&\dt\ar[rd]&&\sq\ar[rd]\ar@/^/[rrrrrrdddddd]\ar[rd]\ar@/_4.0pc/[rrrrrrrrrr]&&\dt\ar[rd]&
&\sq\ar[rd]\ar@/_2.8pc/[rrrrrr]& & \dt\ar[rd] & & \dt\ar[rd] & & \sq\ar[rd]\ar@/^1pc/[rrrrrrdddddd] & & \dt\ar[rd] & &
\dt\ar[rd] & & \dt\ar[rd]       &   \\
\dt\ar[ru]\ar[rd]&&\dt\ar[ru]\ar[rd]&&\dt\ar[ru]\ar[rd]&&\dt\ar[ru]\ar[rd]&&\dt\ar[ru]\ar[rd]&&\dt\ar[ru]\ar[rd]&&
\dt\ar[ru]\ar[rd]&&\dt\ar[ru]\ar[rd]&&\dt\ar[ru]\ar[rd]&&\dt\ar[ru]\ar[rd]&&\dt\ar[ru]\ar[rd]&&\dt\ar[ru]\ar[rd]&&
\dt\ar[ru]\ar[rd]&&\dt\ar[ru]\ar[rd]&&\dt\ar[ru]\ar[rd]&&\dt
\\
 \dt\ar[r] &\dt\ar[ru]\ar[rd]\ar[r]  &\dt\ar[r]
&\dt\ar[ru]\ar[rd]\ar[r] &\dt\ar[r]
&\dt\ar[ru]\ar[rd]\ar[r]&\dt\ar[r] &\dt\ar[ru]\ar[rd]\ar[r]&\dt\ar[r]
&\dt\ar[ru]\ar[rd]\ar[r] &\dt\ar[r]         &\dt\ar[ru]\ar[rd]\ar[r]
&\dt\ar[r]         &\dt\ar[ru]\ar[rd]\ar[r] &\dt\ar[r]
&\dt\ar[ru]\ar[rd]\ar[r] &\dt\ar[r]         &\dt\ar[ru]\ar[rd]\ar[r]
&\dt\ar[r]         &\dt\ar[ru]\ar[rd]\ar[r] &\dt\ar[r]
&\dt\ar[ru]\ar[rd]\ar[r] &\dt\ar[r]         &\dt\ar[ru]\ar[rd]\ar[r]
&\dt\ar[r]         &\dt\ar[ru]\ar[rd]\ar[r] &\dt\ar[r]
&\dt\ar[ru]\ar[rd]\ar[r] &\dt\ar[r]         &\dt\ar[ru]\ar[rd]\ar[r]
&\dt \\
 \dt\ar[ru]\ar[rd]       &                &\dt\ar[ru]\ar[rd]       &
&\dt\ar[ru]\ar[rd]       &                &\dt\ar[ru]\ar[rd]       &
&\dt\ar[ru]\ar[rd] & &\dt\ar[ru]\ar[rd] & &\dt\ar[ru]\ar[rd] &
&\dt\ar[ru]\ar[rd] & &\dt\ar[ru]\ar[rd] & &\dt\ar[ru]\ar[rd] &
&\dt\ar[ru]\ar[rd] & &\dt\ar[ru]\ar[rd] & &\dt\ar[ru]\ar[rd] &
&\dt\ar[ru]\ar[rd] & &\dt\ar[ru]\ar[rd]       &
&\dt \\
  & \dt\ar[ru]\ar[rd]&                       &\dt\ar[ru]\ar[rd]&
 &\dt\ar[ru]\ar[rd]&                       &\dt\ar[ru]\ar[rd]&
 & \dt\ar[ru]\ar[rd] & & \dt\ar[ru]\ar[rd] & & \dt\ar[ru]\ar[rd] & &
 \dt\ar[ru]\ar[rd] & & \dt\ar[ru]\ar[rd] & & \dt\ar[ru]\ar[rd] & &
 \dt\ar[ru]\ar[rd] & & \dt\ar[ru]\ar[rd] & & \dt\ar[ru]\ar[rd] & &
 \dt\ar[ru]\ar[rd] & & \dt\ar[ru]\ar[rd] & \\
\dt\ar[ru]\ar[rd]       &                &\dt\ar[ru]\ar[rd]       &
&\dt\ar[ru]\ar[rd]       &                &\dt\ar[ru]\ar[rd]       &
&\dt\ar[ru]\ar[rd] & & \dt\ar[ru]\ar[rd] & & \dt\ar[ru]\ar[rd] & &
\dt\ar[ru]\ar[rd] & & \dt\ar[ru]\ar[rd] & & \dt\ar[ru]\ar[rd] & &
\dt\ar[ru]\ar[rd] & & \dt\ar[ru]\ar[rd] & & \dt\ar[ru]\ar[rd] & &
\dt\ar[ru]\ar[rd] & & \dt\ar[ru]\ar[rd]       &
&\dt \\
&\sq\ar[ru]\ar@/^1pc/[rrrrrruuuuuu]\ar[rd]\ar@/^5.2pc/[rrrrrrrrrr]&&\dt\ar[ru]\ar[rd]&&\dt\ar[ru]\ar[rd]&&\dt\ar[ru]
\ar[rd]&&\dt\ar[ru]\ar[rd] &&\sq\ar@/^/[rrrrrruuuuuu]\ar[ru]\ar[rd] &&\dt\ar[ru]\ar[rd] &&\dt\ar[ru]\ar[rd] &&\dt\ar[ru]\ar[rd] & & \sq\ar@/^5.2pc/[rrrrrrrrrr]\ar[ru]\ar[rd] & &\dt\ar[ru]\ar[rd] & & \dt\ar[ru]\ar[rd] & &\dt\ar[ru]\ar[rd] & & \dt\ar[ru]\ar[rd] & & \sq\ar[ru]\ar[rd]& \\
\bu\ar[ru]& &\bu\ar[ru]&&\ci\ar[ru]&&\ci\ar[ru]&&\ci\ar[ru]& & \ci\ar[ru] & & \ci\ar[ru] & &\bu\ar[ru] & & \ci\ar[ru] & &\ci\ar[ru] & &\ci\ar[ru] & & \bu\ar[ru] & & \ci\ar[ru] & &\bu\ar[ru] & & \ci\ar[ru]&&\ci \\
\vom      &\sz{0}&0     & & &&&&&&&\sz{\vx_2}&&&\vx_3&&&&&\sz{3\vx_3}&&&\vx_2&&&&2\vx_3&&&\sz{\vx_2+3\vx_3}&
}}}
 \begin{landscape}
 \begin{figure}[ht]
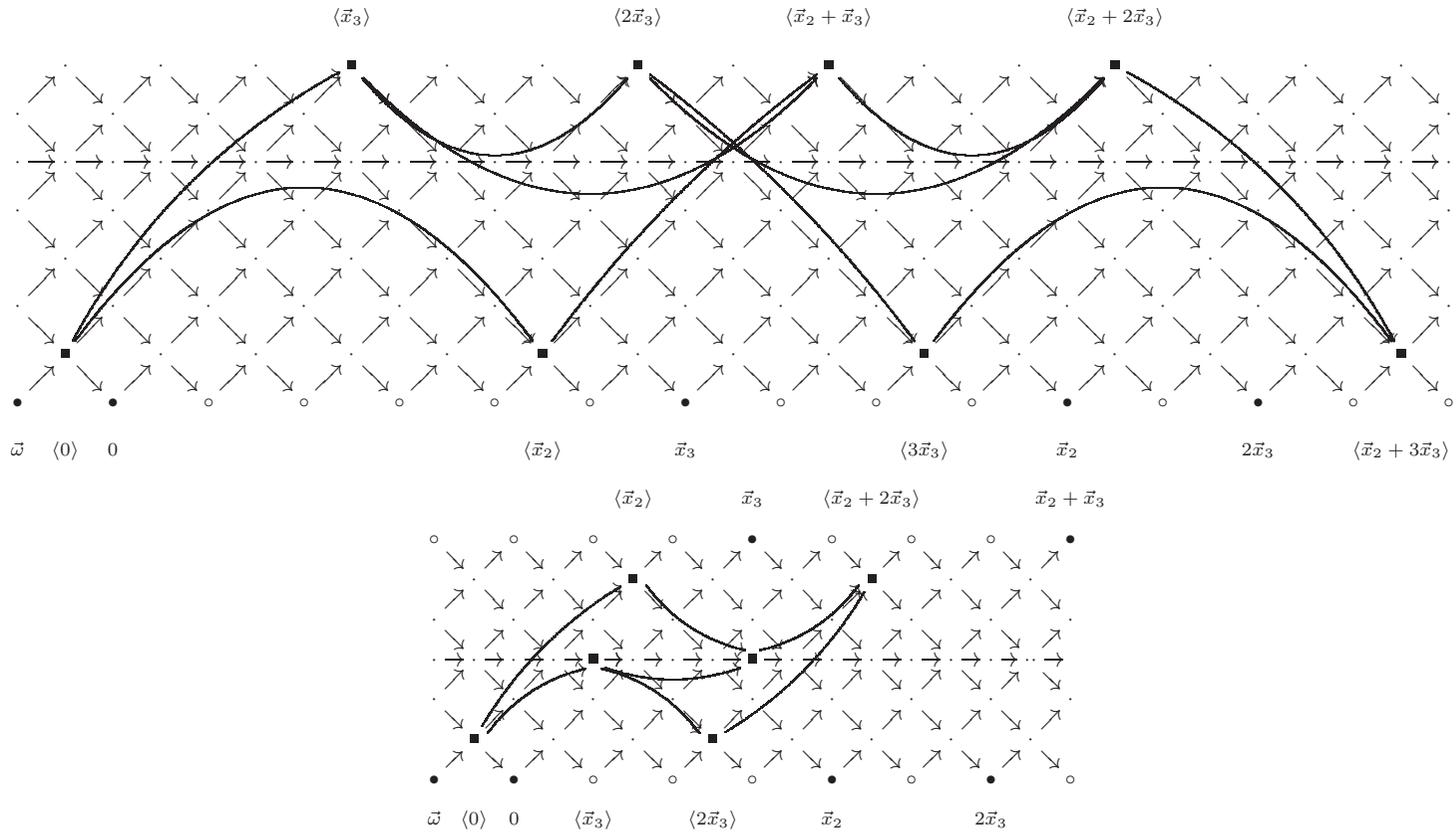

\scriptsize
 $$
 \zdf
 $$
 $$
\dvz
$$ \normalsize
 \caption{Standard tilting rectangles for $\svect\XX$, weight types $(2,3,4)$ and $(2,3,5)$}\label{fig:235}
\end{figure}
\end{landscape}
\subsubsection{The domestic factorization property}
As a byproduct, the investigation of $\svect\XX$ yields an interesting factorization property for the category of vector bundles, and by tilting also for the categories of preprojective (resp.\ preinjective) modules over the tame hereditary algebra $H$ which corresponds to $\coh\XX$ by means of tilting theory. Though representations over tame hereditary algebras and coherent sheaves over weighted projective lines of positive Euler characteristic have been intensively studied for a long time, it seems that the factorization property (ii) of the next theorem has not been considered before.

\begin{theorem}[Domestic factorization property]\label{thm:factorization_domestic}
Assume $\eulerchar\XX>0$, and let  $\svect\XX$ be the stable category of vector bundles on $\XX$. Then the following holds:

(i) As a functor the suspension $[1]:\svect\XX\ra \svect\XX$ is isomorphic to the shift $X\mapsto X(\vx_1)$ by $\vx_1$. It induces on slopes the mapping
$q\mapsto q+\delta(\vx_1)$.

(ii) Assume that $X$ and $Y$ are indecomposable objects of
$\svect\XX$ such that $\slope{Y}-\slope{X}>\delta(\vx_1+\vom)$. Then
each morphism $u:X\ra Y$ in $\vect\XX$ factors through a direct sum of
line bundles.
\end{theorem}

\begin{proof}
The first assertion of (i) either follows from Proposition~\ref{prop:autom_domestic} or from the general result of~Proposition~\ref{prop:2ab:suspension}. The second assertion then follows from the general fact that grading shift induces an action on the slope given by the expression $\slope{X(\vx)}=\slope{X}+\delta(\vx)$. Concerning (iii) we apply $\Hom{}{}{y}$ to the distinguished exact sequence $\mu:0\ra X\up{j}\injh{X}\up{\pi} X[1]\ra 0$ given by the injective hull of $X$. By minimality of the injective hull we know that $X[1]$ is indecomposable, and we obtain an exact sequence $0\ra (X[1],Y)\up{\pi^*}(\injh{X},Y)\up{j^*}(X,Y)\ra \Ext1{}{X[1]}{Y}$. In view of stability the term $\Ext1{}{X[1]}{Y}=\dual{\Hom{}{Y}{X[1](\vom)}}$ is zero, since $\slope{X[1](\vom)}-\slope{Y}=\slope{X}-\slope{Y}+\delta(\vx_1+\vom)>0$. The claim follows.
\end{proof}

\subsection{Euler characteristic zero}\label{sect:euler:zero}
Deviating from our general assumptions we deal in this section with all tubular weight types, including $(2,2,2,2)$.

We recall that for a hereditary category $\Hh$ its derived  category $\Der\Hh$ has a very concrete description as the \define{repetitive category} $\bigvee_{n\in\ZZ}\Hh[n]$, where each $\Hh[n]$ is a copy of $\Hh$ (objects written $X[n]$ with $X\in\Hh$) and where morphisms are given by $\Hom{}{X[n]}{Y[m]}=\Ext{m-n}\Hh{X}{Y}$ and composition is given by the Yoneda product.

\subsubsection{The interval category}
By~\cite[Proposition 5.5]{Geigle:Lenzing:1987} each indecomposable vector bundle is semistable. Moreover, the semistable indecomposable bundles of fixed slope $q$ form a $\PP^1(k)$-family $\Cc_q$ of pairwise orthogonal tubes~\cite{Lenzing:Meltzer:1993}.

\begin{lemma} \label{lem:slope}
Assume $\eulerchar\XX=0$.  There exists a bijection $\al:\QQ\ra\QQ$ such that for each indecomposable vector bundle $E$ of rank $\geq2$ and slope $\slope{E}=q$ we have $\slope(E[1])=\al(q)$.
\end{lemma}
\begin{proof}
For each rational number $q$ the indecomposable vector bundles of slope $q$ form a maximal mutually Hom-orthogonal system $\Cc_q$ of tubes (indexed by $\PP^1(k)$) in the category $\vect\XX$. Passing to the stable category we conclude that the indecomposable vector bundles of rank $\geq2$ and slope $q$ form a maximal mutually Hom-orthogonal system $\sCc_q$ of tubes in $\svect\XX$. Conversely, each maximal mutually Hom-orthogonal system of tubes in $\vect\XX$ (resp.\ in $\svect\XX$) is of the form $\Cc_q$ (resp.\ $\sCc_q$) for some uniquely defined rational number $q$. Hence $\sCc_q[1]$ equals $\sCc_{q'}$ for a uniquely defined $q'$, and we put $q'=\al(q)$.
 \end{proof}

It requests  a significant amount of technical knowledge to determine the bijection $\al$ explicitly. It turns out that $\al$ is a piecewise fractional linear function. For details we refer to the appendix, in particular to Theorem~\ref{thm:suspension}. For the present treatment it suffices to know the particular values of $\al^{-1}$, determined by the following lemma. Let $q$ be a rational number, then the \define{interval category} $\Hh\sz{q}$ is the full subcategory of $\svect\XX$ defined as the additive closure of all indecomposables $X$ with slope in the range $\al^{-1}(q)<\slope{X}\leq q$.

\begin{lemma} \label{lemma:rank_function}
We have
$$
\al^{-1}(0)=\begin{cases}
-4/3 &\textrm{for type } (2,2,2,2)\\
-3/2 &\textrm{for type } (3,3,3)\\
-2   &\textrm{for type } (2,4,4)\\
-3   &\textrm{for type } (2,3,6).
\end{cases}
$$
Moreover, assume $\simp{}$ is quasi-simple in a homogeneous tube of
slope $0$. Then $\simp{}$ has rank $2$, $3$, $4$ or $6$ and $\simp[-1]$
has rank $6$, $6$, $4$ or $6$ according as $\XX$ has weight type
$(2,2,2,2)$, $(3,3,3)$, $(2,4,4)$ or $(2,3,6)$, respectively. Further
we have $\sHom{}{E}{\simp{}}\neq 0$ for each indecomposable $E$ of
$\Hh\sz{0}$ of slope different from $0$.
\end{lemma}

\begin{proof}
  We note first that $\simp{}$ has rank $2$, $3$, $4$ or $6$ according
  as $\XX$ has type $(2,2,2,2)$, $(3,3,3)$, $(2,4,4)$ or $(2,3,6)$,
  respectively. It is not difficult to verify that, depending on the
  weight type, the projective hull of $\simp{}$ in the category
  $\vect\XX$ is given by the corresponding distinguished exact
  sequence below:
$$
(2,2,2,2): 0\ra \simp{}[-1]\lra \bigoplus_{i=1}^4\bigoplus_{j\in\ZZ_2}\Oo(-\vx_i+j\vom)\lra \simp{} \ra 0,
$$
$$
(3,3,3): 0\ra \simp{}[-1]\lra \bigoplus_{i=1}^3\bigoplus_{j\in\ZZ_3}\Oo(-\vx_i+j\vom)\lra \simp{} \ra 0,
$$
$$
(2,4,4):0\ra \simp{}[-1]\lra \bigoplus_{i=2}^3\bigoplus_{j\in\ZZ_4}\Oo(-\vx_i+j\vom)\lra \simp{} \ra 0,
$$
$$
(2,3,6):0\ra \simp{}[-1]\lra \bigoplus_{i=1}^2\bigoplus_{j\in\ZZ_6}\Oo(-i\vx_3+j\vom)\lra \simp{} \ra 0
$$
This settles all the claims except the last one.

To show $\sHom{}{E}{\simp{}}\neq0$ we proceed as follows. For any
indecomposable $F$ of rank $r$, degree $d$ and
$\slope{F}<\slope{\simp{}}=0$ we have
$\dim{\sHom{}{F}{\simp{}}}=\euler{F}{\simp{}}_\XX$ in view of
semi-stability, where the Euler form is calculated in $\coh\XX$. Since
$\simp{}(\vom)=\simp{}$ we have
$$
\euler{F}{S}=1/\bar{p}\sum_{j\in\ZZ_{\bar{p}}}\euler{F(j\vom)}{\simp{}}_\XX=1/\bar{p}\left|\begin{array}{cc}r
    & \bar{p}\\ d & 0\end{array}\right|=-d,
$$
where the Riemann-Roch theorem is used. With the aid of this
expression it is straightforward to establish the last claim.
\end{proof}

\begin{proposition}[Interval category] \label{prop:interval_category} Assume
  $\eulerchar\XX=0$. The full subcategory $\Hh$ of $\svect\XX$
  obtained as the additive closure of all indecomposable vector
  bundles $E$ with slope in the range $\al^{-1}(0)< \slope{E}\leq
  0$ is abelian and equivalent to $\coh\XX$. Moreover, a sequence
  $0\ra X \ra Y \ra Z \ra 0$ is exact in $\Hh\sz{q}$ if and only if
  $X\ra Y \ra Z \ra $ is a distinguished triangle in $\svect\XX$.
\end{proposition}

\begin{proof}
\underline{Step 1.} The category $\Hh$ is abelian. Let $\Cc^{\geq0}$ (resp.\ $\Cc^{\leq0}$) be the additive closure in $\svect\XX$ of all indecomposables of slope $\leq0$ (resp.\ $>\al^{-1}(0)$). Further define $\Cc^{\geq n}=\Cc^{\geq0}[-n]$ and $\Cc^{\leq n}=\Cc^{\leq0}[-n]$ for any $n\geq1$. Then we have $\sHom{}{\Cc^{\leq0}}{\Cc^{\geq1}}=0$, $\Cc^{\leq 0}\subset\Cc^{\leq1}$ and $\Cc^{\geq1}\subset\Cc^{\geq0}$. Moreover, for each $X$ in $\svect\XX$ there is a triangle $X'\ra X \ra  X''\ra $ with $X'\in \Cc^{\leq0}$ and $X''\in \Cc^{\geq1}$. This follows from the fact that each indecomposable object belongs either to $\Cc^{\leq0}$ or to $\Cc^{\geq1}$. We deduce that the pair $(\Cc^{\leq0},\Cc^{\geq0})$ is a $t$-structure of $\svect\XX$ implying by \cite{Beilinson:Bernstein:Deligne:1982} that the heart $\Cc=\Cc^{\leq0}\cap \Cc^{\geq0}$ of the $t$-structure is abelian.

\underline{Step 2.} The category $\Hh$ is hereditary. By means of Step
1 we can identify $\Ext1{\Hh}{X}{Y}$ and $\sHom{}{X}{Y[1]}$, where
$[1]$ denotes the suspension functor for $\svect\XX$. This implies
that $\Hh$ has Serre duality
$\dual{\Ext{1}{\Hh}{X}{Y}}=\sHom{}{Y}{X(\vom)}$ inherited from the
Serre duality in $\svect\XX$. Notice that the AR-translate $X\mapsto
X(\vom)$ from $\svect\XX$ restricts to a self-equivalence of $\Hh$. In
particular the category $\Hh$ is hereditary, that is, has vanishing
$\Ext{i}{}{-}{-}$ for all $i\geq2$.

\underline{Step 3.} The category $\Hh$ is noetherian, that is, each
ascending sequence of subobjects of an object $X$ is stationary. In
view of Lemma~\ref{lemma:rank_function} this follows from
\cite[Lemma~5.1]{Lenzing:Reiten:2006}.

\underline{Step 4.} The category $\Hh$ has a tilting object $T$ whose
endomorphism ring is the canonical algebra $\La$ associated with
$\XX$. To construct $T$ we start with the canonical tilting bundle
$T_\can$ of $\coh\XX$ defined as the direct sum of all $\Oo(\vx)$ with
$\vx$ in the range $0\leq \vx\leq \vc$. According to
\cite{Lenzing:Meltzer:2000} there exists a self-equivalence $\rho$ of
$\Der{\coh\XX}$ acting on slopes by $q\mapsto -1/(1+2q)$ sending
$T_\can$ to a tilting bundle $T=\rho(T_\can)$ of $\coh\XX$ whose
indecomposable summands have a slope in the range $-1/2\leq q
<0$. Clearly, $\sEnd{}{T}=\End{}{T}$ is the canonical algebra attached
to $\XX$. Invoking Serre duality we obtain
$\dual{\Ext1{\Hh}{T}{T}}=\sHom{}{T}{T(\vom)}=0$, hence $T$ has no
self-extensions in $\Hh$. It remains to be shown that $T$ generates
$\Hh$ as an abelian category, that is, by taking direct summands and
third terms of short exact sequences.  For this one shows that for
each indecomposable vector bundle $X$ of rank $\geq2$ we have
$\sHom{}{T}{X}\neq0$ or $\sHom{}{T}{X[1]}\neq0$. The proof of this fact
is mainly K-theoretic and similar to the proof of
Lemma~\ref{lemma:rank_function}. Thus $T$ is tilting in the category
$\Hh$. It then follows from~\cite{Lenzing:1997} that $\Hh$ is
equivalent to a category of coherent sheaves on a weighted projective
line; inspection of exceptional tubes for $\Hh$ then implies
$\Hh\iso\coh\XX$.
\end{proof}

\subsubsection{Tilting object with canonical endomorphism ring}

\begin{theorem}\label{thm:tubular} Assume $\XX$ has tubular weight type
  $(2,3,6)$, $(2,4,4)$, $(3,3,3)$ or $(2,2,2,2)$. Then there exists a
  tilting object $T$ in the stable category $\svect\XX$ whose
  endomorphism ring is the canonical algebra $\La$ of the same weight
  type. In particular, we have triangle equivalences
  $\svect\XX\iso\Der{\coh\XX}$ (depending on the choice of $T$).
\end{theorem}
\begin{proof}
  We keep the notations of Proposition~\ref{prop:interval_category},
  and show that
  $T=\rho(T_\can)$ is tilting in $\svect\XX$. Since $T$ is a tilting
  object in the abelian category $\Hh$, it generates $\Hh$ as an abelian
  category by closing under direct summands and third terms of short
  exact sequences. Since short exact sequences in $\Hh$ agree with
  distinguished triangles in $\svect\XX$ having their three terms in
  $\Hh$, it follows that the thick subcategory $\gentri{T}$ contains
  $\Hh$. Since each indecomposable object of $\svect\XX$ belongs to
  some $\Hh[n]$ it follows that $T$ generates $\svect\XX$ as a
  triangulated category.

Further $T$ has no self-extensions in $\svect\XX$. Invoking Serre
duality, we have to show that for each non-zero integer $n$ the expression
$\sHom{}{T}{T[n]}=\dual{\sHom{}{T[n-1]}{T(\vom)}}$ is zero. Now, for
$n<0$ (resp.\ for $n\geq2$) the left hand side (resp.\ the right hand
side) vanishes. It thus remains to deal with the case $n=1$ where
already $\Hom{}{T}{T(\vom)}$ is zero.
\end{proof}

\subsubsection{Position of the extension bundles}
Here, we have to discard weight type $(2,2,2,2)$, and only consider the tubular weight triples.

\begin{lemma}
Assume $\XX$ is given by a tubular weight triple. Let $\bar{p}$ be the largest weight. Then each indecomposable bundle $F$ of rank two either has integral or half-integral slope. In the case of integral slope it either has quasi-length $2$ in a tube of $\tau$-period $\bar{p}$, and then is an Auslander bundle, or else it is quasi-simple in a tube of $\tau$-period $\bar{p}/2$ (if such a tube exists). In the case of half-integral slope $F$ is quasi-simple in a tube of $\tau$-period $\bar{p}$. In particular $F$ is quasi-simple in $\svect\XX$.
\end{lemma}
\begin{proof}
  This follows from the description of $\coh\XX$ given in
  \cite{Lenzing:Meltzer:1993}.
\end{proof}
By way of example, we discuss weight type $(2,3,6)$ in more detail, the analysis is similar for the two remaining weight triples $(2,4,4)$ and $(3,3,3)$. Let $F$ be an indecomposable vector bundle of rank two. Then only the following cases appear:
\begin{enumerate}
\item $F$ has integral slope, and then either has quasi-length two in
  the tube of $\tau$-period $6$, and then is an Auslander bundle, or
  else $F$ is quasi-simple in the tube of $\tau$-period $3$.
\item $F$ has half-integral slope, and then is quasi-simple in the
  tube of $\tau$-period $6$.
\end{enumerate}
By Theorem~\ref{thm:extension_bundle} any of these possibilities
occurs for the indecomposable summands $\extb{L}{\vx}$,
$\vx=a\vx_3+b\vx_2$, $a=0,\ldots,4$, $b=0,1$ of the tilting rectangle for $\svect\XX$. For $\vx\in\set{0,\vx_2,\vdom,\vdom-\vx_2}$
the extension bundle $F=\extb{L}{\vx}$ is an Auslander bundle, for
$\vx\in\set{\vx_3,\vx_2+\vx_3,\vdom-\vx_3,\vdom-(\vx_2+\vx_3)}$ the
bundle $F$ has half-integral slope and then $\tau$-period $6$. For the
two remaining cases $\vx\in\set{2\vx_3,\vx_2+2\vx_3}$ the bundle $F$
has integral slope and is quasi-simple of $\tau$-period $3$. It
follows that all extension bundles $\extb{L}{\vx}$ are quasi-simple in
$\svect\XX$.

\subsubsection*{Tubular factorization property}
Another, quite interesting, and previously unknown feature of a  category
$\vect\XX$ of vector bundles for tubular type is the factorization property of
Theorem~\ref{thm:factorization_tubular} treated in the Appendix.

\subsubsection*{Automorphism group} For arbitrary weight triples, the
structure of the automorphism group of $\svect\XX$ is still open. The
situation is better for the case of Euler characteristic zero, since
then $\svect\XX$ is equivalent to $\Der{\coh\XX}$. In this case
\cite[Theorem~6.2]{Lenzing:Meltzer:2000} gives the structure of $\Aut{\Der{\coh\XX}}$, where the braid group $B_3$ on three strands plays a central role. In particular, for
weight type $(3,3,3)$
this allows to express the suspension $[1]$ as a composition of
tubular mutations and an automorphism of $\XX$ as follows. Denote by $L$, resp.\ $S$, the tubular mutation~\cite{Meltzer:1997} at the tube in $\svect\XX$ containing $\extb{}{\vx_1}$ respectively $\extb{}{\vdom}$. Further let $\ga$ be the automorphism of order two of $\svect\XX$ obtained by interchanging the roles of $x_2$ and $x_3$ in the projective coordinate algebra $S$ of $\XX$. It then follows from
\cite[Proposition 7.2]{Lenzing:Meltzer:2000} that $[1]=(LS)^3\ga$.

\subsection{Negative Euler characteristic} \label{sect:euler:negative}
Let $\eulerchar\XX<0$. Here, the classification problem for $\svect\XX$ is
wild. The study of these categories is also related to the
investigation of Fuchsian singularities in \cite{KST-2,
  Lenzing:Pena:2011} but, except for weight type $(2,3,7)$, yields a
different stable category of vector bundles (since only one
$\tau$-orbit of line bundles is factored out in the Fuchsian case).

We recall from~\cite{Lenzing:Pena:1997} that for negative Euler
characteristic each Auslander-Reiten component $\Cc$ of $\vect\XX$ has
shape $\ZZ
A_\infty$. Moreover, additivity of the rank function on Auslander-Reiten meshes implies that each line bundle $L$ sits at the border of its component. Further each indecomposable bundle $E$ of rank two, hence each extension bundle, either has quasi-length two in a component $\Cc$ containing a line bundle $L$ or else $E$ is quasi-simple in its component  $\Cc$.

Only the components that contain a line bundle, called \define{distinguished} in \cite{Lenzing:Pena:1997}, are affected when passing to the stable category. Like all the others,
they keep the shape $\ZZ\AA_\infty$ but get a border formed by a $\tau$-orbit of Auslander bundles. We call these components \define{distinguished} in $\svect\XX$.

\begin{proposition} \label{prop:Z_Delta}
Assume $\XX$ has weight type $(p_1,p_2,p_3)$ and $\eulerchar\XX<0$. Then
each Auslander-Reiten component of $\svect\XX$ has shape
$\ZZ\AA_\infty$. Moreover there is a natural bijection between the set
of all Auslander-Reiten components to the infinite set of all regular
Auslander-Reiten components over the wild path algebra $\La_0$ over
the star $[p_1,p_2,p_3]$.
\end{proposition}
\begin{proof}
The first assertion follows from the preceding comments. The argument
further implies that there is a natural bijection between the set of
Auslander-Reiten components in
$\vect\XX$ and in $\svect\XX$, respectively. The second claim then follows
from \cite{Lenzing:Pena:1997}.
\end{proof}
According to a result of Kerner, compare~\cite{Crawley-Boevey:Kerner:1994}, for any two connected wild hereditary algebras there is a `natural' bijection between their `spaces' of regular components. By the preceding result this extends to the `spaces' of components for $\vect\XX$ resp.\ $\svect\XX$ in case of negative Euler characteristic.

Following Happel a triangulated $k$-category $\Tt$ (resp.\ a finite
dimensional $k$-algebra $A$) is called \define{piecewise hereditary}
if $\Tt$ (resp.\ $\Der{\mod{A}}$) is triangle-equivalent to the
bounded derived category $\Der{\Hh}$ of a Hom-finite hereditary
abelian $k$-category. Piecewise hereditary algebras cover a wide range
of phenomena, and they appear in many contexts,
compare~\cite{Happel:Zacharia:2008}, \cite{Happel:Zacharia:2010},
\cite{Happel:Seidel:2010}.

\begin{proposition} \label{thm:Z_Delta}
A category $\svect\XX$ is piecewise hereditary if and only if its
Euler characteristic is non-negative. For $\eulerchar\XX>0$ (resp.\
$\eulerchar\XX=0$) it is triangle-equivalent to the derived $\Der\Hh$,
where $\Hh$ is the module category $\mod{H}$ over a finite dimensional
hereditary algebra $H$ (resp.\ the category $\coh\XX$ of coherent
sheaves over a weighted projective line of tubular type).
\end{proposition}

\begin{proof}
The case of positive Euler characteristic (resp.\ zero Euler
characteristic) is covered by Proposition~\ref{prop:tilting_domestic}
(resp.\ Theorem~\ref{thm:tubular}).

Assume now that $\eulerchar\XX<0$ and $\svect\XX$ is
triangle-equivalent to $\Der\Hh$ for some hereditary $k$-category
$\Hh$. Because $\svect\XX$ has a tilting object, see
Theorem~\ref{thm:std_tilting_obj}. A fundamental result of Happel~\cite{Happel:2001} implies that $\Hh$ may be chosen to be of type $\mod{H}$ with $H$ the path algebra of a finite acyclic quiver or else of type $\coh\YY$ with $\YY$ a weighted projective line. In the first case we have just one Auslander-Reiten component in $\Der{\mod{H}}$ (of type $\ZZ\De$ with $\De$ Dynkin) while
in the second case we have lots of Auslander-Reiten components which are tubes. Because of Proposition~\ref{prop:Z_Delta} this prohibits the existence of a triangle equivalence between $\svect\XX$ and $\Der\Hh$ if $\eulerchar\XX<0$.
\end{proof}

\begin{remark}
There are other useful criteria for checking whether a finite dimensional algebra $A$, supposed to be of finite global dimension, has a chance that its bounded derived category $\Der{\mod{A}}$ is of the form $\svect\XX$, where $\XX$ is given by a weight triple. The most useful ones invoke the \define{Coxeter transformation} or its characteristic polynomial, the \define{Coxeter polynomial}. Recall that the Coxeter transformation $\Phi_{\svect\XX}$ is the automorphism of $\Knull{\svect\XX}$ induced by the Auslander-Reiten translate $\tau$ of $\svect\XX$. The Coxeter polynomial of $\svect\XX$ is defined as the characteristic polynomial of $\Phi_{\svect\XX}$. Since $\svect\XX$ is fractional Calabi-Yau by Proposition~\ref{prop:CY:vom-vc}, the Coxeter transformation of $\svect\XX$ is always periodic; accordingly the Coxeter polynomial of $\svect\XX$ factors into cyclotomic polynomials. The following trisection turns out to be useful for the checking procedure.

\emph{We exclude from the present discussion the weights $(2,2,n)$}
yielding a category $\svect\XX$ triangle-equivalent to
$\mod{k\vAA_{n-3}}$. Then the following properties distinguish

(a) If $\eulerchar\XX>0$, then the Coxeter transformation for $\svect\XX$
is periodic of period 6,12 or 30, and $\Knull{\svect\XX}$ has rank $4,
6$ or $8$.

(b) If $\eulerchar\XX=0$, then the Auslander-Reiten translate and the
Coxeter transformation are periodic of period 2,3,4 or 6, and
$\Knull{\svect\XX}$ has rank $6,8, 9$ or $10$.

(c) If $\eulerchar\XX<0$, then the Coxeter transformation is periodic, the
AR-translate is not periodic, and $\Knull{\svect\XX}$ has rank $\geq
12$.

Note further that in cases (a) and (c) the AR-translation is not periodic.
\end{remark}

\section{The link to representation theory} \label{sect:tilting}

\subsection{The tilting cuboid}
The treatment of the (possibly degenerate) cuboid $T_\cub$, formed by the extension bundles $\extb{L}{\vx}$, for a fixed line bundle $L$, is central for our investigation of the stable categories of vector bundles $\svect\XX$ in case of a weight triple. For special weight types, like $(2,a,b)$, we derive from $T$ further tilting objects that yield interesting applications. They allow us, for instance, to develop a general form to what we call \define{Happel-Seidel duality}, compare~\cite{Happel:Seidel:2010} in the piecewise hereditary case. We recall that $k\vAA_n$ denotes the path algebra of the equioriented quiver $\AA_n$. A different proof of our next result is due to A.~Takahashi (unpublished). It seems that Leszcynski~\cite{Leszcynski:1994} was the first to take an interest in properties of tensor products of path algebras of Dynkin quivers. The Calabi-Yau property of tensor products is investigated by Ladkani~\cite{Ladkani:2012}; further~\cite{Ladkani:2009} is of relevance for the topics of this Section.

\begin{theorem}[Tilting cuboid] \label{thm:std_tilting_obj}
Assume that $\XX$ is given by a weight triple $(a,b,c)$ with $a,\,b,\,c\geq 2$. Let $L$ be a line bundle. Then $$T_\cub(L)=\bigoplus_{0\leq\vx\leq\vdom}\extb{L}{\vx}$$ is a tilting object in $\svect{\XX}$, called the \define{tilting cuboid}, with endomorphism ring $$\sEnd{}{T}\simeq k\vAA_{a-1}\otimes k\vAA_{b-1}\otimes k\vAA_{c-1}.$$
\end{theorem}
\begin{proof}
  Without loss of generality let $L=\Oo$. Write
  $(a,b,c)=(p_1,p_2,p_3)$ and $T=T_\cub(L)$.

(1) $T$ is extension-free: Since $\det(\extb{}{\vx})=\vx+\vom$ we see
by Lemma~\ref{lemma:determinant} that the indecomposable summands of $T$ are
pairwise non-isomorphic. We assume $0\leq\vx,\,\vy\leq\vdom$ and
$n\in\ZZ$, $n\neq 0$. We have to show
that $$\sHom{}{\extb{}{\vx}}{\extb{}{\vy}[n]}=0.$$ For $\vx=\vy$ this
is true, since $\extb{}{\vx}$ and $\extb{}{\vy}$ are exceptional in
$\svect(\XX)$, by Corollary~\ref{cor:exceptional}. So we assume
$\vx\neq\vy$. By Serre duality we
have $$\sHom{}{\extb{}{\vx}}{\extb{}{\vy}[n]}
=\dual\sHom{}{\extb{}{\vy}[n-1]}{\extb{}{\vx}(\vom)}.$$ We assume that
this expression is non-zero, and show that this yields a
contradiction.  By the determinant criterion
Lemma~\ref{lemma:determinant} we get $$(i)\ \ \vx\leq\vy+n\vc\ \
\text{and}\ \ (ii)\ \ \vy+(n-1)\vc\leq\vx+2\vom.$$ This yields
$(n-1)\vc-2\vom\leq\vx-\vy\leq n\vc$.  First assume $n<0$. From (i) we
get $-n\vc\leq\vx<\vy\leq\vdom$, thus $\vc\leq -n\vc<\vdom$, and then
$(a-2)\vx_1+(b-2)\vx_2+(c-2)\vx_3\-\vc=\vdom-\vc>0$, which gives
(normal form!) a contradiction. Thus we assume in the following
$n>0$. Then $\vx\leq\vy$ is not possible, since otherwise
$0\leq\vy-\vx\leq 2\vom$ by (ii), which gives a contradiction. Hence
assume $\vx\not\leq\vy$, equivalently $\vy\leq\vx+\vc+\vom$. By
symmetry, we can assume that we can
write $$\vy-\vx=\pm\ell_1\vx_1\pm\ell_2\vx_2-\ell_3\vx_3$$ with
$0\leq\ell_i\leq p_i-2$ and $\ell_3\geq 1$. Up to symmetry we have to
consider the three cases $(\pm\ell_1,\pm\ell_2)=(\geq 0,\geq
0),\,(\geq 0,<0),\,(<0,<0)$.

\underline{1.~case}: $\vy-\vx=\ell_1\vx_1+\ell_2\vx_2-\ell_3\vx_3$. If
$n\geq 2$, or if $n=1$, $\ell_3=1$ we have $0\leq
(p_1+1)\vx_1+(p_2+1)\vx_2+(p_3-(\ell_3-1))\vx_3+(n-2)\vc$ (normal
form!), which by simple calculations implies
$\vx+2\vom\leq\vy+(n-1)\vc+\vc+\vom$, which is equivalent to
$\vx+2\vom\not\geq\vy+(n-1)\vc$, contradicting~(ii).

It remains to consider the case $n=1$ and $2\leq\ell_3\leq
p_3-2$. By~\eqref{eq:suspension} we have
$\extb{}{\vy}[1]=\extb{}{\vdom-\vy}(\vy-\vom)$. Assume there is
$0\neq\underline{u}\in\sHom{}{\extb{}{\vx}}{\extb{}{\vy}[1]}$.  This
gives the following diagram $$\xymatrix{ 0\ar @{->}[r]&\Oo(\vom)\ar
  @{->}[r]^-{\al}&\extb{}{\vx}\ar @{->}[d]^-{u} \ar
  @{->}[r]^-{\be}&\Oo(\vx)\ar @{->}[r]&0\\ 0\ar @{->}[r]&\Oo(\vy)\ar
  @{->}[r]^-{\al'}&\extb{}{\vy}[1] \ar
  @{->}[r]^-{\be'}&\Oo(\vdom-\vom)\ar @{->}[r]&0. }$$ Since
$\vdom\leq\vc+\vdom=2\vom+\vc+\vom$ we have $\vdom-2\vom\not\geq
0$. Thus $\Hom{}{\Oo(\vom)}{\Oo(\vdom-\vom)}=0$, and therefore
$\be'\circ u\circ\al=0$. This induces a morphism
$\Oo(\vx)\ra\Oo(\vdom-\vom)$, which is non-zero, since $u$ does not
factor through line bundles. In particular $\vx\leq\vdom-\vom$
follows. But then $0\leq\vx\leq\vdom-\vom=\sum_{i=1}^3 (p_i
-1)\vx_i-\vc$, a contradiction, since the right hand term is in normal
form.

\underline{2.~case}: $\vy-\vx=\ell_1\vx_1-\ell_2\vx_2-\ell_3\vx_3$
with $\ell_1\geq 0$ and $\ell_2,\,\ell_3>0$. If $n=1$, we have $0\leq
(p_1-(\ell_1-1))\vx_1+(\ell_2-1)\vx_2+(\ell_3-1)\vx_3+(1-n)\vc$, which
yields $\vy-\vx+n\vc\leq\vc+\vom$, that is, $\vy-\vx+n\vc\not\geq 0$,
contradicting~(i). If $n\geq 3$, or if $n=2$ and $\ell_2=1$ or
$\ell_3=1$ we have $0\leq (\ell_1+1)\vx_1+(p_2-(\ell_2-1))\vx_2
+(p_3-(\ell_3-1))\vx_3+(n-3)\vc$, and this yields
$\vx-\vy+2\vom-(n-1)\vc\leq\vc+\vom$, that is,
$\vx-\vy+2\vom-(n-1)\vc\not\geq 0$, contradicting~(ii). It remains to
consider $n=2$ and $\ell_2,\,\ell_3\geq 2$. Assumption of the
existence of
$0\neq\underline{u}\in\sHom{}{\extb{}{\vx}}{\extb{}{\vy}[2]}$ yields,
now using $\extb{}{\vy}[2]\simeq\extb{}{\vy}(\vc)$, a diagram
$$\xymatrix{ 0\ar @{->}[r]&\Oo(\vom)\ar
  @{->}[r]^-{\al}&\extb{}{\vx}\ar @{->}[d]^-{u} \ar
  @{->}[r]^-{\be}&\Oo(\vx)\ar @{->}[r]&0\\ 0\ar
  @{->}[r]&\Oo(\vom+\vc)\ar @{->}[r]^-{\al'}&\extb{}{\vy}[2] \ar
  @{->}[r]^-{\be'}&\Oo(\vy+\vc)\ar @{->}[r]&0. }$$ Since
$\vy+\vc-\vom\not\geq 0$ (via the normal form), this induces a non-zero
morphism $\Oo(\vx)\ra\Oo(\vy+\vc)$, but this gives a contradiction,
since one checks (via the normal form) that $\vy+\vc-\vx\not\geq 0$.

\underline{3.~case}: $\vy-\vx=-\ell_1\vx_1-\ell_2\vx_2-\ell_3\vx_3$
with $1\leq\ell_i\leq p_i-2$. For $n=1,\,2$ we have
$0\leq\sum_{i=1}^3(\ell_i+1)\vx_i+(n-2)\vc$, which yields
$\vy-\vx+n\vc\leq\vc+\vom$, that is, $\vy-\vx+n\vc\not\geq 0$,
contradicting~(i). If $n\geq 4$, or if $n=3$ and $\ell_i=1$ for at
least one $i$, then $0\leq\sum_{i=1}^3 (p_i+1-\ell_i)\vx_i+(n-4)\vc$,
which yields $\vx+2\vom\leq\vy+(n-1)\vc+\vc+\vom$, that is,
$\vx+2\vom\not\geq\vy+(n-1)\vc$, contradicting~(ii). It remains to
consider the case $n=3$ and all $\ell_i\geq 2$. If we assume
$\sHom{}{\extb{}{\vx}}{\extb{}{\vy}[3]}\neq 0$, then equivalently, by
Serre duality, there is
$0\neq\underline{u}\in\sHom{}{\extb{}{\vy}[2]}{\extb{}{\vx}(\vom)}$,
and this yields a diagram
$$\xymatrix{ 0\ar @{->}[r]&\Oo(\vom+\vc)\ar
  @{->}[r]^-{\al}&\extb{}{\vy}[2]\ar @{->}[d]^-{u} \ar
  @{->}[r]^-{\be}&\Oo(\vy+\vc)\ar @{->}[r]&0\\ 0\ar
  @{->}[r]&\Oo(2\vom)\ar @{->}[r]^-{\al'}&\extb{}{\vx}(\vom) \ar
  @{->}[r]^-{\be'}&\Oo(\vx+\vom)\ar @{->}[r]&0. }$$ Again we get
$\beta'\circ u\circ\alpha=0$, inducing a non-zero morphism
$\Oo(\vom+\vc)\ra\Oo(2\vom)$, which gives a contradiction, since
$\vom\not\geq\vc$.

(2) $T$ generates $\svect{\XX}$: By part (1) the extension bundles
$\extb{}{\vx}$ ($0\leq\vx\leq\vdom$) can be ordered in such a way that
they form an exceptional sequence. Thus the smallest triangulated
subcategory $\mathcal{C}$ of $\svect{\XX}$ containing $T$ is generated
by an exceptional sequence. By~\cite[Thm.~3.2]{Bondal:1989}
and~\cite[Prop.~1.5]{Bondal:Kapranov:1989} then $\svect{\XX}$ is
generated by $\mathcal{C}$ together with $\rperp{\mathcal{C}}$. By
Corollary~\ref{cor:zero-perpendicular} we have
$\rperp{\mathcal{C}}=0$, thus $\svect{\XX}$ is generated by $T$.

(3) The claim concerning the endomorphism ring is derived from
Lemma~\ref{lem:induction-step-pull-back} and
Proposition~\ref{prop:non_zero}.
\end{proof}
In the particular case of $\svect\XX$, an application of the $\LL$-graded version of Orlov's theorem, see Theorem~\ref{thm:Orlov:L-graded}, yields a complete exceptional sequence instead of a tilting object for $\svect\XX$, a result that is less specific than Theorem~\ref{thm:std_tilting_obj}. We note that Orlov's approach yields a different relationship between the categories $\Der{\coh\XX}$ and $\svect\XX$ than provided by the methods of this paper.

We next deal with a quite remarkable property of the tilting cuboid for the stable category of vector bundles.
\begin{corollary} \label{cor:extb:simp:proj:inj}
We put $A=\sEnd{}{T_\cub}$, we fix the line bundle $L$ and assume $0\leq\vx\leq\vdom$. Under the identification $\Der{\mmod{A}}=\svect\XX$ the indecomposable projective $A$-module $P(\vx)$ corresponds to the extension bundle $\extb{L}{\vx}$,  the simple $A$-module $S(\vx)$ corresponds to the Auslander bundle $E_L(\vx)$ and the indecomposable injective $A$-module $I(\vx)$ corresponds to the extension bundle $\extb{L(\vx)}{\vdom-\vx}=(\extb{L}{\vdom-\vx})(\vx)$.

In particular, 'the' exceptional sequence of extension bundles $\extb{L}{\vx}$ is obtained up to suspension and in the reverse order from 'the' exceptional sequence of Auslander bundles $E_L(\vx)$ by means of a sequence of mutations.
\end{corollary}
\begin{proof}
This is immediate from Proposition~\ref{prop:extb:abdle}.
\end{proof}

\begin{corollary}
The Auslander bundles $E(\vx)$, $0\leq\vx\leq\vdom$, form a complete exceptional sequence with respect to any linear order refining the order $\vx\geq\vy$. For weight type different from $(2,2,2)$ they don't form a tilting object.
\end{corollary}
\begin{proof}
Since the indecomposable summands $\extb{}{\vx}$, $0\leq\vx\leq\vdom$, of $T_\cub$ form a complete exceptional sequence, the endomorphism algebra $\Si=\Si_{(a,b,c)}$ is triangular. It follows that as the indecomposable projective $\Si$-modules, also the simple $\Si$-modules form an exceptional sequence, however in the reverse order. Invoking Schur's lemma the endomorphism ring of the direct sum $U$ of all simple modules has the form $k^n$, for some integer $n$. Assuming that $U$ is tilting, connectedness of $\svect\XX$ implies $n=1$ which is only possible for type $(2,2,2)$.
\end{proof}

\begin{corollary} \label{cor:abundles:Z-basis}
The Grothendieck group of $\svect\XX$ is free abelian of rank $(a-1)(b-1)(c-1)$ on the classes $\sclass{\extb{}{\vx}}$ (resp.\ on the classes $\sclass{E(\vx)}$), $0\leq\vx\leq\vdom$.~\qed
\end{corollary}

By $[1,n]$ we denote the poset associated with the equioriented quiver $\vAA_n$.
\begin{corollary}
Assume a weight triple $(a,b,c)$. The endomorphism algebra $\sEnd{}{T_\cub}$ is isomorphic to the incidence algebra of the poset $P={[1,a-1]}\times{[1,b-1]}\times {[1,c-1]} $.~\qed
\end{corollary}
For weight type $(2,2,c)$ and $c=2$ the poset $P$ is a point, for $c\geq3$ it has linear shape. For weight type $(2,b,c)$ with
$b,\,c\geq3$ the poset $P$ has rectangular shape. Finally, for weight
type $(a,b,c)$ with $a,\,b,\,c\geq3$ we obtain `cubical' shape.

\subsection{Nakayama algebras realizable for type $(2,a,b)$}

\begin{proposition} \label{prop:2ab}
Assume weight type $(2,a,b)$. Then the following holds:

(i) We have $\extb{L}{\vdom-\vx}=\extb{L}{\vx}(\bx_1-\vx)$ for any
$0\leq\vx\leq\vdom$.

(ii) The extension bundle $\extb{L}{\vdom}$ is isomorphic to the
Auslander bundle $E_L(\bx_1)$.

(iii) For each indecomposable vector bundle $F$ of rank two we have
$F[1]\iso F(\vx_1)$.
\end{proposition}

\begin{proof}
To prove (i), it suffices to show that the two objects have the same class in the Grothendieck group. Observe that we have short exact sequences
\begin{eqnarray*}
0\lra &L(\vom)&\up{x_1}L(\vx_1+\vom)\lra \simp{1} \ra 0 \\
0\lra &L(\vom+\bx_1-\vx)&\up{x_1} L(\vc+2\vom-\vx) \ra \simp{1} \ra 0
\end{eqnarray*}
where $\simp{1}$ denotes the unique simple sheaf concentrated in $x_1$ with $\Hom{}{L}{S_1}=k$. This uses that $p_1=2$, and that $\vx_1$ does not occur in the expression $\vx=\ell_2\vx_2+\ell_3\vx_3$. Now, $[E_L(\vdom-\vx)]=[L(\vom)]+L(\vx_1+\omega)$ and $[E_L(\bx_1-\vx)]=[L(\vx_1+\omega-\vx)]+[L(\vx_1+2\vom)]$. The claim follows.

Assertion (ii) is a special case of (i). For (iii) we may assume that $F$ is an extension bundle $\extb{L}{\vx}$. Then combination of formula \eqref{eq:suspension} with property (i) yields $$\extb{L}{\vx}[1]=\extb{L}{\vdom-\vx}(\vx-\vom)=\extb{L}{\vx}(\vx_1).$$
\end{proof}

Let $\Tt=\{\extb{}{\vx}\mid 0\leq\vx\leq\vdom\}$. By means of Proposition~\ref{prop:2ab} we fix for each $E\in\Tt$ a distinguished exact sequence $0\ra E\stackrel{j_E}\ra\injh{E}\stackrel{\pi_E}\ra E(\vx_1)\ra 0$.

\begin{proposition}
Assume weight type $(2,a,b)$. For each $E,\,F\in\Tt$ and $u\in\Hom{}{E}{F}$ we have a commutative diagram
  $$\xymatrix{0\ar @{->}[r]& E\ar @{->}[r]^-{\pi_E}\ar @{->}[d]_-{u}&
   \injh{E}\ar @{->}[r]^-{j_E}\ar @{->}[d]_-{\bar{u}}& E(\vx_1)\ar
   @{->}[r]\ar @{->}[d]^-{u(\vx_1)}& 0\\
   0\ar @{->}[r]& F\ar @{->}[r]^-{j_F}&
   \injh{F}\ar @{->}[r]^-{\pi_F}& F(\vx_1)\ar
   @{->}[r]& 0
   }$$ with $\bar{u}$ a diagonal map.
\end{proposition}
\begin{proof}
$0\neq u\in\Hom{}{E}{F}$ is only possible for $E=\extb{}{\vx}$ and $F=\extb{}{\vx+\vy}$ with $0\leq\vx\leq\vx+\vy\leq\vdom$, and by Lemma~\ref{lem:induction-step-pull-back} and Proposition~\ref{prop:cone-v_i} then $\sHom{}{E}{F}=\Hom{}{E}{F}=k$, and $u$ is explicitly known. One may even reduce to $\vy=\vx_i$, and there this is done by an explicit calculation.
\end{proof}

\begin{proposition} \label{prop:2ab:suspension}
For weight type $(2,a,b)$ the suspension functor $[1]$ on $\svect{\XX}$ is given by the shift functor $\sigma(\vx_1)$.
\end{proposition}
\begin{proof}
Let $\Tt=\{\extb{}{\vx}\mid 0\leq\vx\leq\vdom\}$. Then $\underline{\Tt}[1]=\underline{\Tt}(\vx_1)$, moreover $[1]$ and $\sigma(\vx_1)$ agree as functors $\underline{\Tt}\stackrel{\phi}\lra\underline{\Tt}[1]$. By Keller~\cite[Theorem~8.5]{Keller:2007} there are equivalences yielding the following commutative diagram
  $$\xymatrix{\svect{\XX}\ar @{->}[r]^-{\simeq}\ar
    @{->}[d]^-{[1]}_-{\sigma(\vx_1)}& \Der{\mod(\underline{\Tt})}\ar
    @{->}[d]^-{\phi}\\
    \svect{\XX}\ar @{->}[r]^-{\simeq}& \Der{\mod(\underline{\Tt}[1])}.
  }
  $$
\end{proof}

\begin{proposition} \label{prop:2ab:tilting}
Let $\XX$ be of weight type $(2,a,b)$ and $n=(a-1)(b-1)$. We put $U=\bigoplus_{j=0}^{a-2}\extb{}{j\vx_2}$. Then $$T=\bigoplus_{i=0}^{b-2}\tau^i U(i\vx_1)$$ is a tilting object in $\svect{\XX}$ with $\sEnd{}{T}\iso A_n(a)$.
\end{proposition}
\begin{proof}
Invoking Serre duality, one shows easily that $\sHom{}{T}{T[n]}=0$ for each $n\in\ZZ$, $n\neq 0$. Since the number of indecomposable summands of $T$ agrees with the rank of the Grothendieck group of $\svect{\XX}$, it follows that $T$ is a tilting object. Arranging the direct summands of $U$ in the order $E\ra \extb{}{\vx_2}\ra \extb{}{2\vx_2}\ra \cdots\ra \extb{}{(a-2)\vx_2}$ and then continuing in the obvious fashion, we obtain that the quiver of $\sEnd{}{T}$ is the equioriented quiver $\vAA_n$ with $n=(a-1)(b-1)$.  Using a stable form of Lemma~\ref{lem:induction-step-pull-back} we see that the composition of $a-2$ consecutive arrows is non-zero. Putting $E=\extb{}{j\vx_2}$ we see, on the other hand, that the $k$-dual of $\sHom{}{E}{E(\vx_1+\vom)}$ equals $\sEnd{}{E(\vom)}=k$, so that also the composition of $a-1$ consecutive arrows is non-zero. We are left to consider the space $\sHom{}{\extb{}{j\vx_2}}{\extb{}{(j+1)\vx_2}(\vx_1+\vom)}$ whose $k$-dual, by Serre duality, equals $\sHom{}{\extb{}{(j+1}\vx_2}{\extb{}{j\vx_2}}$ and hence is zero. We have thus shown that $\sEnd{}{T}=A_n(a)$.
\end{proof}

\begin{remark} \label{realizable:by:abundles}
For weight type $(2,3,n)$, $n\geq2$, the tilting object $T$ from Proposition~\ref{prop:2ab:tilting} consists of Auslander bundles, since $U$ does. Hence $A_n(3)$ is realizable by a tilting object consisting of Auslander bundles. We conjecture that for weight types $(2,a,b)$ with $a,b\geq4$ there does not exist such a tilting object in $\svect\XX$ consisting only of Auslander bundles.
\end{remark}

\subsection{Happel-Seidel symmetry} \label{ssect:Happel:Seidel}
The phenomenon, treated in this section, was first encountered by Happel-Seidel~\cite{Happel:Seidel:2010} in the context of piecewise hereditary Nakayama algebras. Happel-Seidel symmetry deals with an unexpected derived equivalence of certain Nakayama algebras, which finds a general, and very easy, explanation through stable categories of vector bundles of type $(2,a,b)$. Recall that a finite dimensional algebra $A$ is called \define{Nakayama algebra} if its indecomposable projective or injective modules are uniserial, that is, have a unique composition series. A very natural class of such algebras is formed by the algebras $A_n(r)$, given as the path algebra of the equioriented quiver $1\up{x}2\up{x}3\up{x} \cdots \up{x}n-1\up{x}n$ of type $\AA_n$ subject to all relations $x^r=0$.
\begin{theorem}[Happel-Seidel symmetry] \label{thm:Happel:Seidel:symmetry}
  Let $a,\,b\geq 2$. Then for $n=(a-1)(b-1)$ the algebras $A_n(a)$,
  $A_n(b)$ and the stable category $\svect\XX$ of type $(2,a,b)$ are derived-equivalent.
\end{theorem}
\begin{proof}
 It suffices to interchange the roles of $a$ and $b$ in Proposition~\ref{prop:2ab:tilting}.
\end{proof}

\begin{remark}
For weight type $(2,3,p)$ the tilting object $T_1$ with $\sEnd{}{T_1}=A_{2(p-1)}(3)$ consists of Auslander bundles, the tilting object $T_2$ with $\sEnd{}{T_1}=A_{2(p-1)}(p)$ does not, if $p>3$.
\end{remark}
Table~\ref{tbl:HS_kurz} below corresponds to Table~1 from \cite{Happel:Seidel:2010}. Our table is more complete, and uses a more systematic notation. For a triple $a,b,c$ the symbols $[a,b,c]$, $(a,b,c)$, $\sz{a,b,c|1}$, and $\sz{a,b,c}$ each denote the equivalence type of some triangulated category:
\begin{description}
\item[{$[a,b,c]$}] stands for the type of the bounded derived category of the hereditary star with three branches of length $a,b,c$, respectively. We use the same symbol for the underlying graph, such that $[2,2,n]=\DD_{n+2}$, $[2,3,3]=\EE_6$, $[2,3,4]=\EE_7$, $[2,3,5]=\EE_8$, $[3,3,3]=\tilde\EE_6$, $[2,4,4]=\tilde\EE_7$ and $[2,3,6]=\tilde\EE_8$.
\item[{$(a,b,c)$}] denotes the type of the derived category of coherent sheaves $\Der{\coh\XX}$ over $\XX(a,b,c)$.
\item[{$\sz{a,b,c|1}$}] denotes the type of the \define{extended canonical algebra} in the sense of \cite{Lenzing:Pena:2011}. In the present context, we restrict to the case $\eulerchar\XX<0$, that is, $1/a+1/b+1/c<1$.
\item[{$\sz{a,b,c}$}] denotes the type of the stable category of vector bundles $\svect{\XX}$ with $\XX$ of weight type $(a,b,c)$.
\end{description}
\begin{table}[ht]
\scriptsize
$$
\def\l{{\langle}}
\def\r{{\rangle}}
\def\gr{\cellcolor[gray]{0.8}}
\begin {array}{c|ccccccccccc}
& & & & & & & & & & & \\
11& & & & & & & & & & &[2,2,10]
\\ \hline
10& & & & & & & & & &[2,2,9]&[2,3,9]\\ \hline
9& & & & & & & & &[2,2,
8]&[2,3,8]&[2,3,9]\\ \hline
8& &
& & & & & &[2,2,7]&[2,3,7]&[2,3,8]&(2,3,8]\\ \hline
7& & & & & & &
[2,2,6]&[2,3,6]&[2,3,7]&(2,3,7)&\gr\sz{2,3,7}\\ \hline
6& & & & & &[2,2,5]&[2,3,5]&[2,3,6]&(2,3,6]&(2,3,7]&\gr\sz{2,3,7}\\ \hline
5& & & & &[2,2,4]&[2,3,4]&[2,3,5]&[2,3,6]&[2,3,7]&\gr\sz{2,4,5|1}&\gr\sz{2,4,5}\\ \hline
4& & & &[2,2,3]&[2,3,3]&[2,3,4]&[2,4,4]&(2,4,4)&(2,4,5)&\gr\sz{2,4,5|1}&\gr\sz{2,4,5}\\ \hline
3 & & &[2,2,2]&
[2,2,3]&[2,3,3]&[2,3,4]&[2,3,5]&[2,3,6]&(2,3,6)&(2,3,7)&\gr\sz{2,3,7}\\ \hline
2 & &[3]&[4]&[5]&[6]&[7]&[8]
&[9]&[10]&[11]&[12]
\\ \hline
r/n& &3 &4 &5 &6 &7 &8 &8 &10 &11 &12  \end {array}
$$
\caption{Triangulated type of Nakayama algebra $A_n(r)$ for $n\leq12$.}\label{tbl:HS_kurz}
\end{table}
We observe that the algebras $A_n(2)$ (resp.\ $A_n(n-1)$) are of type $\AA_n$ (resp.\ $\DD_n$). For the further discussion, we will only consider algebras $A_n(r)$ with $r$ different from $2,n-1$.

With the exception of the entries that are marked grey, the triangulated categories displayed in this table are \define{piecewise hereditary}, that is, they are triangle equivalent to the bounded derived category of a hereditary category. For the Nakayama algebras $A_n(r)$ whose bounded derived categories are of type $\svect\XX$ this is untypical, once we assume $n\geq 13$.

Within its category, each of the symbols $[a,b,c]$, $(a,b,c)$ and $\sz{a,b,c}$, respectively, is uniquely determined by the triangulated category $\Tt$ (up to permutation of $a,b,c$), actually by the \define{Coxeter polynomial} of $\Tt$. For $\sz{a,b,c|1}$ this is still open. Concerning $[a,b,c]$ this
is shown, for instance, in \cite[Prop.~3.2]{Lenzing:Pena:2008}. Putting $v_n=(1-x^n)/(1-x)$,
the Coxeter polynomial for $(a,b,c)$ is $\phi=(x-1)^2v_av_bv_c$, see~\cite{Lenzing:Pena:1997}
or \cite{Lenzing:1996}. Factorization of $\phi$ into cyclotomic polynomials then shows that
$\phi$ determines the weight sequence $(a,b,c)$ uniquely up to permutation. For the symbol
$\sz{a,b,c}$ the argument is involved, and we refer to Section~\ref{ssect:periodicity:Coxeter}.

For the special weight types $(2,a,b)$ which only matters for Table~\ref{tbl:HS_kurz} there is a quick argument to recover $a$ and $b$ from the triangulated type $\sz{2,a,b}$. We put $\Tt=\svect{\XX(2,a,b)}$, and first use that the rank of the Grothendieck group of $\Tt$ equals  $(a-1)(b-1)$. Next, we use that by Proposition~\ref{prop:CY} the category $\Tt$ is \define{fractionally Calabi-Yau}, and obtain by the same proposition the Euler characteristic $\eulerchar\XX$ hence the  value of $1/a+1/b$ as another invariant of $\Tt$. This now determines $ab$ and $a+b$, hence the coefficients of the polynomial $(x-a)(x-b)=x^2-(a+b)x+ab$, and then also the set $\{a,b\}$.

Table~\ref{tbl:HS_lang} determines for $n\leq30$ which of the algebras $A_n(r)$ has the Coxeter polynomial of a stable category of vector bundles $\sz{2,r,n}$ with $3\leq r\leq n$. There are foremost the two series of Nakayama algebras
$$
A_{(r-1)(n-1)}(a) \quad \textrm{and}\quad A_{(r-1)(n-1)}(n)
$$
which are related by Happel-Seidel symmetry. For instance, the sequence $\sz{2,3,r}$, $r\geq3$, of stable categories of vector bundles is represented by the algebras $A_{2(r-1)}(3)$ which figure in the left column of Table~\ref{tbl:HS_lang}. By Happel-Seidel symmetry the same series is also represented by the algebras $A_{2(r-1)}(r)$ figuring in the table on a ray starting in the coordinate $(r,2(r-1))$, corresponding to type $\sz{2,r,r}$.

In addition, for fixed $r\geq3$ there occur two further \define{spontaneous series} $A_{(r-1)(rn-2)}(rn)$ living conjecturally on $\svect\XX$ for type $\sz{2,r,rn-1}$, and $A_{(r-1)rn}(rn)$ living conjecturally on $\svect\XX$ for type $\sz{2,r,rn+1}$. The members of the spontaneous series are marked in Table~\ref{tbl:HS_lang} by entries in boldface.

We have checked for $n\leq50$ that all possible $A_n(r)$, $r\geq3$, of
type $\sz{a,b,c}$, are covered by the above scheme. It also remains to
exhibit a tilting object for the exotic cases.

Most important is the $\EE$-series, given by the algebras $A_n(3)$,
figuring as the left column of Table~\ref{tbl:HS_lang}, which starts for $n=6,7,8$
with the types $\EE_5=\sz{2,3,4}=A_6(3)$, $\EE_6=A_7(3)$ and
$\EE_8=A_8(3)$. Here, every second member is represented by a category
$\sz{2,3,r}$. (Note that $\sz{2,3,3}=\DD_4$).

\begin{table}[p] 
\footnotesize
$$
\def\l{{\langle}}
\def\r{{\rangle}}
\def\gr{\cellcolor[gray]{0.8}}
\scriptsize
 \xymatrix@C=2.5pc@R=2pc{
  30&\sz{2,3,16}&\sz{2,4,11}&&\sz{2,6,7}&\sz{2,6,7}&&&&\sz{2,4,11}&\mathbf{\sz{2,4,11}}&&&\mathbf{\sz{2,3,16}}&\sz{2,3,16}&&&&&&&&&&&\\ &&&&&&&&&&&&&&&&&&&&&&&&&\\
  28&\sz{2,3,15}\ar@{-}[-2,0]&&\sz{2,5,8}\ar@{-}[-2,0]&&&\sz{2,5,8}&&&&&&&\sz{2,3,15}\ar@{-}[-2,1]&&&&&&&&&&&&\\ &&\sz{2,4,10}\ar@{-}[-3,0]&&&&&&\sz{2,4,10}\ar@{-}[-3,1]&&&&&&&&&&&&&&&&&\\
  26&\sz{2,3,14}\ar@{-}[-2,0]&&&&&&&&&&&\sz{2,3,14}\ar@{-}[-2,1]&\mathbf{\sz{2,3,14}}&&&&&&&&&&&&\\ &&&&\sz{2,6,6}\ar@{-}[-5,0]\ar@{-}[-5,1]&&&&&&&&&&&&&&&&&&&&&\\ 24&\sz{2,3,13}\ar@{-}[-2,0]&\sz{2,4,9}\ar@{-}[-3,0]&\sz{2,5,7}\ar@{-}[-4,0]&&\sz{2,5,7}\ar@{-}[-4,1]&\mathbf{\sz{2,4,9}}&\sz{2,4,9}\ar@{-}[-3,1]&&&\mathbf{\sz{2,3,13}}&\sz{2,3,13}\ar@{-}[-2,1]&&&&&&&&&&&&&&\\ &&&&&&&&&&&&&&&&&&&&&&&&&\\
 22&\sz{2,3,12}\ar@{-}[-2,0]&&&&&&&&&\sz{2,3,12}\ar@{-}[-2,1]&&&&&&&&&&&&&&&\\ &&\sz{2,4,8}\ar@{-}[-3,0]&&&&\sz{2,4,8}\ar@{-}[-3,1]&&&&&&&&&&&&&&&&&&&\\
 20&\sz{2,3,11}\ar@{-}[-2,0]&&\sz{2,5,6}\ar@{-}[-4,0]&\sz{2,5,6}\ar@{-}[-4,1]&&&&&\sz{2,3,11}\ar@{-}[-2,1]&\mathbf{\sz{2,3,11}}&&&&&&&&&&&&&&&\\ &&&&&&&&&&&&&&&&&&&&&&&&&\\ 18&\sz{2,3,10}\ar@{-}[-2,0]&\sz{2,4,7}\ar@{-}[-3,0]&&&\sz{2,4,7}\ar@{-}[-3,1]&\mathbf{\sz{2,4,7}}&\mathbf{\sz{2,3,10}}&\sz{2,3,10}\ar@{-}[-2,1]&&&&&&&&&&&&&&&&&\\ &&&&&&&&&&&&&&&&&&&&&&&&&\\
 16&\sz{2,3,9}\ar@{-}[-2,0]&&\sz{2,5,5}\ar@{-}[-4,0]\ar@{-}[-4,1]&&&&\sz{2,3,9}\ar@{-}[-2,1]&&&&&&&&&&&&&&&&&&\\
 &&\sz{2,4,6}\ar@{-}[-3,0]&&\sz{2,4,6}\ar@{-}[-3,1]&&&&&&&&&&&&&&&&&&&&&\\
 14&\sz{2,3,8}\ar@{-}[-2,0]&&&&&\sz{2,3,8}\ar@{-}[-2,1]&\mathbf{\sz{2,3,8}}&&&&&&&&&&&&&&&&&&\\
  &&&&&&&&&&&&&&&&&&&&&&&&&\\
 12&\sz{2,3,7}\ar@{-}[-2,0]&\sz{2,4,5}\ar@{-}[-3,0]&\sz{2,4,5}\ar@{-}[-3,1]&\mathbf{\sz{2,3,7}}&\sz{2,3,7}\ar@{-}[-2,1]&&&&&&&&&&&&&&&&&&&&\\ &&&&&&&&&&&&&&&&&&&&&&&&&\\
  10&\sz{2,3,6}\ar@{-}[-2,0]&&&\sz{2,3,6}\ar@{-}[-2,1]&&&&&&&&&&&&&&&&&&&&&\\ &&\sz{2,4,4}\ar@{-}[-3,0]\ar@{-}[-3,1]&&&&&&&&&&&&&&&&&&&&&&&\\
 8&\sz{2,3,5}\ar@{-}[-2,0]&&\sz{2,3,5}\ar@{-}[-2,1]&\mathbf{\sz{2,3,5}}&&&&&&&&&&&&&&&&&&&&&\\
 &&&&&&&&&&&&&&&&&&&&&&&&&\\
 6 &\sz{2,3,4}\ar@{-}[-2,0]&\sz{2,3,4}\ar@{-}[-2,1]&&&&&&&&&&&&&&&&&&&&&&&\\
 &&&&&&&&&&&&&&&&&&&&&&&&&\\
 4 &\sz{2,3,3}\ar@{-}[-2,0]\ar@{-}[-2,1]&&&&&&&&&&&&&&&&&&&&&&&&\\ \hline
 n/r &3&4&5&6&7&8&9&10&11&12&13&14&15&16&\\
 }
$$
\caption{Algebras $A_n(r)$, $n\leq30$, yielding categories $\svect\XX$}\label{tbl:HS_lang}
\end{table}

\section{Calabi-Yau dimension and Euler characteristic} \label{sect:CY}
In this section we show that for a weight triple the category $\svect\XX$ is always fractional Calabi-Yau. This uses that for a weight triple $\XX$ is given by an $\LL$-graded hypersurface singularity, where the two-fold suspension $[2]$ on $\svect\XX$ is given by grading shift with the canonical element, see Corollary~\ref{cor:double_extension}.

Let $\Tt$ be a triangulated category with Serre duality. Let $\SS$ denote the \define{Serre functor} of $\Tt$. Assume the existence of a smallest integer $n\geq1$ yielding an isomorphism $\SS^n\iso[m]$ of functors for some integer $m$. (Here, $[m]$ denotes the $m$-fold suspension of $\Tt$.) Then $\Tt$ is called \define{Calabi-Yau of fractional CY-dimension} $\frac{m}{n}$. Note that the ``fraction'' $\frac{m}{n}$ is kept in uncanceled format, that is, it is treated as the ordered pair $(m,n)$, if not stated otherwise.

\subsection{The category $\coh\XX$} ~\label{ssect:coh:CY}
For the bounded derived category $\Der{\coh\XX}$ of coherent sheaves on $\XX$ the Serre functor $\SS$ is given by the functorial isomorphisms $\SS(X)=X(\vom)[1]=X[1](\vom)$, or in different notation, $\SS=\tau[1]=[1]\tau$.  It follows that the bounded derived category $\Der{\coh\XX}$ is fractional Calabi-Yau if and only if there exist integers $n\geq 1$ and $m$ such that $[n]\tau^n=[m]$. This happens if and only if $n=m$ and $\tau^n=1$, that is, $n=m$ and $n\vom=0$ and thus is possible only for the tubular weight types $(2,2,2,2)$, $(3,3,3)$, $(2,4,4)$ and $(2,3,6)$ where the order $n$ of the Auslander-Reiten translate $\tau$ equals $2$, $3$, $4$ or $6$ respectively. Thus for Euler characteristic zero the CY-dimension is just
$$
\frac{2}{2},\; \frac{3}{3},\; \frac{4}{4} \textrm{ or } \frac{6}{6},
$$
according as $\XX$ has weight type $(2,2,2,2)$, $(3,3,3)$, $(2,4,4)$ or $(2,3,6)$, respectively. For all the other weight types the Auslander-Reiten translate of $\Tt=\Der{\coh\XX}$ is not periodic, hence $\Tt$ is not fractional Calabi-Yau.

\subsection{The category $\svect\XX$}
It is remarkable that, in the case of a weight triple $(p_1,p_2,p_3)$, the stable category $\svect\XX$ of vector bundles is always fractional Calabi-Yau. Moreover, up to cancelation, the CY-dimension only depends on the Euler characteristic of $\XX$.

It is easy to determine the CY-dimension of $\svect\XX$ up to cancelation.

\begin{proposition} \label{prop:CY:vom-vc} Assume $\XX$ is a weighted projective line. Then the smallest integer $n\geq0$ such that $n\vom$ belongs to $\ZZ\vc$ is $\bp$, the least common multiple of the weights.
\begin{equation} \label{eq:vc:vom}
\bp\,\vom=\de(\vom)\vc.
\end{equation}

Assuming that we deal with a weight triple, the category $\svect\XX$ is fractional Calabi-Yau. Moreover,  up to cancelation, we have
\begin{equation} \label{eq:CYdim:easy}
\CYdim{\svect\XX}=1-2\eulerchar\XX.  
\end{equation}
\end{proposition}
\begin{proof}
Let $t$ denote the number of weights. Then $\vom=(t-2)\vc-\sum_{i=1}^t\vx_i$ and $\de(\vom)=\bp\left((t-2)-\sum_{i=1}^t1/p_i)\right)$. Assume $n\geq1$. Expressing $n\vom$ in normal form, we see that $n\vom$ belongs to $\ZZ\vc$ if and only if each $p_i$ is a divisor of $n$. The first claim now follows from the fact that obviously $\bp\vom=\de(\vom)\vc$, see~\cite{Geigle:Lenzing:1987}.

For the second claim, we use that formula \eqref{eq:vc:vom} translates to $\tau^{\bp}=[2\de(\vom)]$, hence to $\SS^{\bp}=[\bp+2\de(\vom)]$ where $\SS=\tau[1]$ is the Serre functor of $\svect\XX$. This yields the quotient $(\bp+2\de(\vom))/\bp=1-2\eulerchar\XX$ agrees with the CY-dimension of $\svect\XX$, up to cancelation.
\end{proof}

It is more tricky to determine the exact value of $\CYdim{\vect\XX}$. Here, we distinguish the three cases:
(i) weight type $(2,2,n)$, (ii) weight type $(2,a,b)$ with $a,b\geq3$, (iii) weight type $(a,b,c)$ with $a,b,c\geq3$.

\subsubsection{Weight type $(2,2,n)$, $n\geq2$} In this case the
Picard group $\Pic\XX$ does not act faithfully on the category
$\svect\XX$, since line bundle twist $\si$ by $\vx_2-\vx_3$ fixes the Auslander bundles and then all objects, see Figure~\ref{fig:226}. Using \cite{Miyachi:Yekutieli:2001} it follows that $\si$ is (isomorphic to) the identity functor on $\svect\XX$. It follows that $\bar\LL=\LL/\ZZ(\vx_2-\vx_3)$ acts faithfully on $\svect\XX$. Denoting classes in $\bar\LL$ by brackets, it follows that $n$ is the smallest integer with the property $n[\vom]\in\ZZ[\vx_1]$. This shows that $\CYdim{\svect\XX}$ equals $\frac{n-1}{n}$. This result is well known since $\svect\XX\iso\Der{\mmod{k\vAA_n}}$.

For the remaining cases (ii) and (iii) we need the next result which may alternatively be deduced from the position of the line bundles in the components of the Auslander-Reiten quiver of $\svect\XX$.

\begin{proposition} 
For weight triples $(p_1,p_2,p_3)$ not of type $(2,2,n)$ the $\LL$-action on isomorphism classes of Auslander bundles is faithful.
\end{proposition}
\begin{proof}
Let $E$ be an Auslander bundle, $0\neq\vx\in\LL$ and assume $E\iso E(\vx)$. By Corollary~\ref{cor:abundles:stable:morphisms} the elements $\vx$ and $-\vx$ belong to $\set{\bx_1,\bx_2,\bx_3}$. Up to renumbering, we have to deal with one of the two cases: (1) $\vx=-\vx=\bx_3$, (2) $\vx=\bx_1=-\bx_2$. In the first case we obtain a normal form representation $0=2\bx_1=2\vc-2\vx_2-2\vx_3=(p_1-2)\vx_1+(p_2-2)\vx_2$ of zero, implying $p_1=2$ and $p_2=2$, thus type $(2,2,n)$.  In the second case we get a normal form representation $0=\bx_1+\bx_2=(p_1-1)\vx_1+(p_2-1)\vx_2+(p_3-2)\vx_3-\vc$ of zero which yields a contradiction since all $p_i$ are $\geq2$.
\end{proof}

\subsubsection{Weight type $(2,a,b)$}
We recall that by Proposition~\ref{prop:2ab:suspension} for a weight triple $(2,a,b)$ the suspension functor $[1]$ is given by the shift with $\vx_1$. For the rest of the section we assume moreover that $\XX$ is not of type $(2,2,n)$.

\begin{proposition} \label{prop:CY:2ab}
Assuming a weight triple $(2,a,b)$ with $a,b\geq3$, the Calabi-Yau dimension $d_\XX$ of $\svect\XX$ is given as
$$
d_\XX=\frac{\lcm(a,b)\left(2-\frac{2}{a}-\frac{2}{b}\right)}{\lcm(a,b)}
$$
\end{proposition}
\begin{proof}
We first determine the minimal $n\geq1$ such that $n\vom$ has the form $m\vx_1$ for some integer $m$. Since $n\vom=n\vc-n\vx_1-n\vx_2-n\vx_3$ this happens if and only if $n\vx_1+n\vx_2$ belongs to $\ZZ\vx_1$. Passing to normal forms this happens if and only if $a$ and $b$ divide $n$, that is, if and only if $\lcm(a,b)$ divides $n$. Thus the minimal $n$ such that $n\vom\in \ZZ\vx_1$ is $n=\lcm(a,b)$. In this case $n\vom=m\vx_1$ with $m=n-2n/a-2n/b$. Expressed differently, $n=\lcm(a,b)$ is the smallest positive integer such that $\SS^n=[n+m]$, yielding the CY-dimension $\frac{m+n}{n}$.
\end{proof}

\subsubsection{Weight type $(a,b,c)$ with $a,b,c\geq3$} Here, we need the following result.
\begin{proposition} \label{prop:suspension:Ltwist}
Assume a weight triple. Then the suspension functor $[1]$ of $\svect\XX$ is isomorphic to a grading shift if and only if $\XX$ has weight type $(2,a,b)$, up to permutation.
\end{proposition}

\begin{proof}
By Proposition~\ref{prop:2ab:suspension} the condition is clearly sufficient. We show that it is also necessary, and assume we don't have weight type $(2,2,n)$. So assume $\vy$ belongs to $\LL$ such that for each vector bundle $X$ we have $X(\vy)\iso X[1]$ in $\svect\XX$. (We do not need to assume that these isomorphisms are functorial in $X$.) Then $X(\vc)=X[2]\iso X(2\vy)$ holds for each $X$, and we obtain $2\vy=\vc$ since $\LL$ acts faithfully on $\svect\XX$. Specializing to the Auslander bundle $E$ we obtain from Corollary~\ref{cor:suspension:extb} that $E(\vy)=E[1]=\extb{}{\vdom}(-\vom)$ or, equivalently, $E(\vy+\vom)=\extb{}{\vdom}$. From $\sHom{}{E}{\extb{}{\vdom}}=k$ we obtain $\sHom{}{E}{E(\vy+\vom)}=k$ implying by Proposition~\ref{cor:abundles:stable:morphisms} that $\vy$ belongs to $\set{\vx_1,\vx_2,\vx_3}$. In view of $2\vy=\vc$ this is possible only for weight type $(2,a,b)$.
\end{proof}

\begin{proposition} \label{prop:CY} Assume $\XX$ is given by a weight
  triple $(p_1,p_2,p_3)$ with $p_i\geq3$ for all $i$. We put
  $\bp=\lcm(p_1,p_2,p_3)$. Then the category $\svect\XX$ is Calabi-Yau
  of fractional Calabi-Yau dimension $d_\XX$ given as follows:
  \begin{equation*}
    d_\XX=\frac{\bp-2\de(\vom)}{\bp}=\frac{\bp(3-\frac{2}{p_1}-\frac{2}{p_2}-\frac{2}{p_3})}{\bp},
  \end{equation*}
with $\bp=\lcm(p_1,p_2,p_3)$.
\end{proposition}
\begin{proof}
The proof will be based on the fact that the two-fold suspension $[2]$ for $\svect\XX$ equals the line bundle twist $\si(\vc)$ for the canonical element $\vc$, see Corollary~\ref{cor:double_extension}.  Assume that $n\geq1$ is minimal such that $\tau^n=[r]$ for some integer $r$, implying $d_\XX=\frac{n+r}{n}$. By our assumption on the weights, $r$ cannot be odd, since $r=2m+1$ implies that $[1]$ equals the shift by $n\vom-m\vc$ yielding in view of Proposition~\ref{prop:suspension:Ltwist} that we deal with type $(2,a,b)$, a contradiction. Thus $r=2m$ for some integer $m$ such that $d_\XX=\frac{n+2m}{n}$ in this case. The condition $\tau^n=[2m]$ translates to $n\vom=m\vc$ since $\LL$ acts faithfully on the set of isomorphism classes of $\svect\XX$. Now, the minimal $n$ with this property is $n=\bp$, which is easily seen by writing $n\vom$ in normal form. Moreover $\bp\vom=\de(\vom)\vc$ proving that $d_\XX=\frac{\bp+2\de(\vom)}{\bp}$, as claimed.
\end{proof}
\subsection{Periodicity of Coxeter transformation} \label{ssect:periodicity:Coxeter}
Recall that the \define{Coxeter transformation} of a triangulated
category $\Tt$ with Serre duality is the automorphism of the
Grothendieck group $\Knull\Tt$ induced by the Auslander-Reiten
translation $\tau=\SS[-1]$, where $\SS$ denotes the Serre functor for
$\Tt$. If additionally $\Knull\Tt$ is finitely generated free over
$\ZZ$, then the characteristic polynomial of the Coxeter
transformation is called the \define{Coxeter polynomial} of $\Tt$.  From
Proposition~\ref{prop:CY} we then deduce the following result.

\begin{proposition} \label{prop:coxeter:periodic}
Assume $\XX$ has weight type $(p_1,p_2,p_3)$ with $p_i\geq2$. Then the Coxeter transformation $\phi$ of $\svect\XX$ is always periodic. Its order, the Coxeter number $h$, is given as follows:
\begin{itemize}
\item[(i)] For weight type $(2,2,n)$, $n\geq2$, we have $h=n$.
\item[(ii)] For the remaining weight types we have $h=\bp:=\lcm(p_1,p_2,p_3)$.
\end{itemize}
\end{proposition}
\begin{proof}
Assertion (i) is well known, since for weight type $(2,2,n)$ the category $\svect\XX$ is triangle-equivalent to the derived category of a Dynkin quiver of type $\AA_{n-1}$. (With slight modifications, concerning the $\LL$-action on $\svect\XX$ the proof of (ii) would also work in this case.)

Concerning (ii), we recall from formula~\eqref{eq:vc:vom} that $\bp\,\vom=\de(\vom)\vc$, showing that $\tau^{\bp}=[2\de(\vom)]$ and implying $\phi^{\bp}=1$. It then suffices to show that the minimal integer $n\geq1$ with $\sclass{E(n\vom)}=\sclass{E}$ is a multiple of $\bp=\lcm(a,b,c)$. Since $E$ is exceptional in $\svect\XX$, we obtain the identity
$$
1=\euler{\sclass{E}}{\sclass{E}}=\euler{\sclass{E}}{\sclass{E(n\vom)}}
$$
implying---by the definition of the Euler form---the existence of an integer $m$ such that $\sHom{}{E}{E(n\vom+m\vc)}\neq0$. (This uses Corollary~\ref{cor:double_extension}.) By Corollary~\ref{cor:abundles:stable:morphisms} we then obtain that $n\vom+m\vc$ belongs to the set $\set{0,\bx_1,\bx_2,\bx_3}$. If $n\vom+m\vc=0$, then $n$ is a multiple of $\bp$ by Proposition~\ref{prop:CY:vom-vc}, finishing the proof in this case.

Assume now for contradiction that $n\vom+m\vc=\bx_i$ holds with $i=1,2$ or $3$. In this case $$\sclass{E(\bx_i)}=\sclass{E(n\vom+m\vc)}=\sclass{E(n\vom)}=\sclass{E}.$$ By Corollary~\ref{cor:four:abundles} the objects $E$ and $E(\bx_i)$ are nonisomorphic indecomposable summands from the same tilting object for $\svect\XX$, the tilting cuboid. Hence their classes in $\Knull{\svect\XX}$ cannot be the same, contradiction.
\end{proof}
For the next result we use shortcut notation $(m,n)$ (resp.\ $[m,n]$) for the greatest common divisor (resp.\ least common multiple) for integers $m,n$.

\begin{proposition}[{Hille-M\"{u}ller~\cite{Hille:Mueller:2012}}] \label{prop:coxeter_polynomial}
  Let $\XX$ have weight type $(a,b,c)$.  With $u_n=x^n-1$, the Coxeter
  polynomial of $\svect\XX$ is given by the expression
\begin{equation} \label{eq:coxpol}
\coxpol{\sz{a,b,c}}=
\frac{u_{[a,b,c]}^{abc/[a,b,c]}\,u_{a}\,u_{b}\,u_{c}}
{u_1\,u_{[a,b]}^{(a,b)}\,u_{[b,c]}^{(b,c)}\,u_{[c,a]}^{(c,a)}}.
\end{equation}
If $a,b,c$ are pairwise coprime, the formula simplifies to
\begin{equation}\label{eq:coxpol:coprime}
\coxpol{\sz{a,b,c}}=\frac{u_{abc}\,u_a\,u_b\,u_c}{u_{ab}\,u_{bc}\,u_{ca}}.
\end{equation}
Moreover, the weight triple $(a,b,c)$ is determined by $\coxpol{\sz{a,b,c}}$, up to permutation.
\end{proposition}
Thus we do not loose any information when passing from the category $\vect\XX$, or $\coh\XX$, to the stable category $\svect\XX$.

Representing polynomials $u_n$ as a product of cyclotomic polynomials,
this yields explicit factorizations of $\coxpol{\sz{a,b,c}}$ into
cyclotomics $\Phi_d$. For instance, we obtain
$\coxpol{\sz{2,3,7}}=\Phi_{42}$ and
$\coxpol{\sz{2,3,9}}=\Phi_2^2\Phi_{6}\Phi_{18}^2$.  While
formula~\eqref{eq:coxpol:coprime} allows immediate recognition of the
weight triple $(a,b,c)$ from the Coxeter polynomial, this is more
difficult for arbitrary weight triples, see~\cite{Hille:Mueller:2012}.

\section{ADE-chains} \label{sect:ADE}
The topics discussed in this section are closely related to the discussion of the Calabi-Yau property for $\svect\XX$ with a slight shift in emphasis. Instead of investigating a single triangulated category that is fractional Calabi-Yau, we look at whole sequences of such categories. There is much experimental evidence that sequences $(A_n)$ of finite dimensional algebras obeying the `same building law', enjoy a close relationship between representation-theoretic properties of their bounded derived categories $D^b(A_n)$ and the spectral properties of $A_n$ (Coxeter transformation, Coxeter polynomial, etc.). An attempt to formalize `nice' building laws is the concept of sequences $(A_n)$ of \emph{accessible algebras}~\cite{Lenzing:Pena:2008} which are obtained from the base field by forming successive 1-point-extension by exceptional modules. An important example is formed by the sequences of Nakayama algebras $A_n(r)$, $n\geq2$, for a fixed nilpotency degree $r\geq3$, see Section~\ref{ssect:Happel:Seidel}. Notice that the sequence of triangulated categories $\Tt_n=\Der{\mod{A_n(3)}}$, $n=2,3,\cdots$, interpolates the sequence $\svect\XX(2,3,n)$, $n=2,3,\cdots$, since $\svect\XX(2,3,n)$ is equivalent to $\Tt_{2(n-1)}$.

Table~\ref{tbl:ADE-chain} summarizes previous results and displays for $\svect\XX$ with $\XX$ of type$(2,3,n)$ the fractional Calabi-Yau dimension, the Euler characteristic $\eulerchar\XX$, the Coxeter number $h$, the representation type, and the derived type of $\svect\XX$ for small values of $p$.
\small
\begin{table}[H]
\begin{displaymath}
\def\l{\langle} \def\r{\rangle}
\renewcommand{\arraystretch}{1.4}
\begin{tabular}{|c||cccc|c|cccc|} \hline
  $n$ & 2 & 3 & 4 & 5 & 6 & 7 & 8 & 9&$n$\\
\hline
CY-dim&$\frac{1}{3}$
&$\frac{2}{3}$&$\frac{10}{12}$&$\frac{14}{15}$&$\frac{6}{6}$&$\frac{22}{21}$&$\frac{26}{24}$&
$\frac{10}{9}$&$\frac{\lcm(3,n)\cdot
    (1-2\cdot\eulerchar{\XX})}{\lcm(3,n)}$\\\hline
$\eulerchar\XX$&$\frac{1}{3}$
&$\frac{1}{6}$&$\frac{1}{12}$&$\frac{1}{30}$&$0$
&$-\frac{1}{42}$&$-\frac{1}{12}$&
$-\frac{1}{18}$&$\frac{1}{n}-\frac{1}{6}$\\\hline
$h$&3&6&24&30&6&42&24&18&$\lcm(6,n)$\\\hline
type      &$\AA_2$       &$\DD_4$      &$\EE_6$       &$\EE_8$
&$(2,3,6)$     &$\l 2,3,7\r$  &
$\l 2,3,8\r$  &$\l 2,3,9\r$&$\l 2,3,n\r$ \\ \hline
repr.\ type&\multicolumn{4}{c|}{repr.-finite}& tubular&
\multicolumn{4}{c|}{wild, new type}\\\hline
\end{tabular}
\end{displaymath}
\caption{The ADE-chain $\svect{\XX(2,3,n)}$}\label{tbl:ADE-chain}
\end{table}
\normalsize
\noindent Table~\ref{tbl:ADE-chain} expresses a typical property of what may be called an ADE-chain. For small values of $n$, the category $\svect\XX$ yields Dynkin type. Then, for $n=6$, the sequence passes the `borderline' of tubular type and then continues with wild type. While such situations occur frequently,  it is quite rare that one gets an infinite sequence of categories $\Tt_n$, where the size of $\Tt_n$, measured in terms of the Grothendieck group, is increasing with $n$, and where, most importantly, the members $\Tt_n$ are fractional Calabi-Yau.

This allows the following attempt in to 'define' the notion of an
ADE-chain $(\Tt_n)$ of triangulated categories $\Tt_n$, by requesting the three properties below:
\begin{enumerate}
\item[(1)] The sequence $(\Tt_n)$ forms an infinite
  chain $\Tt_1\subset \Tt_2
  \subset \Tt_2\subset \cdots$ of triangulated categories with Serre
  duality which are fractionally Calabi-Yau;
\item[(2)] Each category $\Tt_n$ has a tilting object $T_n$, hence a
  Grothendieck group which is finitely generated free of rank $n$;
\item[(3)] The endomorphism rings $A_n=\End{}{T_n}$ are a subsequence
  of a sequence of algebras $(B_n)$ that form an accessible chain of
  finite dimensional algebras in the sense of
  \cite{Lenzing:Pena:2008}.
\end{enumerate}
In our example the request (3) can be satisfied by means of the
algebras $B_n=A_n(3)$. We note, however, that the interpolating categories $\Der{\mod{A_n(3)}}$ may fail to be fractionally Calabi-Yau if $n$ is odd. For instance, for $n=9$ (or $n=11$) the interpolating category is of type $\Der{\coh\XX}$ with $\XX$ of type $(2,3,5)$ (resp.\ $(2,3,7)$); by Section~\ref{ssect:coh:CY} these two categories are not fractionally Calabi-Yau. More generally, one can check that the Coxeter transformation of $A_n(3)$ is not periodic if $n$ is congruent to $-1$ or $-3$ modulo $6$, preventing the CY-property for these cases.

Further ADE-chains, in the above sense, are given by the sequences $(\Tt_n)$ with $\Tt_n=\svect\XX(2,a,n)$, where $a\geq3$ is a fixed integer. Here, we can put $B_n=A_n(a)$ to achieve (3).

Of course, the chain $(\svect\XX(2,3,n))$ and the interpolating sequence $(A_n(3))$ are distinguished in various respects.  First, and most important, the sequence extrapolates the exceptional Dynkin cases $\EE_6$, $\EE_7$, and $\EE_8$ to an infinite sequence (thereby keeping the CY-property). It is thus justified to think of it as the $\EE_n$-chain. Second, the algebras for type $(2,3,n)$, the categories $\svect\XX$ of this chain have a tilting object consisting of Auslander algebras, a property that seems to be exceptional. Third, these categories are closely related to the study of invariant subspaces for nilpotent operators, see~\cite{Ringel:Schmidmeier:2008}, \cite{Kussin:Lenzing:Meltzer:2010apre}. Concerning further conjectural aspects of ADE-chains we refer to the account~\cite{Ringel:2008} of Ringel, forming the layout for a workshop on the ``ADE-chain'', November 2008, Bielefeld.

\section{Comments and Problems} \label{sect:comments:problems}
\subsection{Triangle singularities versus Kleinian and Fuchsian singularities} \label{ssect:triangle:fuchsian}
For each weighted projective line $\XX$ of non-zero Euler characteristic there is an associated $\ZZ$-graded singularity $R$ wich is called \define{Kleinian}, or simple, for $\eulerchar\XX>0$ and which is called \define{Fuchsian} for $\eulerchar\XX<0$. We have $R=\bigoplus_{n\geq0}S_{-n\vom}$ in the Kleinian and $R=\bigoplus_{n\geq0}S_{n\vom}$ in the Fuchsian case, where $S$ denotes the projective coordinate algebra of $\XX$. Different from the case of triangle singularities there is no restriction on the number of weights. The singularity category $\Dsing\ZZ{R}$ for a Kleinian or Fuchsian singularity also has an interpretation as a stable category of vector bundles, where---compared to the treatment of triangle singularities---a different Frobenius structure on the category of vector bundles is considered, with the indecomposable projective-injective objects forming the Auslander-Reiten orbit $\tau^\ZZ\Oo$ of the structure sheaf. The associated stable triangulated category then is the factor category $\vect\XX/[\tau^\ZZ\Oo]$ which is naturally equipped with a $\ZZ$-action generated by the grading shift $X\mapsto X(\vom)$ by the dualizing element. Thus on a superficial level (Frobenius structure, Serre duality, Auslander-Reiten quiver) the present investigation of triangle singularities is quite similar to the treatment of Kleinian and Fuchsian singularities in \cite{Lenzing:Pena:2011}. However, formally, the concepts of a triangle singularity and a Kleinian resp.\ Fuchsian singularity only agree for the two weight types $(2,3,5)$ (Kleinian) and $(2,3,7)$ (Fuchsian), that is, the only weight types where $\LL=\ZZ\vom$.  Moreover, the treatment of Euler characteristic zero, that is, of the tubular weights $(2,2,2,2)$ $(3,3,3)$, $(2,4,4)$ and $(2,3,6)$ only fits into a framework where the attached triangulated stable category is obtained as the factor category of $\vect\XX$ modulo the system of \emph{all} line bundles; compare also~\cite{Ueda:2006}.

Since $\LL=\ZZ\vom$ exactly for the weight types $(2,3,5)$ and $(2,3,7)$, the action of the Picard group on $\vect\XX$ does not descend to $\vect\XX/[\tau^\ZZ\Oo]$ for the other weight types, causing a lesser amount of symmetry for the Kleinian or Fuchsian situation. Whereas the stable category is always fractionally Calabi-Yau in the triangle case, there are only finitely many weight types where this happens in the Fuchsian case, see~\cite[Section~4]{Lenzing:Pena:2011}. (In the Kleinian case the stable category is always triangle-equivalent to the derived category of the path algebra of a Dynkin quiver, thus it is known to be fractional Calabi-Yau, compare~\cite{Miyachi:Yekutieli:2001}.)

We recall that a Fuchsian singularity is a hypersurface if and only if it belongs to Arnold's 14-member \define{strange duality list} of exceptional unimodal singularities given by their weight types (Dolgachev numbers) from Table~\ref{tbl:strange:duality}.
\small
\begin{table}[H]
\renewcommand{\arraystretch}{1.1}
\begin{center}
\begin{tabular}{|c|ccccccc|} \hline
weights&(2,3,7)&(2,3,8)&(2,3,9)&(2,4,5)&(2,4,6)&(2,4,7)&(2,5,5)\\
CY-dim &$\frac{22}{21}$   &$\frac{16}{15}$ &$\frac{26}{24}$&$\frac{16}{15}$&$\frac{12}{11}$&$\frac{20}{18}$&$\frac{22}{20}$\\ \hline
weights&(2,5,6)&(3,3,4)&(3,3,5)&(3,3,6)&(3,4,4)&(3,4,5)&(4,4,4)\\
CY-dim &$\frac{18}{16}$&$\frac{26}{24}$&$\frac{20}{18}$&$\frac{17}{15}$&$\frac{18}{16}$ &$\frac{15}{13}$ &$\frac{14}{12}$\\\hline
\end{tabular}
\end{center}
\caption{The Calabi-Yau dimension for Arnold's strange duality list}\label{tbl:strange:duality}
\end{table}
\normalsize
It is not difficult to show that each of these singularities yields a singularity category that is fractionally Calabi-Yau. From Table~\ref{tbl:strange:duality} then follows a remarkable fact.
\begin{proposition} Two exceptional unimodal Fuchsian singularities have the same CY-dimension if and only if they are related by Arnold's strange duality.
\end{proposition}
\begin{proof}
For the proof it suffices to compare Table~\ref{tbl:strange:duality} with the table from~\cite[page~185]{Arnold:Gusein-Zade:Varchenko:1985} describing strange duality.
\end{proof}
Comparing Propositions~\ref{prop:CY:2ab} and~\ref{prop:CY} with Table~\ref{tbl:strange:duality}, we observe that, for the same weight type, the values of the CY-dimensions for the triangle and the Fuchsian case usually are different. Another major difference between triangle singularities and Kleinian resp.\ Fuchsian singularities concerns the Gorenstein number. Whereas for triangle singularities this number can get arbitrarily big, it will always be $+1$ in the Kleinian and $-1$ in the Fuchsian case. Since for pairwise coprime weights, triangle singularities will automatically be $\ZZ$-graded, this statement also concerns $\ZZ$-graded singularities. Another difference concerns the shape of tilting objects in the Fuchsian case: As is shown in~\cite{Lenzing:Pena:2011}, there is always a tilting object whose endomorphism ring is an \define{extended canonical algebra}, that is, an algebra obtained from the canonical algebra associated with the weighted projective line by one-point-extension with an indecomposable projective.

Let us finally point out that the authors of \cite{KST-1,KST-2} take a different approach to analyze the singularity category of a Kleinian or Fuchsian singularity; their approach is based on graded maximal Cohen-Macaulay modules and matrix factorizations, not on stable categories of vector bundles, as proposed in this paper.

\subsection{Tubular versus Elliptic curves}
Assume $\XX$ is tubular, that is, of weight type $(2,2,2,2)$, $(3,3,3)$, $(2,4,4)$ or $(2,3,6)$. Then the classification of indecomposable bundles over the weighted projective line $\XX$ is very similar to Atiyah's classification of vector bundles on a smooth elliptic curve, compare~\cite{Atiyah:elliptic} and \cite{Lenzing:Meltzer:1993}. Indeed the relationship between tubular and elliptic curves is very close: Assume the base field is algebraically closed of characteristic different from $2$, $3$ and $5$. Let $C$ be a smooth elliptic curve of $j$-invariant $0$. Given a tubular weight type $(p_1,\ldots,p_t)$ it can be shown that $C$ admits an action of the cyclic group $G$ of order $\bar{p}=\lcm(p_1,\ldots,p_t)$ such that the category $\mathrm{coh}_G(C)$ of $G$-equivariant coherent sheaves on $C$ is equivalent to $\coh\XX$. For type $(2,2,2,2)$ this is \cite[Example~5.8]{Geigle:Lenzing:1987}, for the remaining weight triples this is unpublished work by Lenzing-Meltzer; see also Polishchuk's paper ~\cite[Theorem~1.8]{Polishchuk:2006}.  Thus $\svect\XX$ has the additional description as stable category $\underline{\mathrm{vect}}_G\textrm{-}C$ of $G$-equivariant vector bundles on $C$.

In view of this close relationship it is remarkable that the category of vector bundles on a smooth elliptic curve does not admit a Frobenius structure turning \emph{all} the line bundles into the indecomposable projective-injectives:
\begin{remark}
Assume $C$ is a smooth elliptic projective curve over an algebraically closed base field $k$. Then there does not exist an exact structure on $\vect\XX$ such that the indecomposable injectives are just the line bundles on $\XX$. Recall that $\coh{C}$ has Serre duality of the form $\dual{\Ext1{}{X}{Y}}=\Hom{}{Y}{X}$. In particular, we have $\Ext1{}{\Oo}{\Oo}=k$, resulting in an exact sequence $\mu:0\ra\Oo\up{\iota}E\up{\pi}\Oo\ra 0$ which is almost-split and whose central term, the Auslander bundle $E$, is indecomposable. Assume, for contradiction, that $E$ has an injective hull $E\up{h}\bigoplus_{i=1}^n L_i$ (for some exact structure on $\vect{C}$) with line bundles $L_1,\ldots,L_n$. For each point $x$ of $C$ we have a short exact sequence $0\ra \Oo\up{\al_x}\Oo(x)\up{\be_x} \simp{x}\ra 0$, where $\Oo(x)$ is a line bundle and $\simp{x}$ is the simple sheaf concentrated at $x$. Since $\mu$ is almost-split, each $\al_x:\Oo\ra \Oo(x)$ extends to a map $\bar{\al}_x:E\ra \Oo(x)$ which, by assumption, factors through the injective hull of $E$. We thus obtain factorizations
$$E\up{\bar\al_x}\Oo(x)=[E\up{(h_i)} \bigoplus_{i=1}^n L_i \up{(v_i)} \Oo(x)]$$
for each $x\in C$. In particular, for each $x\in C$ there exists an index $i=1,\ldots,n$ such that the composition $E\up{h_i}L_i\up{v_i}\Oo(x)$ is non-zero. Because $C$ is an elliptic curve each indecomposable vector bundle is semistable; moreover, since $h_i$ is not an isomorphism (as a constituent of the injective hull of $E$) the passage to slopes yields
$$
0=\mu(E) \leq \deg{L_i}\leq \deg(\Oo(x))=1.
$$
Hence each $L_i$ has slope $0$ or $1$. If $L_i$ has slope $0$, then $L_i\iso \Oo$. If $L_i$ has slope $1$ then $v_i:L_i\ra \Oo(x)$ is an isomorphism. Since we have infinitely many points this is impossible, contradicting our assumption on the existence of a Frobenius structure on $\vect\XX$ having the afore mentioned properties.
\end{remark}

\subsection{Automorphism group of $\svect\XX$}
Throughout we deal with a weight triple $(p_1,p_2,p_3)$ of weights $p_i\geq2$. As usual we assume $2\leq p_1\leq p_2\leq p_3$.
By the automorphism group $\Aut{\Cc}$ of a $k$-category $\Cc$ we understand the group of all isomorphism classes of $k$-linear self-equivalences of $\Cc$. If $\Cc$ is triangulated, then the members of $\Aut{\Cc}$ are additionally assumed to `be' exact functors.
 
Each automorphism $\al$ of the category $\coh\XX$ preserves the rank. In particular, $\al$ maps line bundles to line bundles and thus induces an automorphism of $\svect\XX$, viewed as a triangulated category. Recall that the \define{Picard group} $\Pic\XX$ of $\XX$ is defined as the group of line bundle twists $\si(\vx)$, $X\mapsto X(\vx)$, of $\coh\XX$. Moreover, let $\Aut\XX$ denote the automorphism group of $\XX$ which (for $t_\XX\geq3$) can be identified with the finite subgroup of permutations of $(p_1,p_2,p_3)$ preserving the weights. Clearly, $\Aut\XX$ acts on $\LL$ by permuting the generators $\vx_1,\vx_2,\vx_3$, and hence $\Aut\XX$ acts on $\Pic\XX$. It is shown in \cite[Theorem 3.4]{Lenzing:Meltzer:2000} that $\Aut{\coh\XX}$ is isomorphic to the semi-direct product $\Pic\XX \sdir \Aut\XX$ corresponding to the afore mentioned action. In the present discussion the weight triple $(2,2,n)$ yields a
somewhat degenerate situation.

\begin{proposition} \label{prop:autom_domestic}
For any weight triple distinct from $(2,2,n)$ with $n\geq2$, the natural homomorphism
$$
\pi:\Pic\XX \sdir \Aut\XX=\Aut{\coh\XX}\lra \Aut{\svect\XX}
$$
is injective.
Moreover, $\pi$ is surjective if $\eulerchar\XX>0$.
\end{proposition}

\begin{proof}
For the injectivity of $\pi$ we use that two Auslander bundles $E(\vx)$ and $E(\vy)$ are isomorphic in $\svect\XX$ if and only if $\vx=\vy$. This uses our assumption that the weight triple is not of type $(2,2,n)$.

Concerning surjectivity of $\pi$ for $\eulerchar\XX>0$ one deals separately with the various weight types and uses \cite[Theorem 4.1 and Theorem 1.8]{Miyachi:Yekutieli:2001}.
\end{proof}

For $\eulerchar\XX=0$ we have $\svect\XX\iso\Der{\coh\XX}$, hence the automorphism group is known by \cite{Lenzing:Meltzer:2000}. In particular $\pi$  is not an epimorphism in this case. For $\eulerchar\XX<0$ and weight type different from $(2,a,b)$, the suspension $[1]$ is not in the image of $\pi$ by Proposition~\ref{prop:suspension:Ltwist}, at least if $\Aut\XX$ is trivial. Still assuming $\eulerchar\XX<0$, it is not known wether the automorphism group of $\svect\XX$ will be generated by the image of $\pi$ and the suspension.

For $\eulerchar\XX\neq0$ we expect that all automorphisms of $\svect\XX$ preserve rank two for indecomposables. (For rank $\geq3$ it is known, see Remark~\ref{rem:suspension:not:preserves:rank}, that suspension is not preserving the rank of indecomposables.)

\subsection{Action of the Picard group on bundles of rank two}
\begin{remark}
The Picard group $\Pic\XX$ of $\XX$ acts on the set $\Vv_2$ of isomorphism classes of indecomposable rank two bundles (resp.\ its subclass of Auslander bundles) by line bundle twist (grading shift) $F\mapsto F(\vx)$. While this action is transitive in the case of Auslander bundles this is not longer the case for the action on $\Vv_2$. It follows from Theorem~\ref{thm:extension_bundle} that $\Vv_2/\Pic\XX$ has at most $(p_1-1)(p_2-1)(p_3-1)$ orbits. This bound is much to generous; by Proposition~\ref{prop:cuboid:symmetry} it can be substantially reduced by---roughly---a factor $2$ (resp.\ $4$ and $8$) in the linear, the rectangular and the cubical case. Lower bounds may be obtained by determining Coxeter polynomials for the perpendicular categories of extension bundles, formed in $\svect\XX$. For instance, for weight type $(2,3,20)$ the set $\Vv_2$ has ten $\LL$-orbits. In general, however, the exact number of orbits is not known.
\end{remark}

\subsection{Open problems}
At present, we don't have a factorization property for negative Euler characteristic that is comparable to Theorem~\ref{thm:factorization_domestic} and Theorem~\ref{thm:factorization_tubular}.

Another open question (for $\eulerchar\XX<0$) concerns how far an exceptional object $E$ in $\svect\XX$ is determined by its class $\sclass{E}$ in $\Knull{\svect\XX}$. For $\eulerchar\XX\geq0$ we know that $\svect\XX$ is triangle-equivalent to the derived category of a hereditary category. This immediately implies that two exceptional objects $E$ and $F$ in $\svect\XX$ have the same class in $\Knull{\svect\XX}$ if and only if $F\iso E(n\vc)$ for some integer $n$. It is not known whether a similar result holds in general.

The situation is similar for the next problem. Let $\Tt$ be a triangulated category with a complete exceptional sequence.  Recall that mutations yield a braid group action on such sequences, see~\cite{Gorodentsev:Rudakov:1987}. Moreover, this action is transitive if $\Tt$ is the derived category of modules over a finite dimensional hereditary algebra~\cite{Crawley-Boevey:1992,Ringel:1994a} or if $\Tt$ is the derived category of coherent sheaves on a weighted projective line~\cite{Meltzer:1995,Kussin:Meltzer:2002}. By these results the braid group action is transitive on $\svect\XX$ for non-negative Euler characteristic. We don't know whether the action is transitive for negative Euler characteristic.

Assume a weight triple $(a,b,c)$ and further that $F$ is exceptional in $\svect\XX$. By work of Bondal~\cite{Bondal:1989} the perpendicular category $\rperp{F}$ is again triangulated.  What are the possible types for $\rperp{F}$? For non-negative Euler characteristic it is known that $\rperp{F}$ must be piecewise hereditary. It is not difficult to see that, for negative Euler characteristic the category $\rperp{F}$ is not always piecewise hereditary,  but then the possible types for $\rperp{F}$ are not known. By contrast, assuming the minimal wild tripe $(2,3,7)$, we conjecture that each $\rperp{F}$ formed in $\svect\XX$ will be piecewise hereditary.

\appendix
\section{Comparison of scales} \label{appx:comparison:scales}
This section deals with the tubular case, and solves an intricate problem appearing there, the \define{comparison of scales}. Recall that for Euler characteristic zero, we have an equivalence
$$
\Phi:\Der{\coh\XX}\up{\iso}\svect\XX
$$
of triangulated categories. Such an equivalence is by no means unique, and the following assertions will depend on a fixed (clever) choice of $\Phi$. Observe that $\Phi$ sends the canonical tilting bundle $T_\can$ of $\coh\XX$ to a tilting object $T$ in $\svect\XX$ whose endomorphism ring is isomorphic to the canonical algebra $\La$ attached to $\XX$. Conversely, each choice of a tilting object $T$ in $\svect\XX$ with $\End{}{T}=\La$ determines such an equivalence $\Phi$. In the following we stick to the choice of $T$, hence of $\Phi$, made for the proof of Proposition~\ref{prop:interval_category}, where $T$ is simultaneously tilting in $\coh\XX$ and $\svect\XX$.

The following (important) question arises. What is the relationship between the slope-scales for $\Der{\coh\XX}$ and $\svect\XX$. Recall from \cite{Lenzing:Meltzer:2000} that the slope $(n,\slope{X})$ of an indecomposable $X[n]$ from $\Der{\coh\XX}$, with $X$ indecomposable in $\coh\XX$ and $n\in\ZZ$, lies in the rational helix $\ZZ\times\bar\QQ$ with $\bar\QQ=\QQ\union\set{\infty}$. On the other hand the slope of an indecomposable $E$ of $\svect\XX$ is a rational number $q$. Clearly the equivalence $\Phi$ induces a bijection $\ZZ\times\bar\QQ\ra \QQ$, and the question is: what is this map. It is of particular importance to determine the bijection $\alpha:\QQ\ra\QQ$ corresponding to the automorphism $(n,q)\mapsto (n+1,q)$ of the rational helix. We note that $\alpha$ describes the effect of the suspension of $\svect\XX$ on slopes. It turns out that $\alpha$ is given as a piecewise fractional linear map, see Theorem~\ref{thm:suspension}.

Throughout let $\XX$ be a tubular weighted projective line over an algebraically closed field $k$.

\subsection*{Projective covers of homogeneous quasi-simple objects}
The next proposition is a substantial generalization of Lemma~\ref{lemma:rank_function}. It allows, in particular, to extend Proposition~\ref{prop:interval_category} to slope categories $\Hh\sz{q}$ for an arbitrary rational number $q$. Here, $\Hh\sz{q}$ is the full subcategory of $\svect\XX$ given as the additive closure of all indecomposables $F$ with a slope in the range $\al^{-1}(q)<\slope{F}\leq q$.
\begin{proposition}
  Let $\mathbb{X}$ be a tubular weighted projective line. Assume $S_q$ is homogeneous quasi-simple of slope $0\leq q<1$. Let $q=d/r$ with $r>0$ and $(d,r)=1$. Denote by $\Ll_0$ a representative system of line bundles of degree zero. Then the minimal projective cover of $S_q$ has the form as follows:
  \begin{enumerate}
        \item for weight type $(2,3,6)$ $$0\ra S_q [-1]\ra\bigoplus_{L\in\Ll_0}\biggl(L^d \oplus L(-\vx_3)^r \oplus L(-2\vx_3)^{r-d}\biggr)\ra S_q \ra 0.$$ Here, $\Ll_0$ is given as a $\tau$-orbit of order $6$.
        \item for weight type $(2,4,4)$, $(3,3,3)$ and $(2,2,2,2)$ $$0\ra S_q [-1]\ra\bigoplus_{L\in\Ll_0}\biggl(L^d \oplus L(-\vx_2)^r \biggr)\ra S_q \ra 0.$$ Here, $\Ll_0$ is given by two, three or four $\tau$-orbits, respectively, each of order $4$, $3$ and $2$, respectively.
   \end{enumerate}
\end{proposition}
The first exponent $d$ vanishes for $q=0$, yielding a simplified expression in this case.
\begin{proof}
We assume that the weight type is $(2,3,6)$. The proof in the other cases is analogous. We first determine an irredundant system of morphisms from $\Ll_0$, $\Ll_{-1}$, $\Ll_{-2}$ to $S_q$, and later show that this system already yields the projective cover.

If $q=0$, there are no morphisms from $\Ll_0$ to $S_q$ (since $\Ll_0$ and $S_q$ lie in different tubes of slope $0$).

\newcommand{\LF}[2]{\langle #1,#2\rangle}
\newcommand{\DLF}[2]{\langle\langle #1,#2\rangle\rangle}
  If $q>0$ the dimension of $\Hom{}{L}{S_q}$ ($L\in\Ll_0$) is
  given by $$\LF{L}{S_q}=\frac{1}{6}\DLF{L}{S_q}=
  \begin{vmatrix}
    1 & r\\
    0 & d
  \end{vmatrix}=d,$$
showing that for each $L\in\Ll_0$ there exists $d$ linearly independent morphisms $$y_1^L,\dots,y_d^L\colon L\ra S_q.$$ Since $L$ is a line bundle and $S_q$ a vector bundle, each of these maps is mono.

The line bundles of slope $-1$ have the form $L(-\vx_3$ with $L\in\mathcal{K}_0$. Note that the only morphisms from $\Ll_1$ to $\Ll_0$ are -- up to scalars -- those of the form $L(-\vx_3)\stackrel{x_3}\ra L$ ($L\in\Ll_0$). Together with the previous step this shows that for each $L\in\Ll_0$ the space of morphisms from $L(-\vx_3)$ to $S_q$ is spanned by $$y_i^L \circ x_3 \ \ \ \ i=1,\dots,d.$$ This system is linearly independent since in view of $\bigl(\sum_{i=1}^d \alpha_i y_i^L \bigr)\circ x_3 =0$ we obtain from $\sum_{i=1}^d \alpha_i y_i^L$ an induced map $L/L(-\vx_3)\ra S_q$ in $\coh \XX$ which is zero, since $L/L(-\vx_3)$ is simple and $S_q$ is a bundle. This shows that $\sum_{i=1}^d \alpha_i y_i^L =0$, and then $\alpha_i =0$ for all $i=1,\dots,d$.

Since moreover, with the same argument as before, we get for $L\in\Ll_0$ $$\dim\Hom{}{L(-\vx_3)}{S_q}=
  \begin{vmatrix}
    1 & r\\
   -1 & d
  \end{vmatrix}=d+r,$$
we obtain a system of $r=(d+r)-d$ ``new'' maps $$z_1^L,\dots,z_r^L \colon L(-\vx_3)\ra S_q,$$ whose classes modulo $\langle y_1^L \circ x_3,\dots,y_d^L \circ x_3 \rangle$ form a $k$-basis.

Next we deal with the morphisms from $L(-2\vx_3)$ to $S_q$ factoring through $\Ll_{0}$ and $\Ll_{-1}$. (This consideration is not needed for the other weight types.) As in the previous case we have only one map -- up to scalars -- $$L(-2\vx_3)\stackrel{x_3}\ra L(-\vx_3)$$ from $L(-2\vx_3)$ to $\Ll_{-1}$. Moreover, we have -- up to scalars -- two morphisms $$L(-2\vx_3)\stackrel{x_3^2}\ra L,\ L(-2\vx_3)\stackrel{x_2}\ra L(\vx_2 -2\vx_3)$$ from $L(-2\vx_3)$ to $\Ll_0$.

We hence get a $2d$-dimensional space of morphisms from $L(-2\vx_3)$ to $S_q$, factoring through $\Ll_0$. This includes already a $d$-dimensional space of morphisms factoring through $\Ll_{-1}$. Note that the space of all morphisms factoring through $\Ll_{-1}$ has dimension $d+r$.

Summarizing, we get a $(2d+r)$-dimensional space of morphisms from $L(-2\vx_3)$ to $S_q$ factoring through $\Ll_0$ or $\Ll_{-1}$. Since for $L\in\Ll_0$ we have $$\dim\Hom{}{L(-2\vx_3}{S_q}= \begin{vmatrix} 1 & r\\ -2 & d \end{vmatrix}=d+2r$$ we get $(d+2r)-(2d+r)=r-d$ ``new'' maps.

Finally we have to show that line bundles of slope $\leq -3$ (respectively $\leq -2$ for the other weight types) cannot contribute to the projective cover of $S_q$: We know already by the projective covers of the Auslander bundles that $\alpha(n)=n$ for each integer $n$. Since $0\leq q<1$ we get $-3\leq\alpha^{-1}(q)<-2$. Hence we get $\mu S_q[-1]\in [-3,-2)$ enforcing that all line bundles in the projective cover of $S_q$ have slope $0$, $-1$ or $-2$. (Note: $-3$ cannot appear even if $q=0$.)
\end{proof}

\subsection*{Action of the suspension on slopes}

We denote the suspension by $E\mapsto E[1]$. Note that the objects of
$\vect\XX$ and $\svect\XX$ are the same. Let $E$ be an
indecomposable vector bundle, say of slope $q\in\QQ$. Since the
suspension preserves Auslander-Reiten components and by using the
Riemann-Roch theorem, it is easy to see, that the slope of (the stable
part of) $E[1]$ is independent of $E$. Thus we get a bijective,
monotonous map $\alpha\colon\QQ\ra\QQ$. Moreover, considering a
projective cover of a quasi-simple, homogeneous bundle $S_q$ of slope
$q$, it is easy to see that $\alpha(q)>q$ for all $q\in\QQ$. The next
assertion is a slight extension of Lemma~\ref{lem:slope}.
\begin{proposition}
There is a bijective, monotonous map $\alpha\colon\QQ\ra\QQ$ with $\alpha(q)>q$ for all $q\in\QQ$ and such that $\mu (E[1])=\alpha (\mu (E))$ for each indecomposable $E$ in $\svect\XX$.
\end{proposition}

\begin{theorem}[Slope action of suspension] \label{thm:suspension}
In the different tubular cases the map $\alpha^{-1}\colon\QQ\ra\QQ$ is given as follows:
  $$\begin{array}{|l|l|r|}
   \hline
   \text{weights} & \alpha^{-1} & \alpha^{-1}(0)\\
    \hline\hline
   (2,3,6) & q\mapsto q-3 & -3\\
    \hline
   (2,4,4) & q\mapsto q-2 & -2\\
    \hline
   (3,3,3) & n+q\mapsto
   \begin{cases}
     n+\frac{5q-3}{2-3q} & 0\leq q\leq\frac{1}{2}\\
     n+\frac{q}{1-3q} & \frac{1}{2}\leq q\leq 1
   \end{cases} & -3/2\\
   \hline
   (2,2,2,2) & n+q\mapsto
   \begin{cases}
     n+\frac{11q-4}{3-8q} & 0\leq q\leq\frac{1}{3}\\
     n+\frac{q}{1-4q} & \frac{1}{3}\leq q\leq 1
   \end{cases} & -4/3\\ \hline
  \end{array}$$
\end{theorem}

\begin{theorem}
  \begin{enumerate}
\item In the cases $(2,3,6)$ and $(2,4,4)$ the suspension functor $[1]\colon\svect\XX\ra\svect\XX$ is isomorphic to the functor $\sigma\colon E\mapsto E(\vx_1)$ induced by grading shift with $\vx_1$.
\item In case $(3,3,3)$ the two-fold suspension $[2]$ is isomorphic to the functor induced by the grading shift $\sigma\colon E\mapsto E(\vec{c})$ with the canonical element.
\item In case $(2,2,2,2)$ the two-fold suspension functor $[2]$ is \emph{not} induced by a grading shift.
  \end{enumerate}
\end{theorem}
\begin{proof}
(1) and (2). It is shown in Corollary~\ref{cor:double_extension} that in all triple weight cases the two-fold suspension $[2]$ is given by the functor induced by grading shift $\sigma(\vc)\colon E\mapsto E(\vec{c})$, which already shows (2). We further make use of Theorem~\ref{thm:tubular} by which we have $\svect{\XX}\simeq\Der{\coh\XX}$. We start with determining the possible square roots $\varphi$ of $[2]$ in $\Aut{\Der{\coh\XX}}$ and set $\psi=\varphi\circ [-1]$. We have $\psi\in\Aut{\Der{\coh\XX}}$ with $\psi^2=1$. We conclude that $\psi$ preserves slopes, and hence in particular $\psi\in\Aut{\coh\XX}$ is slope-preserving. Therefore $\psi(\Oo)=\Oo(\vx)$ for some $\vx\in t\LL$, the torsion subgroup of $\LL$. Set $\gamma=\sigma(-\vx)\circ\psi$. Then $\gamma(\Oo)=\Oo$, that is, $\gamma\in\Aut{\XX}$.

\underline{Case $(2,3,6)$}: Here, $\Aut{\XX}=1$. Hence $\gamma=1$ and $\psi=\sigma(\vx)$. From $\psi^2=1$ we derive $2\vx=0$. Since $t\LL(2,3,6=\langle\vom\rangle$ of order $6$, we conclude $\vx=0$ or $\vx=3\vom$. This means, $\psi=1$ or $\psi=\tau^3$, and then $\varphi=[1]$ or $\varphi=[1]\circ\tau^3$. In other words: $[2]=\sigma(\vc)$ in $\Aut{\svect{\XX}}$ has precisely the two square roots $\sigma(\vx_1)$ and $\sigma(\vx_1+\vom)$. Since for all Auslander bundles $E$, $F$ we have $E[1]\simeq E(\vx_1)$ and $\tau^3F\not\simeq F$, the only possibility is $[1]=\sigma(\vx_1)$.

\underline{Case $(2,4,4)$}: By computing in $\Aut{\Der{\coh\XX}}$ one gets $8$ possible square roots of the identity functor $1$ on $\Der{\coh\XX}$, and hence $8$ possible $\varphi\in\Aut{\svect{\XX}}$ with $\varphi^2=[1]$. These are of two different kinds:
  \begin{itemize}
\item[kind A] $\varphi=\sigma(\vx_1+\vx)$ with $2\vx=0$; there are precisely $4$ such $\vx$.
\item[kind B] $\varphi=\sigma(\vx_1+\vx)\circ\gamma$ with $\gamma\in\Aut{\XX}$ the automorphism commuting the two points of weight $4$, and $\vx\in\langle\vx_2-\vx_3\rangle$; these are $4$ additional possibilities.
  \end{itemize}
Using again that for Auslander bundles $E[1]=E(\vx_1)$ holds, we conclude that only $\varphi=\sigma(\vx_1)$ remains.

(3) In case $(2,2,2,2)$ we derive from the table above that $\alpha^{-1}(0)=-12/5$ is not an integer, and thus $[2]$ is not preserving the rank and hence cannot be induced by a line bundle twist.
\end{proof}

\begin{remark}\label{rem:suspension-3-3-3}
The preceding proof for the tubular cases $(2,3,6)$ and $(2,4,4)$ does not extend to the case $(3,3,3)$, since in case $(3,3,3)$ the suspension $[1]$ is \emph{not} given by Picard shift. The reason is that $[2]=\sigma(\vc)$, and $\vc$ has degree $3$, which is not divisible by $2$. Similar considerations as before show that even $[1]$ is not induced by an automorphism of $\coh\XX$ (and hence also not of $\vect{\XX}$).
\end{remark}

\subsection{The tubular factorization property}

The following factorization theorem exhibits a previously unknown feature of the categories of coherent sheaves $\coh\XX$ on a tubular weighted projective line. The result is an easy byproduct of the present investigation and confirms, from a different perspective,
the interest in an explicit expression for the action $\alpha$ of the suspension on slopes.

\begin{theorem}[Tubular factorization property] \label{thm:factorization_tubular}
Let $E$ and $E'$ be indecomposable in $\vect \XX$ with slopes $\mu E=q$ and $\mu E'=q'$. If $q'>\alpha (q)$ then every morphism $E\ra E'$ factors through a direct sum of line bundles.
\end{theorem}
\begin{proof}
We may assume that $E$ and $E'$ have rank $\geq2$. We form the distinguished exact sequence $\mu:0\ra E\up{j}\injh{E}\up{\pi}E[1]\ra 0$ given by the injective hull of $E$. Applying $\Hom{}{-}{E'}$ to $\mu$ we obtain an exact sequence
$$
0\ra (E[1],E')\up{\pi^*}(\injh{E},E')\up{j^*}(E,E')\ra\Ext1{}{E[1]}{E'}.
$$
Now $\dual{\Ext1{}{E[1]}{E'}}=\Hom{}{E'}{E[1]}$ is zero since $q'=\slope{E'}>\slope{E[1]}=\al(q)$. It follows that each morphism from $u:E\ra E'$ factors through the injective hull $\injh{E}$.
\end{proof}

\section{The two-fold suspension} \label{appx:2-fold suspension}

Throughout this section, which is modeled upon~\cite{Yoshino:1990}, let $S=k[x_1,x_2,x_3]/(f)$ be a triangle singularity of type $(a,b,c)$, that is, $f=x_1^a+x_2^b+x_3^c$, and $S$ is $\LL=\LL(a,b,c)$-graded. Note that $\deg(f)=\vc$. Moreover, denote by $P=k[x_1,x_2,x_3]$ the corresponding $\LL$-graded polynomial algebra. Then $S=P/(f)$ is a complete intersection, in particular Gorenstein (both in the graded sense). This yields
\begin{itemize}
\item $\CMgr{\LL}{S}$ inherits the exact structure from $\modgr{\LL}{S}$, which makes $\CMgr{\LL}{S}$ a Frobenius category; here, $\projgr{\LL}{S}$ forms the system of projective-injective objects.
\item $\sCMgr{\LL}{S}=\CMgr{\LL}{S}/[\projgr{\LL}{S}]$ is triangulated.
\end{itemize}

\begin{proposition}
Let $M\in\CMgr{\LL}{S}$. Then there exists a ``natural'' exact sequence of graded $P$-modules
  \begin{equation}
    \label{eq:nat-P-seq}
    0\ra G\stackrel{\psi}\ra F\ra M\ra 0
  \end{equation}
with finitely generated graded free $P$-modules $G$ and $F$ so that by reduction mod $f$ a ``natural'' exact sequence
  \begin{equation}
    \label{eq:red-nat-P-seq}
    0\ra
    M(-\vc)\ra\overline{G}\stackrel{\bar{\psi}}\ra\overline{F}\ra M\ra 0
   \end{equation}
 in $\modgr{\LL}{S}$ is obtained.
\end{proposition}
\begin{proof}
Tensoring sequence~\eqref{eq:nat-P-seq} with the graded $P$-module $M$  yields the exact sequence $$\Tor{1}{P}{M}{S}\ra\overline{G}\stackrel{\bar{\psi}} \ra\overline{F}\ra M\ra 0,$$ where we use the notation $\overline{X}=X/fX(-\vc)$. We compute $\Tor{1}{P}{M}{S}$ via the $P$-projective resolution
  \begin{equation}
    \label{eq:P-proj-res}
    0\ra P(-\vc)\stackrel{f}\ra P\ra S\ra 0.
  \end{equation}
Tensoring of~\eqref{eq:P-proj-res} with the graded $P$-module $M$ yields the exact sequence $$0\ra\Tor{1}{P}{M}{S}\ra M(-\vc)\stackrel{f\circ -}\lra M\ra M\otimes_P S\ra 0,$$ and since $f\circ -=0$, we get $\Tor{1}{P}{M}{S}\simeq M(-\vc)$, which holds functorially in $M$.
\end{proof}
Since $\sCMgr{\LL}{S}\simeq\svect{\XX}$ this yields:
\begin{corollary} \label{cor:double_extension}
Assume that the weighted projective line $\XX$ is given by a weight triple. Then the two-fold suspension $[2]$ of $\svect\XX$ is equivalent to the line bundle twist $X\mapsto X(\vc)$ by the canonical element.\qed
\end{corollary}

\begin{remark}
We recall from Remark~\ref{rem:suspension-3-3-3} that for weight type $(3,3,3)$ the suspension $[1]$ of $\svect\XX$ is \emph{not} given by a line bundle twist shift, and is not even induced by an automorphism of $\coh\XX$. For a more general statement compare Proposition~\ref{prop:suspension:Ltwist}.
\end{remark}

\section{An $\LL$-graded version of Orlov's theorem} \label{appx:Orlov}

In the beginning of this section we allow arbitrary weight type.
We apply the (graded) global section
functor $\Gamma\colon\coh\XX\lra\Modgr{\LL}{S}$, $
  X\mapsto\bigoplus_{\vx\in\LL}\Hom{}{\Oo(-\vx)}{X}$ to the almost-split sequence
\begin{equation}
  \label{eq:AR-seq-orlov}
  0\ra\Oo(\vom)\stackrel{\alpha}\ra E\stackrel{\beta}\ra\Oo\ra 0
\end{equation}
defining the Auslander bundle.

We recall that $\Gamma$ induces an
equivalence $$\Gamma\colon\vect{\XX}\stackrel{\simeq}\lra\CMgr{\LL}{S}.$$
Applying $\Gamma$ to~\eqref{eq:AR-seq-orlov} yields the exact sequence
\begin{equation}
  \label{eq:seq-Auslander-module}
  0\ra S(\vom)\stackrel{\alpha}\ra M\stackrel{\beta}\ra
  S\stackrel{\pi}\ra k\ra 0,
\end{equation}
where $\ker(\pi)=\mathfrak{m}$, the unique graded maximal ideal of $S$, and $M$ is the $\LL$-graded \define{Auslander module} which is (maximal) Cohen-Macaulay, and further indecomposable for $t_\XX\geq 3$. Note further, that~\eqref{eq:AR-seq-orlov} is obtained back from~\eqref{eq:seq-Auslander-module} by sheafification.

Next we recall from \cite{Buchweitz:1986}, \cite{Orlov:2009} and \cite[Thm.~5.1]{Geigle:Lenzing:1987} that there are natural equivalences $$\svect{\XX}\stackrel{\Gamma}\lra\sCMgr{\LL}{S} \stackrel{\phi}\lra\frac{\Der{\modgr{L}{S}}}{\Der{\projgr{L}{S}}} \stackrel{def}=\Sing{\LL}{S}$$ where $\Der{\projgr{L}{S}}$ stands for the full subcategory of perfect complexes, and $\phi(M)=M$ with the right hand side interpreted as the complex with $M$ concentrated in degree zero.

\begin{proposition}
By means of the equivalence $$\Theta=[1]\circ\phi\circ\Gamma\colon\svect{\XX}\lra\Sing{\LL}{S}$$ the Auslander bundle $E$ is mapped to $k$.

Further, $\Theta$ commutes with the respective $\LL$-actions on $\svect{\XX}$ and $\Sing{\LL}{S}$; in particular $\Theta(E(\vx))=k(\vx)$.
\end{proposition}
\begin{proof}
  From~\eqref{eq:seq-Auslander-module} we get short exact sequences
  \begin{equation}
    \label{eq:two-ex-seqs}
    0\ra\mathfrak{m}\ra S\ra k\ra 0,\quad 0\ra S(\vom)\ra
    M\ra\mathfrak{m}\ra 0.
  \end{equation}
Passing to $\Sing{\LL}{S}$, we obtain distinguished exact triangles
\begin{equation}
  \label{eq:two-dist-triangels}
  \mathfrak{m}\ra S\ra k\ra\quad\text{and}\quad S(\vom)\ra
    M\ra\mathfrak{m}\ra.
\end{equation}
By standard properties of the Verdier quotient this implies
$\mathfrak{m}\simeq k[-1]$ and $M\simeq\mathfrak{m}$ in
$\Sing{\LL}{S}$, implying the claim.
\end{proof}

Let $S=k[x_1,x_2,x_3]/(h)$ with $h=x_1^{p_1}+x_2^{p_2}+x_3^{p_3}$ the
$\LL$-graded projective coordinate algebra of $\XX$.

\begin{definition}
The \define{Gorenstein number} $a=a(\XX)$ of $S$ (or $\XX$) is defined
as the signed index $\epsilon=[\LL:\ZZ\vom]$, where $\epsilon$ equals
$+1$, $0$, or $-1$,  according as $\eulerchar\XX$ is $>0$, $0$
respectively $<0$. In particular, we put $a=0$ for $\eulerchar\XX=0$.
\end{definition}
The next fact is straightforward to verify.
\begin{lemma} \label{lemma:g_number}
The Gorenstein number $a$ equals $p_1p_2p_3(1/p_1+1/p_2+1/p_3-1)$,
hence agrees with $p_1p_2p_3\,\eulerchar\XX=-p_1p_2p_3/\lcm(p_1,p_2,p_3)\delta(\vom)$.

Moreover, for $a\neq0$, the number $|a|$ equals the number of Auslander-Reiten
orbits of line bundles.~\qed
\end{lemma}
If the weights $p_1,p_2,p_3$ are pairwise coprime, then the Gorenstein number is just the Gorenstein parameter in Orlov's setting\footnote{There are different names around: Gorenstein number, Gorenstein parameter, $a$-invariant and, maybe, others. Also different authors stick to different sign-conventions. Here, we adhere to Orlov's choice in \cite{Orlov:2009}.}. And then the following theorem is explicitly covered by Orlov's original version. An $\LL$-graded version of Orlov's theorem, adapted to the present context, yields the following trichotomy.

\begin{theorem} \label{thm:Orlov:L-graded} Let $a=a(\XX)$ be the Gorenstein number of $S=k[x_1,x_2,x_3]/(x_1^{p_1}+x_2^{p_2}+x_3^{p_3})$ (or $\XX$). Then the following assertions hold:
\begin{itemize}
\item[(i)] For $a>0$ the category $\svect\XX$ is a triangulated subcategory of $\Der{\coh\XX}$ obtained as the (right) perpendicular subcategory with respect to an exceptional sequence formed from the $\Oo(\vx)$ with $\vx$ in the slope range $-\de(\vom)<\de(\vx)\leq0$.

\item[(ii)] For $a=0$ the triangulated categories $\svect\XX$ and $\Der{\coh\XX}$ are equivalent.

\item[(iii)]For $a<0$, the category $\Der{\coh\XX}$ is a triangulated subcategory of $\svect\XX$ obtained as the (right) perpendicular subcategory with respect to an exceptional sequence formed from the Auslander bundles  $E(\vy)$ where $\vy$ satisfies $0\leq\de(\vy)<\de(\vom)$.
\end{itemize}
\end{theorem}

\begin{proof}
The proof of the $\ZZ$-graded version of Orlov's theorem~\cite{Orlov:2009} turns over to the $\LL$-graded case taking the following precautions: For $a>0$ the sequence $\Oo(i)$ with $-a<i\leq 0$ is replaced by the system of all $\Oo(\vx)$ with $-a<\de(\vx)\leq0$. For $a=0$ there are no changes, as was already observed by Ueda~\cite{Ueda:2006}. For $a<0$ Orlov deals with the 'modules' $k(n)$, $n=0,\ldots,a-1$ in the singularity category $\Sing\LL{S}=\Der{\modgr{\ZZ}{S}}/\Der{\projgr{\ZZ}{S}}$.  By means of the identifications $\Sing\LL{S}=\sCMgr\LL{S}=\svect\XX$ we may replace $k$ from $\Sing\LL{S}$ by the Auslander bundle $E$ in $\svect\XX$. (The occurrence of the additional suspension [1] does not harm.)
\end{proof}

\begin{corollary}
The Grothendieck group $\Knull{\svect\XX}$ is free abelian of rank
$(p_1-1)(p_2-1)(p_3-1)$.
\end{corollary}

\begin{proof}
The proof is obvious, combining Lemma~\ref{lemma:g_number} with the fact that $\Knull{\coh\XX}$ has rank $p_1+p_2+p_3-1$.
\end{proof}

\section*{Acknowledgements}
The third author was supported by the Polish Scientific Grant Narodowe Centrum Nauki DEC-2011/01/B/ST1/06469.

\bibliographystyle{abbrv}

\def\cprime{$'$}

\end{document}